\documentclass[12pt]{elsarticle}
\usepackage[margin=2.0cm]{geometry}
\usepackage{amssymb,amsmath,amscd}
\usepackage{graphicx}
\usepackage{color}
\usepackage{subfigure}
\usepackage{pdfsync}
\usepackage[normalem]{ulem}
\usepackage[mathcal]{euscript}
\graphicspath{{FIGURES/}{../FIGURES/}}

\newtheorem{theorem}{Theorem}[section]

\newtheorem{proposition}[theorem]{Proposition}

\newproof{proof}{Proof}

\newtheorem{definition}[theorem]{Definition}
\newtheorem{assumption}[theorem]{Assumption}

\newtheorem{remark}[theorem]{Remark}

\newcommand{\Omegaref}{\widehat{\Omega}}

\definecolor{dgreen}{rgb}{0,0.6,0.9}

\newcommand{\ds}{\displaystyle}

\def\N{\mathbb N}

\definecolor{ddgreen}{rgb}{0.5,0.2,1}
\definecolor{dred}{rgb}{0.92,0,0}
\definecolor{dgreen}{rgb}{0,0.6,0}

\def\Rd{\color{black}}
\def\Gd{\color{black}}
\def\B{\color{black}}
\def\Bd{\color{black}}

\newcommand{\ext}{{\rm ext}}

\newcommand{\M}{\mathcal{M}}
\newcommand{\pM}{\mathsf{M}}

\newcommand{\A}{\mathcal{A}}

\newcommand{\R}{\mathbb{R}}
\renewcommand{\N}{\mathbb{N}}

\newcommand{\AS}{\mathsf{AS}}

\def\div{\,{\rm div}\,}

\newcommand{\imag}{{\rm i}}

\newcommand{\be}{{\bf e}}
\newcommand{\bn}{{\bf n}}
\newcommand{\bu}{{\bf u}}
\newcommand{\bv}{{\bf v}}

\newcommand{\bC}{{\bf C}}
\newcommand{\bE}{{\bf E}}
\newcommand{\bG}{{\bf G}}

\newcommand{\bF}{{\bf F}}

\newcommand{\bL}{{\bf L}}

\newcommand{\bzeta}{{\boldsymbol{\zeta}}}

\newcommand{\curl}{\,{\bf curl}\,}

\def\div{\,{\rm div}\,}

\newcommand{\Horot}{\,{\bf H}_0({\rm rot};\Omega)\,}
\newcommand{\Hcurl}{\,{\bf H}({\mathbf{curl}};\Omega)\,}
\newcommand{\HcurloD}{\,{\bf H}_{0,\Gamma_D}({\mathbf{curl}};\Omega)\,}
\newcommand{\Hcurlr}{\,{\bf H}({\mathbf{curl}};{\Omegaref})\,}
\newcommand{\Hocurl}{\,{\bf H}_0(\mathbf{curl};\Omega)\,}

\newcommand{\Hocurlr}{\,{\bf H}_0(\mathbf{curl};{\Omegaref})\,}
\newcommand{\Hdiv}{\,{\bf H}({\rm div};\Omega)\,}

\newcommand{\Hdivr}{\,{\bf H}({\rm div};{\Omegaref})\,}
\newcommand{\Hodiv}{\,{\bf H}_0({\rm div};\Omega)\,}

\newcommand{\Hodivr}{\,{\bf H}_0({\rm div};{\Omegaref})\,}

\newcommand{\Hone}{\,{H}^1(\Omega)\,}
\newcommand{\Honeo}{\,{H}^1_0(\Omega)\,}
\newcommand{\Ltwo}{\,{L}^2(\Omega)\,}
\newcommand{\Ltwoo}{\,{L}^2(\Omega)\, }
\newcommand{\Honer}{\,{H}^1(\Omegaref)\,}
\newcommand{\Honeor}{\,{H}^1_0(\Omegaref)\,}
\newcommand{\Ltwor}{\,{L}^2(\Omegaref)\,}
\newcommand{\Ltwoor}{ \,{L}^2(\Omegaref)\,  }

\newcommand{\zero}{\,{\bf 0}}

\newcommand{\Honesp}{S_{p_1,p_2,p_3} (\Xi_1, \Xi_2, \Xi_3)}
\newcommand{\Hcurlsp}{S_{p_1-1,p_2,p_3} (\Xi'_1, \Xi_2, \Xi_3) \times S_{p_1,p_2-1,p_3} (\Xi_1, \Xi'_2, \Xi_3)\times S_{p_1,p_2,p_3-1} (\Xi_1, \Xi_2, \Xi'_3)}
\newcommand{\Hdivsp}{S_{p_1,p_2-1,p_3-1} (\Xi_1, \Xi'_2, \Xi'_3) \times S_{p_1-1,p_2,p_3-1} (\Xi'_1, \Xi_2, \Xi'_3)\times S_{p_1-1,p_2-1,p_3} (\Xi'_1, \Xi'_2, \Xi_3)}
\newcommand{\Ltwosp}{S_{p_1-1,p_2-1,p_3-1} (\Xi'_1, \Xi'_2, \Xi'_3) }
\newcommand{\Szero}{\hat{X}^0}
\newcommand{\Sone}{\hat{X}^1}
\newcommand{\Stwo}{\hat{X}^2}
\newcommand{\Sthree}{\hat{X}^3}

\newcommand{\Szeroh}{\hat{X}^0_h}
\newcommand{\Soneh}{\hat{X}^1_h}
\newcommand{\Stwoh}{\hat{X}^2_h}
\newcommand{\Sthreeh}{\hat{X}^3_h}

\newcommand{\Szeroo}{\hat{X}^0_0}
\newcommand{\Soneo}{\hat{X}^1_0}
\newcommand{\Stwoo}{\hat{X}^2_0}
\newcommand{\Sthreeo}{\hat{X}^3_0}

\newcommand{\Szerooh}{\hat{X}^0_{0,h}}
\newcommand{\Soneoh}{\hat{X}^1_{0,h}}
\newcommand{\Stwooh}{\hat{X}^2_{0,h}}
\newcommand{\Sthreeoh}{\hat{X}^3_{0,h}}

\newcommand{\Xzero}{X^0}
\newcommand{\Xone}{X^1}
\newcommand{\Xtwo}{X^2}
\newcommand{\Xthree}{X^3}

\newcommand{\Xzeroh}{X^0_h}
\newcommand{\Xoneh}{X^1_h}
\newcommand{\Xtwoh}{X^2_h}
\newcommand{\Xthreeh}{X^3_h}

\newcommand{\Zzeroh}{Z^0_h}
\newcommand{\Zoneh}{Z^1_h}
\newcommand{\Ztwoh}{Z^2_h}
\newcommand{\Zthreeh}{Z^3_h}

\newcommand{\Xzeroo}{X^0_0}
\newcommand{\Xoneo}{X^1_0}
\newcommand{\Xtwoo}{X^2_0}
\newcommand{\Xthreeo}{X^3_0}

\newcommand{\Xzerooh}{X^0_{0,h}}
\newcommand{\Xoneoh}{X^1_{0,h}}
\newcommand{\Xtwooh}{X^2_{0,h}}

\newcommand{\uzero}{\phi}
\newcommand{\uone}{\bu}
\newcommand{\utwo}{\bv}
\newcommand{\uthree}{\varphi}

\newcommand{\pb}{\iota}

\def\bF{{\bf F}}

\def\bC{{\bf C}}

\def\bv{{\bf v}}

\renewcommand{\sp}{{S}}

\newcommand{\grad}{\,{\bf grad}\,}
\newcommand{\rot}{\,{\rm rot}\,}
\newcommand{\brot}{\,{\bf rot}\,}

\newcommand{\Moneh}{{\mathcal M}^{1}_1}
\newcommand{\Monev}{{\mathcal M}^{1}_2}
\newcommand{\Mzero}{{\mathcal M}^{0}}
\newcommand{\Mtwo}{{\mathcal M}^{2}}

\let\hat\widehat

\begin{document}

\begin{frontmatter}
\title{Isogeometric Methods for Computational Electromagnetics: \\B-spline and T-spline discretizations}

\author[imati]{A.~Buffa}
\ead{annalisa.buffa@imati.cnr.it}
\author[unipv,imati]{G.~Sangalli}
\ead{Giancarlo.sangalli@unipv.it}
\author[imati]{R.~V\'azquez}
\ead{vazquez@imati.cnr.it}

\address[imati]{Istituto di Matematica Applicata e Tecnologie Informatiche 'E. Magenes' del CNR\\ via Ferrata 1, 27100, Pavia, Italy}
\address[unipv]{Dipartimento di Matematica, Universit\`a di Pavia, via Ferrata 1, 27100, Pavia, Italy}

\begin{abstract}
In this paper we introduce methods for electromagnetic wave
propagation, based on splines and on T-splines. We define spline
spaces which form a De Rham complex and, following the isogeometric
paradigm, we map them on domains which are (piecewise) spline or NURBS
geometries.  We analyse their geometric structure, as related to the
connectivity of the underlying mesh, \Gd and we give a  physical
interpretation of the fields  degrees-of-freedom through the concept of
control fields. \B  The theory is then extended to the case of meshes with T-junctions, leveraging on the recent theory of T-splines. The use of T-splines enhance our spline methods with local refinement capability and numerical tests show the efficiency and the accuracy of the techniques we propose. 
\end{abstract}

\end{frontmatter}


\section{Introduction}
Electromagnetic field computations and, more generally, the numerical discretization of equations enjoying a relevant geometric structure, is one of the most interesting challenge of numerical analysis for PDEs  and several results have been obtained in the last decade. Indeed, only for Galerkin methods, three Acta Numerica overview papers have been  published:  by Hipmair \cite{HIP02a}, by Arnold, Falk and Winther \cite{AFW06}, and by Boffi \cite{Boff10}, addressing different aspects of the problem.

 On the one hand, discrete schemes have to preserve the geometric structure of the underlying PDEs in order to avoid spurious behaviors, instability or non-physical solutions (see e.g., the pioneering paper \cite{BFGP}). \Bd For electromagnetics, as it is clear from the references above, numerical schemes have to be related with a discrete De Rham complex. \B On  the other hand, especially in view of high frequency computations, numerical schemes need to be efficient and accurate. This requires many features, and among others it requires adaptivity, or at least local mesh refinement capability, in order to capture the strong singularities of the electromagnetic field,   possibly driven by \emph{a-posteriori} error estimator as, e.g., in \cite{BrSc08}. 

In this paper we present and analyse discretization techniques for electromagnetic fields based on splines and  generalizations of splines, as NURBS (\cite{Piegl}) or T-splines (\cite{Sederberg_Zheng} or below). Our work originates from IsoGeometric Analysis (IGA),  \cite{IGA-book}. Isogeometric analysis has been introduced in 2005 by Hughes and co-authors in the seminal paper \cite{HCB05} to solve structural mechanic problems directly on the geometry output by a CAD system, and has set the paradigm to use splines, NURBS or their generalization as generating functions for the construction of Galerkin spaces. This idea has been proved to be extremely effective and IGA is spreading very fast across different scientific communities: structural mechanics (see e.g., \cite{IGA-book}, \cite{MR2443159}, \cite{Cottrell_Reali_Bazilevs_Hughes}, \cite{Echter2010374}, \cite{Lipton_Evans_Bazilevs}, \cite{rank12}, \cite{MR2835758}), geometric modeling (see e.g., \cite{DQS10}, \cite{Martin_Cohen_Kirby}, \cite{CoMa10} and also \cite{Zhang2012}) and numerical analysis (see e.g., \cite{BBCHS06}, \cite{Beirao_Cho_Sangalli} , \cite{BRSV11}, \cite{nwidth-iga}, \cite{Vuong_giannelli_juttler_simeon}, \cite{BePa12}, \cite{kleissieti}). 

 In this paper we present the recent advances in the use of the isogeometric paradigm and  spline-based  methods for  electromagnetic wave computations. This research has started with the two papers \cite{Buffa_Sangalli_Vazquez} and \cite{BRSV11} and can likely be considered as still in infancy (see also \cite{ratnani} for the applications of this results). This paper aims at showing the potential impact  of these techniques in the electromagnetic community by addressing several aspects: from the geometric structure of the proposed methods, to local refinement strategies. 

  We introduce the \textbf{spline complex} studied in \cite{BRSV11} (see \eqref{eq:3D-parametric-diagram-bc} and \eqref{eq:diag-discrete}) and we present its properties: we construct canonical bases so that the matrices representing differential operators are the incidence matrices of the underlying meshes, and this enlightens the relation between the spline complex and the geometry of the underlying meshes. We show that for different choices of the  degree of splines, the spline complex is isomorphic to the co-chain complex or to the chain complex of the underlying mesh.
Besides this interesting fact, we also introduce the concept of \emph{control fields} in analogy to control points which are ubiquitous in spline theory (see e.g., \cite{cohen2001geometric} or \cite{DeBoor}) which provide the correct physical interpretation of degrees of freedom. 
Finally, we extend the results of \cite{BRSV11} to \textbf{multi-patch} geometries, i.e., geometries which are piecewise spline or NURBS mapping of the unit cube. We refer the reader to \cite{kleissieti} for a detailed description of this class of geometries. 

The second major contribution in this paper is a step towards
\textbf{adaptivity} for spline-based methods. Leveraging on the recent
work on T-splines, we design a two dimensional \textbf{T-spline
  complex} where meshes with T-junctions can be used to allow for
adaptivity.  T-splines are  the most attractive way to break the
tensor product structure of splines while keeping their structure and
their accuracy. T-splines have been introduced in
\cite{Sederberg_Zheng} and \cite{Sederberg_Cardon} and their use as a fundamental tool to enhance isogeometric analysis with adaptivity has been proposed in \cite{Bazilervs_Calo_Cottrell_Evans}.  A series of papers has followed \cite{Buffa_Cho_Sangalli}, \cite{Scott_Li_Sederberg_Hughes}, \cite{Wang2011477}, \cite{Beirao_Buffa_Cho_Sangalli},  together with the relevant class of Analysis Suitable T-splines \cite{LZSHS12}, \cite{BBCS12} which we use in our construction. The two dimensional T-spline complex is used to treat three dimensional problems with symmetry. We should also mention that  the definition and use of T-splines in three dimension are not yet well understood, but object of an intensive study. Their use will allow, on a longer time perspective, to design full adaptive algorithms, on very general geometries parametrized on  totally unstructured meshes. We refer the reader to \cite{Scott} for a monograph on the modern use of T-splines in geometric modeling. 

\Bd Finally, we should remark that the spline spaces we study in this
paper have a wide domain of applications and can be applied
successfully to the discretization of  other problems than
electromagnetics. In fact, they can be used to solve the Darcy flows
equations or more generally the Hodge laplacian operator as detailed
in \cite{AFW06} and  \cite{AFW-2}. Moreover, thanks to the regularity
of spline spaces, their use in fluids is very promising. In the paper
\cite{Buffa_deFalco_Sangalli} they are used for the first time to
solve the Stokes equations, \Gd in  \cite{EvHu12-4} the
Stokes eigenvalue problem  is addressed, and \B in the sequence of three papers \cite{EvHu12}, \cite{EvHu12-2}, \cite{EvHu12-3} they are applied to solve Stokes and Brinkman equations, steady and unsteady Navier-Stokes equations, providing impressive results. 
\B

The outline of the paper is the following. In Section \ref{sec:notation} we set up the notation for the problems we address, in Section \ref{sec:prel-splin-nurbs} we present known results about splines and NURBS in a self-contained way; in Section \ref{sec:spline-complex} we present the spline complex and all the related results while in Section \ref{sec:TT} we introduce the T-spline complex and analyse its properties. Finally, in Section \ref{sec:num} we present numerical results: the first ones are two and three dimensional, academic tests aiming at demonstrating the validity of the proposed approach. As a last example, we compute the propagation in  a waveguide with geometric inhomogeneity, on a three dimensional locally refined mesh. 






\section{Notation}
\label{sec:notation}

In this section we present the notation that we need to describe
the time-harmonic Maxwell problem.
Let $\Omega \subset \R^3$ be a bounded Lipschitz domain. We denote by
$L^2(\Omega)$ the space of complex square integrable functions on
$\Omega$, endowed with standard $L^2$ norm $\| \cdot
\|_{L^2(\Omega)}$, and by $\bL^2(\Omega)$ their vectorial
counterparts. The Hilbert space $H^1(\Omega)$ contains functions
of $L^2(\Omega)$ such that their first order derivatives also belong
to $L^2(\Omega)$. We denote by $H^1_0(\Omega) \subset H^1(\Omega)$ the subspace of functions with homogeneous boundary condition.
 We will also make use of the space $\Hcurl$, constituted by all
 functions in $\bL^2(\Omega)$ such that their curl also belongs to
 $\bL^2(\Omega)$, and $\Hdiv$, the space of functions in
 $\bL^2(\Omega)$ such that their divergence belongs to $L^2(\Omega)$.
 Moreover, we denote by $\Hocurl$
 (resp. $\Hodiv$) the subspace of $\Hcurl$
 (resp. $\Hdiv$) of fields with vanishing tangential
 (resp. normal) component.

For the sake of simplicity, we assume that the domain $\Omega$, referred to as physical domain in
the following,  is bounded Lipschitz and simply connected, and that
its boundary $\partial \Omega$ is connected. We also assume that it is
defined through a \Bd continuously differentiable \B  parametrization with \Bd continuously differentiable \B  inverse which we
denote as  $\bF: \Omegaref \longrightarrow \Omega$, where $\Omegaref$
will be referred to as the \emph{parametric domain}. Further assumptions on the
geometrical mapping $\bF$ will be given later. 

Some notation will be borrowed from the context of differential forms:
first of all, we define the spaces
 \begin{eqnarray*}
\Szero := \Honer, \; \Sone := \Hcurlr, \; \Stwo := \Hdivr, \; \Sthree := \Ltwor ,\\
\Xzero := \Hone, \; \Xone := \Hcurl, \; \Xtwo := \Hdiv, \; \Xthree := \Ltwo;
\end{eqnarray*}
Since the parametrization $\bF$ and its inverse are smooth, we can
define the pullbacks that relate these spaces as (see
\cite[Sect.~2.2]{HIP02a}):
\begin{equation}
\begin{array}{ll}
 \pb^0(\uzero) := \uzero \circ \bF , &\quad \uzero \in \Xzero , \\
 \pb^1(\uone) := (D\bF)^T(\uone \circ \bF) , &\quad \uone \in \Xone , \\
 \pb^2(\utwo) := \det(D\bF) (D\bF)^{-1} (\utwo \circ \bF) , & \quad \utwo \in \Xtwo , \\
 \pb^3(\uthree) := \det(D\bF) (\uthree \circ \bF) , & \quad \uthree \in \Xthree , \label{def:iota3}
\end{array}
\end{equation}
where $D\bF$ is the Jacobian matrix of the mapping $\bF$. Then, due to the curl and divergence conserving properties of $\pb^1$ and $\pb^2$, respectively (see \cite[Sect.~3.9]{Monk}, for instance), the following commuting De~Rham diagram is satisfied (see \cite[Sect.~2.2]{HIP02a}):
\begin{equation} \label{eq:diag1}
\begin{CD}
\mathbb{R} @>>> \Szero @>\hat \grad>> \Sone @>\hat \curl>> \Stwo @>\hat \div>> \Sthree @>>> 0 \\
@. @A\pb^0AA @A\pb^1AA @A\pb^2AA @A\pb^3AA \\
\mathbb{R} @>>> \Xzero @>\grad>> \Xone @>\curl>> \Xtwo @>\div>> \Xthree @>>> 0 . \\
\end{CD}
\end{equation}

We are also interested in  spaces with boundary
conditions,  denoted with the subindex $0$,
\begin{eqnarray*}
\Szeroo := \Honeor, \; \Soneo := \Hocurlr, \; \Stwoo := \Hodivr, \;
\Sthreeo :=  \Ltwoor, \\
\Xzeroo := \Honeo, \; \Xoneo := \Hocurl, \; \Xtwoo := \Hodiv, \;
\Xthreeo := \Ltwoo,
\end{eqnarray*}
for which  the De~Rham diagram reads
\begin{equation} \label{eq:diag2}
\begin{CD}
0 @>>> \Szeroo @>\hat \grad>> \Soneo @>\hat \curl>> \Stwoo @>\hat
\div>> \Sthreeo @>\int >> \mathbb{R}  \\
@. @A\pb^0AA @A\pb^1AA @A\pb^2AA @A\pb^3AA \\
0 @>>> \Xzeroo @>\grad>> \Xoneo @>\curl>> \Xtwoo @>\div>> \Xthreeo @> \int>> \mathbb{R} , \\
\end{CD}
\end{equation}
which also expresses the integral preserving property of
$\pb^3$. 

\Bd 
\begin{remark}
As it is well known, the exactness of the sequences \eqref{eq:diag1} and \eqref{eq:diag2} relies on the assumption that $\Omega$ (and $\hat\Omega$) has a  trivial topology. All what we develop in this paper applies in principle also to the case of arbitrary topology but we do not present all the details here. 
\end{remark}
\B 

\section{Preliminaries on splines and NURBS}
\label{sec:prel-splin-nurbs}

We give here a brief overview on B-splines and, in the spirit of \cite{Bazilervs_Calo_Cottrell_Evans}, we also introduce some concepts that will be needed in the definition of T-splines. For more details on B-splines we  refer \B  the reader to \cite{HCB05,Piegl}.

\subsection{Univariate B-splines}
\subsubsection{Knot vector and B-spline functions, refinement, spline derivatives} \label{sect:knot_vector}
Given two positive integers $p$ and $n$, we say that $\Xi := \{\xi_1,\dots, \xi_{n+p+1}\}$ is a $p$-open knot vector if
$$
\xi_1 =\ldots=\xi_{p+1} < \xi_{p+2} \le \ldots \le
\xi_{n} < \xi_{n+1}=\ldots=\xi_{n+p+1},
$$
where repeated knots are allowed and denote by  $m_j$  the multiplicity of the knot
 $\xi_j$. We assume  $ m_j \leq p+1$ for all internal knots.\B

From the knot vector $\Xi$, B-spline functions of degree $p$ are defined following the well-known Cox-DeBoor recursive formula: we start with piecewise constants ($p=0$):
\begin{equation}\label{eq:cox-deboor-1}
N_{i,0}(\zeta) = \left \{
\begin{array}{ll}
1 & \text{if } \xi_i \leq \zeta < \xi_{i+1}, \\
0 & \text{otherwise},
\end{array}
\right.
\end{equation}
and for $p \ge 1$ the \emph{B-spline}  functions are defined by the recursion
\begin{equation}\label{eq:cox-deboor-2}
N_{i,p}(\zeta) = \frac{\zeta - \xi_i}{\xi_{i+p} - \xi_i} N_{i,p-1}(\zeta) + \frac{\xi_{i+p+1} - \zeta}{\xi_{i+p+1} - \xi_{i+1}} N_{i+1,p-1}(\zeta).
\end{equation}
 This gives a set of $n$ B-splines that form a basis of the space of
 \emph{splines}, that is, piecewise polynomials of degree $p$ with
 $p-m_j$ continuous derivatives at the internal knots $\xi_j$, for $j
 = p+2, \ldots, n$.  We denote this univariate spline space by
 \begin{equation}
   \label{eq:spline-univariate-space}
   \sp_p(\Xi) = \text{span} \{ N_{i,p}, \, i=1,\ldots, n \}
 \end{equation}
  An example of some B-splines is given in Figure~\ref{fig:1}.
\begin{figure}[pt!]
\centering
\includegraphics[width=.5\textwidth]{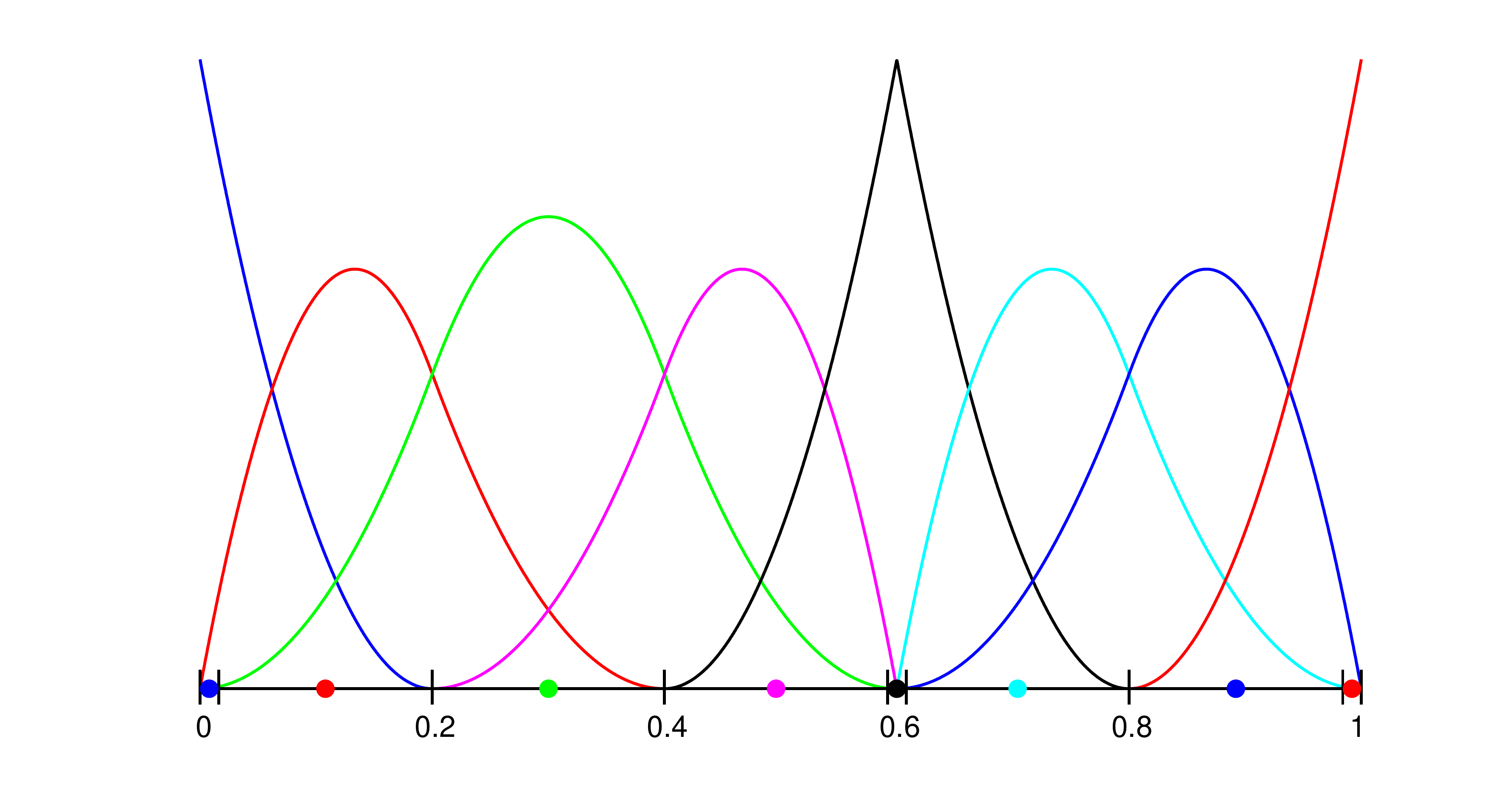} \hspace{-1cm}\includegraphics[width=.5\textwidth]{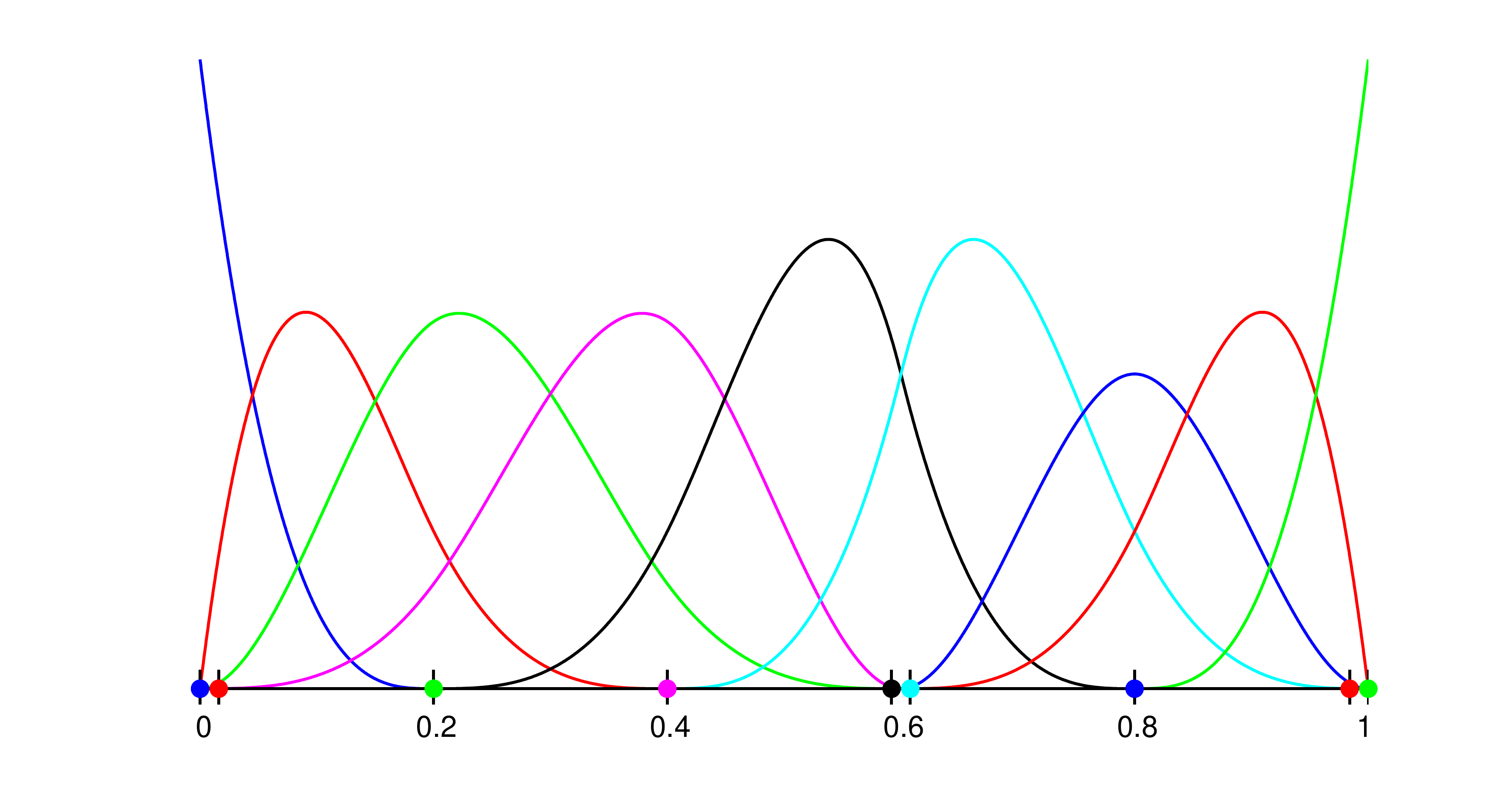}
\caption{Example of B-splines of degree $2$ (left) and $3$ (right).} \label{fig:1}
\end{figure}
 Notice that the B-spline function $N_{i,p}$ is supported in the
 interval $[\xi_i, \xi_{i+p+1}]$, and in fact its definition only
 depends on the knots within that interval. For this reason, we define
 the \emph{local knot vector} $\Xi_{i,p} = \{\xi_i, \xi_{i+1}, \ldots,
 \xi_{i+p+1}\}$, and we will sometimes denote $N_{i,p}(\zeta) \equiv
 N[\Xi_{i,p}](\zeta)$.

In the context of splines, three kinds of refinement are possible, as explained in \cite{HCB05}:
\begin{enumerate}
\item  \emph{$k$-refinement} which corresponds to successive application of the Cox-DeBoor formula
  \eqref{eq:cox-deboor-1}--\eqref{eq:cox-deboor-2}. Regularity is
  raised together with the degree:  therefore, the spaces
  \eqref{eq:spline-univariate-space} are not nested under
  {$k$-refinement} but, at each step (degree and regularity elevation),
  the dimension of the space increases by $1$.  The name {$k$-refinement} has been introduced in \cite{HCB05};
\item \emph{$h$-refinement} which corresponds to mesh refinement and is obtained by knot insertion. Let $\bar\Xi := \{\xi_1,\dots,\xi_{k}, \bar \xi, \xi_{k+1}, \xi_{n+p+1}\}$ be the knot vector after inserting the knot $\bar \xi$ in $\Xi$. Then, the new B-spline functions $\{ \bar N_{1,p} (\zeta),\ldots, \bar N_{n+p+2,p} (\zeta)\} $  are constructed as follows:
  \begin{equation}
    \label{eq:href}
  \bar N_{i,p} (\zeta)  = \alpha_i N_{i,p} (\zeta) + (1-\alpha_i) N_{i-1,p} (\zeta) 
  \end{equation}
  where  $\alpha_i =1  $, if  $1\leq i\leq k-p$, $\alpha_i =
  \frac{\displaystyle \strut  \bar \xi - \xi_i }{\displaystyle \strut
    \xi_{i+p} - \xi_i}$, if $k-p+1\leq i\leq k$ and, $\alpha_i =0$ for
  $k+1\leq i \leq n+p+2$.  When $\bar \xi$ is  equal to $\xi_k$ or
  $\xi_{k+1}$ or to both, the  knot insertion  corresponds to reduction of the inter-element regularity at $\bar \xi $. \B
\item \emph{$p$-refinement} which corresponds to the degree raising with fixed interelement regularity, and generates a sequence of nested spaces. 
\end{enumerate}

Assuming the maximum multiplicity of the internal knots is less than
or equal to $p$,  i.e., the B-spline functions are at least continuous, the derivative of the B-spline $N_{i,p}$ is a spline
as well. Indeed, it belongs to the spline  space  $S_{p-1}(\Xi')$,
where $\Xi' = \{ \xi_2, \ldots, \xi_{n+p} \}$ is a ($p-1$)-open knot
vector. Obviously, the regularity of splines in $\sp_{p-1}(\Xi')$ is one less than the
regularity in $\sp_p(\Xi)$.

 In the following we assume that $\xi_1 = 0$ and $\xi_{n+p+1} = 1$. The domain  $(0,1)$ of definition of the spline functions is the one-dimensional  \emph{parametric} domain. On it, the knot vector
$\Xi$ induces a partition  of the interval $(0,1)$ that we denote by
$\M$.  Precisely, we define  $\M$ as the set of the knot spans
$(\xi_i,\xi_{i+1})$, $i=\lceil p/2\rceil +1 ,...,n + \lfloor
p/2\rfloor$, that can also be empty due to knot multiplicity greater
than $1$. Empty intervals still play a role in the definition of
B-splines and are graphically represented as points close one to the
other, as proposed in \cite{Buffa_Cho_Sangalli}.  
Note that in this representation of $\M$, the number of
lines is the knot multiplicity with one exception: for each
boundary knot (at 0 or 1) of an open knot vector in $\M$ we represent
only a  multiplicity  of  $\lfloor p/2\rfloor + 1$, which is  $(p+1)/2$
lines if $p$ is odd, and $p/2+1$ lines if $p$ is even (see Figure~\ref{fig:1}). \B  The reason for
this construction of $\M$ will be motivated in the next
section. 

 Finally, it is worth  noting the relationship between the space
$\sp_p(\Xi)$ and the space of derivatives $\sp_{p-1}(\Xi')$, and their
respective meshes $\M$ and $\M'$. The meshes $\M$ and $\M'$ may differ only as regards the number of points at the boundary. 
Indeed, according to the definition above,  if $p$ is odd both meshes
coincide, and if $p$ is even the number of elements of $\M'$ with
respect to $\M$ is reduced by two, one on each side. 
\B

\subsubsection{Anchors and Greville sites}\label{sect:anchors}

In this section we present the concept of anchors and of Greville sites as points 
in the parametric space $(0,1)$  which may be associated to each B-spline function. 
Greville sites,  which are also known as knot averages, \B are classical and can be found for instance in \cite{DeBoor},
while the concept of anchors has been introduced recently in
\cite{Bazilervs_Calo_Cottrell_Evans}. %

Since splines are not interpolatory, the association of functions to
points (or, as we will see, other geometric entities)  is somehow more
arbitrary than with Lagrangian  finite elements. Anchors and Greville
sites are  two different choices, and we present here both.  Anchors
are very useful
when dealing with non-tensor product extensions of splines as
T-meshes, while Greville sites (and related geometric entities) carry

degrees of freedom in a more natural way.


  Given a B-spline function  $N_{i,p}(\zeta)$,  and its local knot vector  $\Xi_{i,p} =\{  \xi_i, \xi_{i+1}, \ldots, \xi_{i+p+1}\}$, we set: if $p$ is odd, the anchor $A$   associated with $N_{i,p}(\zeta)$ is  the central knot of $\Xi_{i,p} $. If $p$ is even, the anchor $A$   associated with $N_{i,p}(\zeta)$  is chosen to be the midpoint of the central knot span of $\Xi_{i,p} $, namely: $\zeta^A := \frac{\displaystyle \strut \xi_{i+p/2} + \xi_{i + p/2+1}}{\displaystyle \strut 2}$. The position of the anchors for degrees $p=2$ and $p=3$ are represented in Figure~\ref{fig:1}.

Note that obviously the correspondence between anchors and B-splines functions is one to one, but different anchors $A\neq \bar{A}$  may lie at the same position $\zeta^A=\zeta^{\bar{A}}$. A remedy to this abuse of notation, 
at the cost of more complex definition,  is proposed  in \cite{BBSV12} 
where  the use of both an \emph{index} and a \emph{parametric} domain is proposed. 

The set of anchors is denoted as $\A_p=\A_p(\Xi)$. When $p$ is odd
anchors are located at all knots \B   of the partition $\M$ (which
may be repeated), while when $p$ is even  anchors are located at midpoints
 of all elements in $\M$ (including the ones of zero
area). Indeed, this fact is the reason for the definition of $\M$ in
particular as regards to the multiplicity of boundary knots.\B

Most often, we will use anchors to index functions and local knot vectors. Namely, for an anchor $A\in \A_p$,  $\Xi^A_p$ and $B^A_p(\zeta)$ will denote the corresponding local knot vector and B-spline function, respectively.  When no confusion occurs, the subindex may be removed.

\begin{remark}\label{rem:interpolatory-B-splines}
The B-splines are, in general,  not interpolatory at the anchor $A \in \A_p(\Xi)$, while they are interpolatory at knots having multiplicity $p$. 
This always happens at $\zeta=0$ and $\zeta=1$, and happens in the interior of
the parametric domain where the basis is $C^0$ continuous, i.e., at knots with multiplicity $p$.  See e.g., Figure~\ref{fig:1}(left).
\end{remark}

Given $A\in \A_p$,  and $\Xi^A_p =\{  \xi_i,.., \xi_{i+p+1} \}$ for some $i$, then the \emph{Greville site} is defined as: 
\begin{equation}
  \label{eq:greville1D}
  \gamma^A = \frac{\xi_{i+1} + \ldots + \xi_{i+p}}{p}.
\end{equation}
Unlike anchor positions, Greville sites are all different one from the
other, when the multiplicities $m_j$ verify $m_j\leq p$ and thus
B-splines are all continuous. The Greville sites induce a partition
of the unit interval, referred as Greville mesh and denoted $\M_G$.
These concepts are ubiquitous in spline theory and geometry representation. Greville sites are intimately related to
 \emph{control points} and \emph{control polygon} whose
properties we briefly recall in the next section.

\subsubsection{B-spline curves}
A B-spline curve $\Gamma$  in $\R^3$ is defined by a parametrization in the interval $(0,1)$, in the form
\begin{equation}\label{eq:Bspline-curve}
\bF(\zeta) = \sum_{A \in \A_p} \bC^A B^A_p(\zeta), \qquad 0 < \xi < 1,
\end{equation}
where $\bC^A \in \R^3$ are called the \emph{control points}. Control
points are in a one-to-one correspondence with B-spline basis
functions. The piecewise linear interpolation of the control points
gives the control polygon $\Gamma_C$. See Figure~\ref{fig:ctrl_polygon} for an example.

The control points $\bC^A $ have an important role not only in the
definition of the spline parametrization \eqref{eq:F-Bspline}, but also in
the visualization and interaction with spline geometries within CAD
softwares.
Indeed, it is common in CAD softwares to represent, together with the
parametrized curve  $\Gamma$, the control points $\bC^A$ and the
associated control polygon $\Gamma_C$. Typically,
the CAD user defines or interacts with  the control points in order to
input and modify the geometry.  
Since the B-splines are not in general interpolatory (recall Remark
\ref{rem:interpolatory-B-splines}), then the control polygon
$\Gamma_C$  differs from
$\Gamma$, but it is ``close'' to it.   Precisely,  $\Gamma_C$
converges to $\Gamma$  under $h$-refinement. 
This convergence is proved, e.g., in \cite{DeBoor}  and
discussed here below. 

We introduce the usual Lagrangian basis for piecewise linear polynomials on the Greville mesh $\M_G$, denoted by $\lambda^A(\cdot)$:
\begin{equation}\label{eq:lambda-1D}
\lambda^A(\gamma^{A'})=\left \{ 
  \begin{aligned}
    1&\text{ if } A=A',\\
    0&\text{ if } A\neq A'.\\
  \end{aligned}
\right .
\end{equation}
The control polygon $\Gamma_C$  is then parametrized by  the mapping $\bF_C:[0,1] \rightarrow
\Gamma_C $ defined by
\begin{equation}\label{eq:Fh-Bspline-1D}
 \bF_C(\zeta) = \sum_{A \in \A_{p}} \bC^A
\lambda^A(\zeta), \qquad 0 < \xi < 1,
\end{equation}
\B
that is, $\bF_C$ and  $\bF$ share the same
control points.  When $\bF$ is smooth enough, the following approximation estimate holds (see, e.g.,
\cite[Ch. XI]{DeBoor}): 
\begin{equation}
  \label{eq:-control-polygon-convergence}
  \underset{\zeta \in [0,1] }{\sup} \| \bF(\zeta)  -   \bF_C(\zeta) \| \simeq h^2,
\end{equation}
$h$ denoting the mesh-size. In other words, $\Gamma_C$ approximates
$\Gamma$ up to an error $O(h^2)$ under $h$-refinement. A graphical representation of this convergence can be seen in Figure~\ref{fig:ctrl_polygon}. 

\begin{figure}[pt!]
\centering
\includegraphics[width=.48\textwidth]{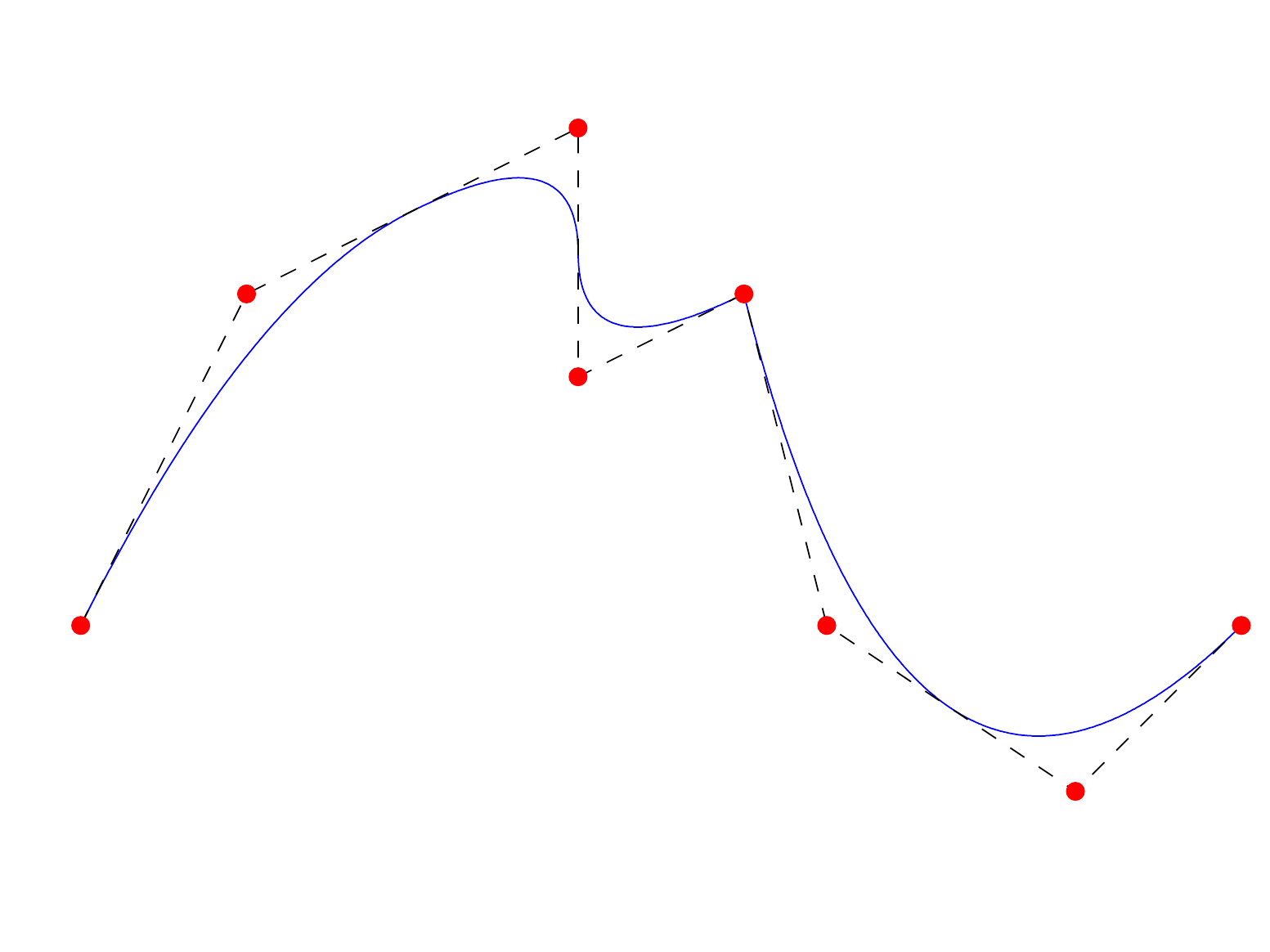} \includegraphics[width=.48\textwidth]{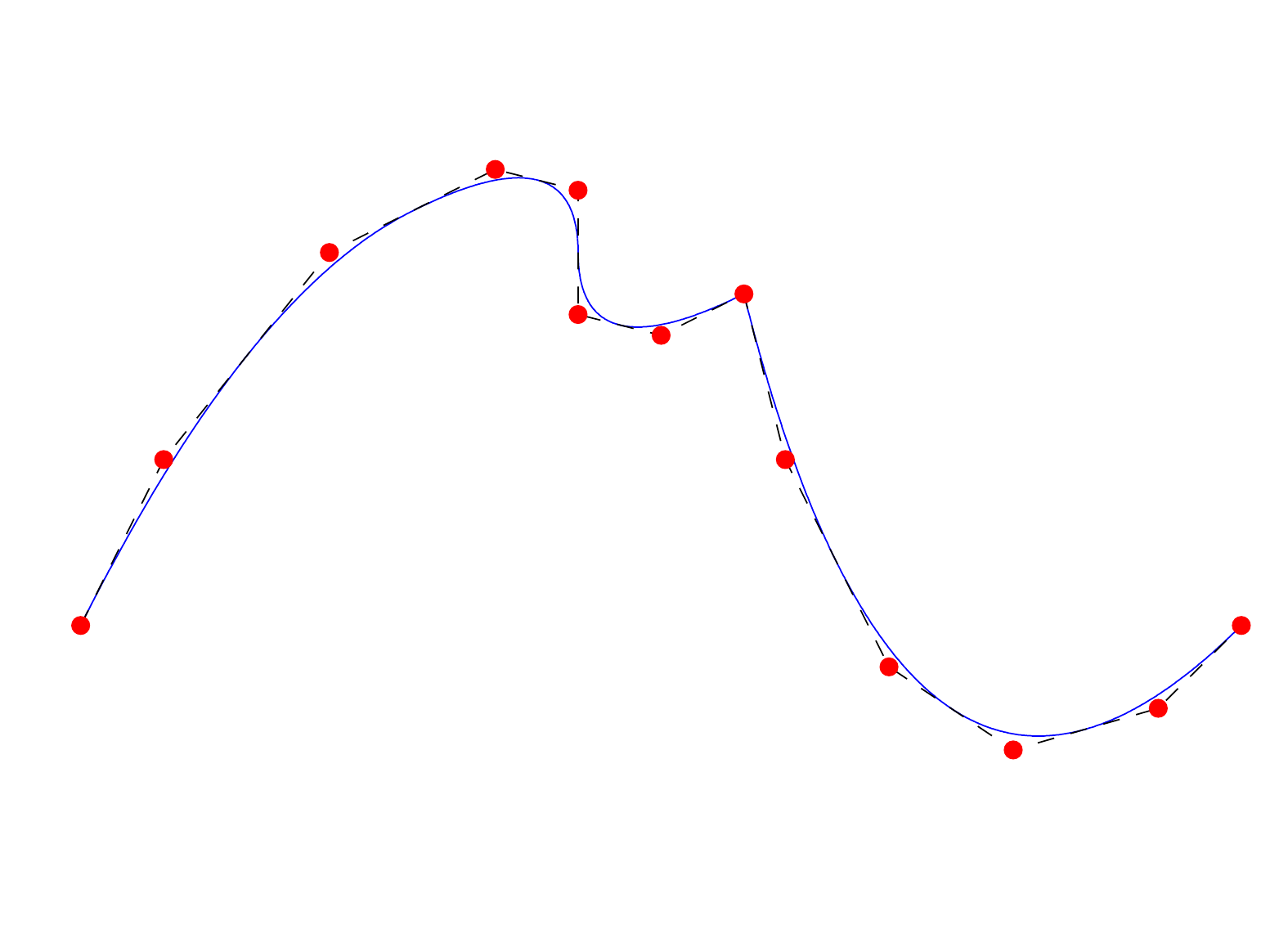} \vspace{-1cm}
\caption{B-spline curve and its control polygon (left), and the same curve after one step of $h$-refinement (right).} \label{fig:ctrl_polygon}
\end{figure}

\B

\subsection{Multivariate B-splines}\label{sect:bivariate}
Multivariate B-splines are defined from univariate B-splines by tensor
product, see for instance \cite{Piegl,DeBoor}. Anchors are defined
in a similar way. We give here a quick overview. 

\subsubsection{Knot vectors, B-spline functions, anchors, Greville sites}\label{sect:b-anchors}
Let $d$ be the space dimensions (in practical cases,
$d=2,3$). Assume  $n_\ell \in \N$, the degree $p_\ell \in \N$ and  the
\mbox{$p_\ell$-open} knot vector  $\Xi_\ell = \{\xi_{\ell,1}, \ldots,
\xi_{\ell,n_\ell + p_\ell + 1} \}$ are given, for $\ell =
1,\ldots,d$. These  knot vectors define a tensor product mesh $\M$ in
the parametric domain $\hat\Omega = (0,1)^d$ 
where, as in Section
\ref{sect:anchors}, we have to  take into account the knot
multiplicity.  The multiplicity of a knot vector in $\Xi_\ell$ is
represented graphically by lines ($d=2$) or planes ($d=3$) one close to the other, while the boundary is treated exactly as in one dimension.

The set of anchors is defined on $\M$ as the Cartesian product
$\A_{p_1,\ldots,p_d}(\Xi_1,\ldots,\Xi_d) = \A_{p_1}(\Xi_1) \times
\ldots \times  \A_{p_d}(\Xi_d)$. Considering, for example, the
trivariate case ($d=3$) and recalling the definitions from Section
\ref{sect:anchors} for the univariate case, we have that: if all
$p_\ell$  are odd the anchors lie at the vertices of the mesh; if both
$p_1$ and $p_2$ are odd and $p_3$ is even, then the anchors are
middle-points of the vertical edges  of  $\M$;  if both $p_1$ and $p_2$ are even and $p_3$ is odd, then the anchors are
centers of the horizontal faces of  $\M$;   if  all $p_\ell$  are
even the anchors lie at the center of the elements  of  $\M$, and so
on.  As in the univariate case, the anchors may be located at the
center of zero length edges or zero area faces or empty elements,
according to  knot repetition.
Also,  the computation of the local knot vectors for each anchor
follows from the univariate case.  Given an anchor $A =(A_1, A_2, A_3)\in
\A_{p_1,p_2,p_3} \equiv \A_{p_1,p_2,p_3}(\Xi_1,\Xi_2,\Xi_3)$, we have that its coordinates are
$(\zeta_1^A, \zeta_2^A, \zeta_3^A) = (\zeta^{A_1}, \zeta^{A_2},
\zeta^{A_3})$.
The three  local knot vectors (one in each coordinate direction)  corresponding to $A$ are defined as $\Xi_i^A := [\Xi_i]_{p_i}^{A_i}$ for $i=1,2,3$ and  
 the B-spline  associated to $A$ is constructed by tensor product as: 
\begin{equation}\label{eq:Bsplines-trivariate}
B^A_{p_1,p_2,p_3}(\bzeta) = B^{A_1}_{p_1}(\zeta_1)  B^{A_2}_{p_2}(\zeta_2)  B^{A_3}_{p_3}(\zeta_3).
\end{equation}
with $\bzeta = (\zeta_1, \zeta_2, \zeta_3) \in \hat\Omega = (0,1)^3$.  

The B-spline functions \eqref{eq:Bsplines-trivariate} span the space $\sp_{p_1,p_2,p_3}(\Xi_1,\Xi_2,\Xi_3)$ (or simply
$\sp_{p_1,p_2,p_3}$), which is the {\Rd space of piecewise
polynomials of degree $p_\ell$ in the $x_\ell$ direction on $\M$}, whose continuity at the internal mesh plane $\zeta_\ell =
\xi_{\ell,k}$  is $C^{p_\ell-m_{\ell,k}}$,  $m_{\ell,k} $ being the
multiplicity of $\xi_{\ell,k} $ in the knot vector $ \Xi_{\ell}$.

To each anchor $A \in\A_{p_1,p_2,p_3}(\Xi_1,\Xi_2,\Xi_3)$ (or, equivalently, to each B-spline
function \eqref{eq:Bsplines-trivariate}) we also associate a Greville
site in the natural way, that is 
\begin{equation}
  \label{eq:greville3D}
 \boldsymbol  \gamma^A = (\gamma_1^A,\gamma_2^A,\gamma_3^A)
\end{equation}
where each $ \gamma_i^A$ is defined as in  \eqref{eq:greville1D}, from
the local knot vector  $\Xi^A_i$. 
Connecting adjacent  Greville sites,  we obtain the \emph{Greville
  mesh} $ \M_G$, which is a regular tensor product  mesh with
all elements of positive volume.

\subsubsection{Spline and NURBS geometries, multi-patch domains}

\label{sec:spline-geometries} 
Analogously to spline curves, a trivariate single-patch spline parametrization of the domain $\Omega \subset\R^3$ is  $\bF:\hat  \Omega \rightarrow
\Omega $ defined as a linear combination of B-splines,
\begin{equation}\label{eq:F-Bspline}
\bF(\bzeta) = \sum_{A \in \A_{p_1,p_2 ,p_3}} \bC^A
B^A_{p_1,p_2,p_3}(\bzeta), \qquad \text{with } \bzeta \in \hat \Omega ,
\end{equation}
where $\bC^A \in \R^3$ are called \emph{control points}. In a similar
way, it is also possible to define bivariate spline domains in $\R^2$
 or surfaces in $\R^3$, which are commonly used  in CAD (see, e.g.,  \cite{Piegl, cohen2001geometric}).

The control points $\bC^A $  have again the same  important role in
the visualization and interaction with geometries within CAD
softwares. Now, the concept of control polygon generalizes to the control mesh $\pM_C$, that is, the mesh connecting the
control points. Figure~\ref{fig:pipe-ctrl-net} shows an example geometry, with its control points and
control mesh. The control mesh defines a 
polyhedral domain, denoted $ \Omega_C$, which is an approximation of
$\Omega$: again,  the control domain
$ \Omega_C$ converges to $\Omega$  under $h$-refinement. 
\begin{figure}[h!]
\centering
\includegraphics[width=.4\textwidth]{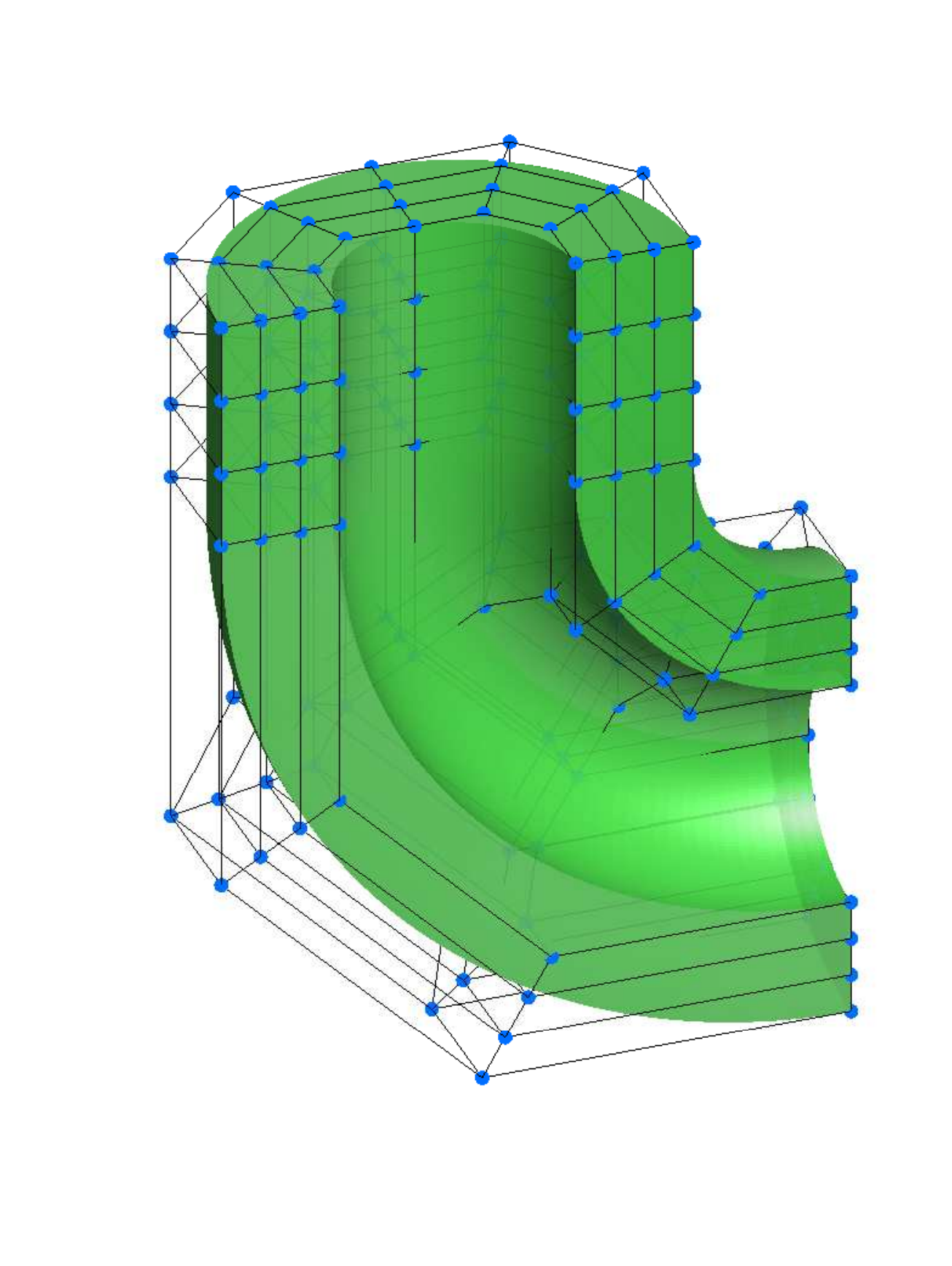}
\vspace*{-1cm}
\caption{Representation of a geometry (green), with its control points (blue)  and control mesh (black lines) for splines of degree $3$.} \label{fig:pipe-ctrl-net}
\end{figure}
  This is
stated as in the univariate case: we introduce the usual Lagrangian
basis for piecewise trilinear polynomials on the tridimensional
Greville mesh $\M_G$, still denoted  by $\lambda^A(\cdot)$, for the
sake of brevity,
\begin{displaymath}
\lambda^A(\boldsymbol \gamma^{A'})=\left \{ 
  \begin{aligned}
    1&\text{ if } A=A',\\
    0&\text{ if } A\neq A'.\\
  \end{aligned}
\right .
\end{displaymath}
The control mesh $\pM_C$ is the image of the Greville mesh $\M_G$ through the piecewise trilinear mapping 
 $\bF_C:\hat \Omega \rightarrow
 \Omega_C$, 
\begin{equation}\label{eq:Fh-Bspline}
 \bF_C(\bzeta) = \sum_{A \in \A_{p_1,p_2 ,p_3}} \bC^A
\lambda^A(\bzeta), \qquad \text{with } \bzeta \in \hat \Omega,
\end{equation}
 which is a parametrization of  $\Omega_C$, since
\begin{displaymath}
  \bF_C ( \boldsymbol  \gamma^A) =  \bC^A.
\end{displaymath}
When $\bF$ is
 smooth enough, as for \eqref{eq:-control-polygon-convergence}, we have
\begin{equation}
  \label{eq:-control-mesh-convergence}
  \underset{\bzeta \in \hat \Omega }{\sup} \| \bF(\bzeta)  -   \bF_C(\bzeta) \| \simeq h^2.
\end{equation}

The control mesh plays a fundamental role in structural mechanics applications where the unknowns are sought as displacements of control points. In our work, we will show how this interpretation can be used also in other contexts. 
\begin{remark}\label{rem:control-points-linear-case}
  When $p_1=p_2=p_3=1$ (and all anchors have multiplicity one) the Greville sites coincide with the anchor representations, i.e.,
  $ \boldsymbol \gamma^A = \bzeta^A$, and $   \bF(\bzeta) = \bF_C (\bzeta)$, $\forall \bzeta \in   \hat \Omega $, that is,  $ \Omega_C$ and  $\Omega$
  coincide.
\end{remark}


 In CAD and isogeometric analysis the geometry is often
parametrized by Non Uniform Rational B-splines (NURBS). 
NURBS are generated  from projective
transformations of splines (see \cite{Piegl}).   A trivariate single-patch NURBS
parametrization of the domain $\Omega \subset\R^3$ is a function  $\bF:\hat  \Omega \rightarrow
\Omega $ defined   as quotient of  linear combination of  B-splines, 
\begin{equation}\label{eq:F}
\bF(\bzeta) = \frac{\sum_{A \in \A_{p_1,p_2 ,p_3}}   \bC^A w^A B^A_{p_1,p_2 ,p_3}(\bzeta)}{\sum_{A'\in \A_{p_1,p_2 ,p_3}} w^{A'} B^{A'}_{p_1,p_2 ,p_3}(\bzeta)}, \qquad \bzeta \in \hat\Omega,
\end{equation}
where $ \bC^A $ are the NURBS control points  and  $w^A$ the positive
NURBS  weights.

In order to enhance flexibility and allow for more complex  geometries, the  definition of tensor-product spline and NURBS
parametrized domain can be generalized to domains that are union of
$N$ images of  cubes; precisely 
\begin{equation}
  \label{eq:F-multi-patch}
  \text{closure } (\Omega )= \bigcup_{k=1,\ldots,N} \text{closure } (\Omega_k)
\end{equation}
where the $ \Omega_k = \bF_k (\hat \Omega)$ are referred to as \emph{patches}, and are assumed to be disjoint. Each patch has its own
parametrization $ \bF_k $, defined on its own spline or NURBS space. The
whole $\Omega$ is then referred  to as a \emph{multi-patch} domain.
For the construction of discrete fields on  a \emph{multi-patch}
domain $\Omega$ we will introduce in Section
\ref{sec:multi-patch-complex-continuity} suitable  \emph{conformity} assumptions. These will
restrict the framework to configurations where it is easy to implement
the proper continuity of the fields at the patches interface.  

In this paper,  $\Omega$ is assumed to be  parametrized  either  by  
spline or NURBS functions but the unknown fields are always
constructed by splines. This means that, in case of NURBS
geometries,  we  leave the isoparametric concept which is a
fundamental assumption for isogeometric methods in the context of continuum mechanics (see \cite{IGA-book}). 



\section{The spline complex}
\label{sec:spline-complex}
This section is devoted to present the \mbox{spline} spaces that are
compatible with the  De~Rham complex. The definition of the spaces is
taken from \cite{BRSV11}, and is given in three dimensions (though the
same construction is generalizable  to arbitrary dimension). We first
recall the construction on the parametric domain $\Omegaref$, and then
the discrete spaces  on the physical domain $\Omega$ are obtained  
 by the {\it push-forward} mapping associated to \eqref{def:iota3}. As shown
in \cite{BRSV11}, it is
also possible to complement these spaces with commuting and continuous
projectors, in the setting of the so called Finite Element Exterior
Calculus (see \cite{AFW06}), however  this issue is not discussed
here. Instead, we discuss  the selection of a suitable 
basis for the implementation of the proposed spaces, and  the meaning
of the associated degrees-of-freedom. We will see that the proposed spline spaces
are a natural high-order extension of classical low-order N\'ed\'elec hexahedral
finite elements of the first family (see \cite{NED80}),  obtained in this setting for degree $p_1=p_2=p_3=1$, and that in a natural way they  are related to
cochain or chain complexes of the mesh where they are defined.

\subsection{Complex on the parametric domain $\hat \Omega$}
\label{sec:compatible -splines-parametric-domain}

We recall the following property of univariate splines, from Section
\ref{sect:knot_vector}: the derivative  of a (continuous) spline is a
spline, and in particular
\begin{equation}
  \label{eq:1D-diagram}
        \begin{CD}
{S_{p}(\Xi)} @ > \frac{d}{d\zeta}>>  {S_{p-1}(\Xi')},
\end{CD}
\end{equation}
where $\Xi'$ is the $(p-1)$-open knot vector that coincides with 
 the $p$-open knot vector  $\Xi$ except for the boundary knot 
repetitions. Moreover, we have that the
derivative of the B-spline associated to an anchor $A$ in
$\A_p(\Xi)$ is a linear combination of the B-splines
associated to the previous and next  adjacent anchors $A^-$ and $A^+$ in $\A_{p-1}(\Xi')$ (only one
adjacent anchor for the first and last $A \in \A_p(\Xi)$); precisely,
denoting by $\Xi_p^A $ the local knot vector (formed by  $p+2$ knots) of $A$
and  by  $ \Xi_{p-1}^{A^\pm} $ the local knot vectors  (formed by  $p+1$
knots) of $A^\pm$, and  with the general  notation of  Section
\ref{sect:knot_vector}, we have
\begin{equation}  \label{eq:derivative-of-splines}
\frac{d}{d\zeta}  N[\Xi_p^A](\cdot) = \frac{p}{|\Xi_{p-1}^{A^-} |}
N[\Xi_{p-1}^{A^-}](\cdot)   - \frac{p}{|\Xi_{p-1}^{A^+} |} N[\Xi_{p-1}^{A^+}](\cdot) ,
\end{equation}
where $ |\Xi_{p-1}^{A^\pm} |$ are the length of the support of the
$(p-1)$-degree B-splines $
N[\Xi_{p-1}^{A^\pm}]$, that is, the difference of the last and first
knots in the local knot vectors $ \Xi_{p-1}^{A^\pm} $. An example is given in Figure~\ref{fig:derivative}. When $A$ is the first
(resp., last) anchor, \eqref{eq:derivative-of-splines} holds with
$N[\Xi_{p-1}^{A^-}]=0 $ (resp., $N[\Xi_{p-1}^{A^+}]=0$). This is a well known
property of B-splines (see \cite{Piegl,DeBoor}) that also suggests the
following scaling of the basis functions of $S_{p}(\Xi)$ and  $S_{p-1}(\Xi')$

\begin{equation}
  \label{eq:spline-basis}
  S_{p}(\Xi) = \text{span} \left \{ B^A_p(\cdot) \equiv
    N[\Xi_p^A](\cdot):\, A \in \A_p(\Xi) \right \} ,
\end{equation}
\begin{equation}
  \label{eq:spline-derivative-basis}
  S_{p-1}(\Xi') = \text{span}\left \{ D^A_{p-1}(\cdot)
    =\frac{p}{|\Xi_{p-1}^{A}|} N[\Xi_{p-1}^A](\cdot):\, A \in \A_{p-1}(\Xi')
  \right \} .
\end{equation}
The  scaling in \eqref{eq:spline-derivative-basis} gives the Curry-Schoenberg B-splines (see
\cite[Ch.~IX]{DeBoor}), that have been already used in isogeometric
analysis in \cite{ratnani}. 
Indeed, with the bases \eqref{eq:spline-basis} and
\eqref{eq:spline-derivative-basis}, the matrix associated to the
operator $ \frac{d}{d\zeta}$ is the \emph{edge-vertex incidence} matrix
related to the mesh $\M$,
when $p$ is odd, or the \emph{vertex-edge incidence} matrix related to $\M$,
when $p\geq 2$ is even. We recall that $\M$ also contains zero length
edges and repeated vertices.

\begin{figure}[h!]
\centering
\includegraphics[width=.5\textwidth]{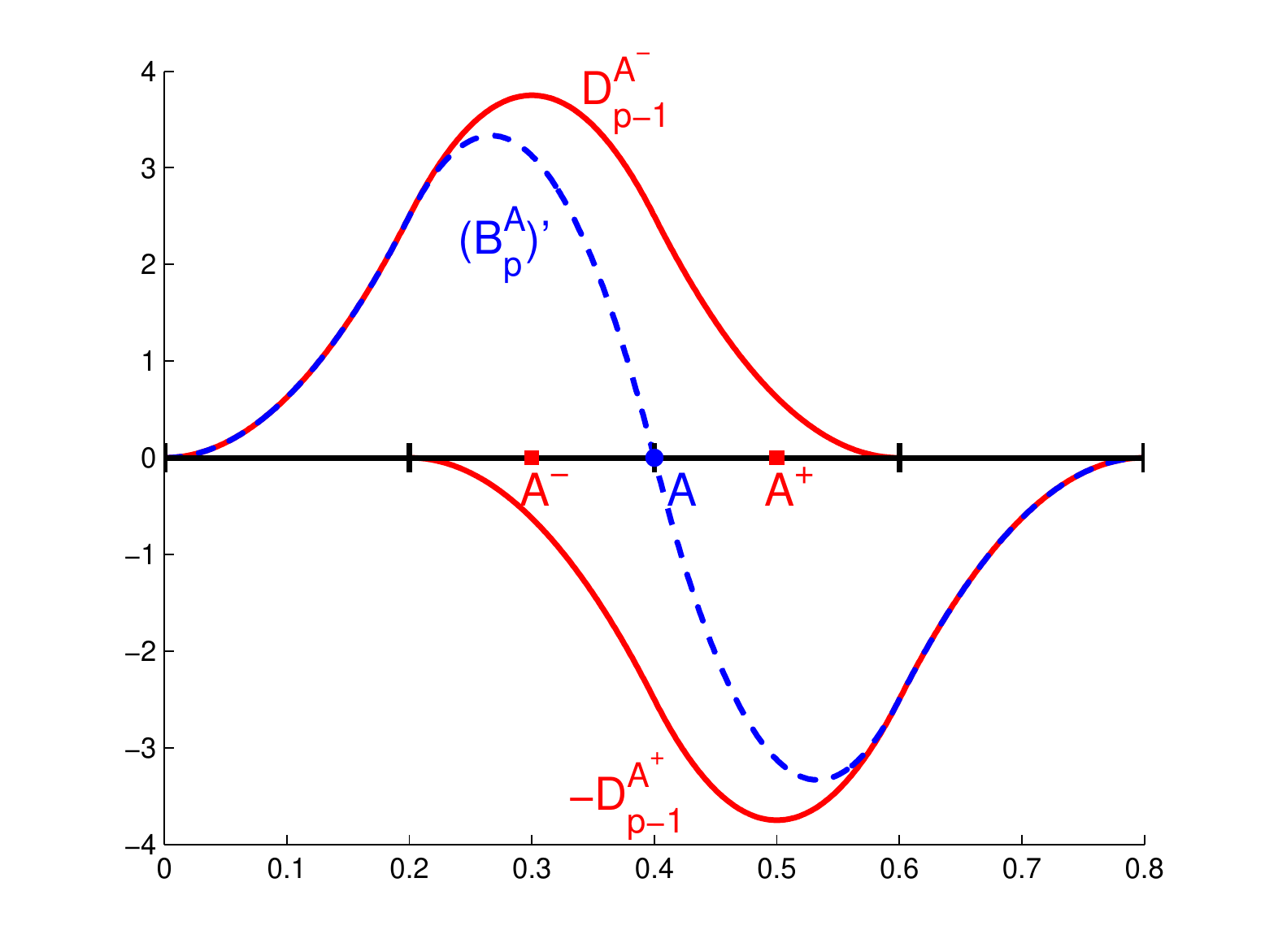}
g\caption{Derivative of the spline associated to the anchor $A$ as a linear combination of the splines associated to $A^-$ and $A^+$.} \label{fig:derivative}
\end{figure}

The observations above are the key ingredients of the trivariate
construction. Following \cite{BRSV11}, and using the notation of
Section \ref{sect:bivariate},  we introduce the following discrete
\B  spaces \B on the parametric domain $\hat \Omega$

\begin{equation}
    \label{eq:discrete-forms-on-omegahat}
  \begin{aligned}
\Szeroh &:= \Honesp , \\
\Soneh &:= \Hcurlsp ,  \\
\Stwoh &:= \Hdivsp , \\
\Sthreeh &:= \Ltwosp .
\end{aligned}
\end{equation}
From \eqref{eq:1D-diagram},  they form a De Rham complex:
\begin{equation}\label{eq:3D-parametric-diagram}
\begin{CD}
 \mathbb{R} @>>>\Szeroh @>\hat \grad>> \Soneh @>\hat \curl>> \Stwoh @>\hat \div>> \Sthreeh@>>> 0. 
\end{CD}
\end{equation}

Moreover, we have the following result.
\begin{theorem}\label{th:exact}
  The sequence \eqref{eq:3D-parametric-diagram} is exact. 
\end{theorem}
\begin{proof}
This result has been already presented in \cite{BRSV11}. We present an alternative proof, that will be useful Section~\ref{sect:T-compatible-parametric}.

We have to show that in  \eqref{eq:3D-parametric-diagram} it holds
\begin{equation}
  \label{eq:exactenss-splines-1}
\R =  \mathrm{ker}( \hat \grad),
\end{equation}
\begin{equation}
  \label{eq:exactenss-splines-2}
  \mathrm{im}(\hat \grad) =  \mathrm{ker}( \hat\curl),
\end{equation}
\begin{equation}
  \label{eq:exactenss-splines-3}
  \mathrm{im}(\hat\curl)=   \mathrm{ker}( \hat\div),
\end{equation}
\begin{equation}
  \label{eq:exactenss-splines-4}
  \mathrm{im}(\hat\div)=   \Sthreeh.
\end{equation}
In particular, we have to prove the inclusion $\supseteq$  in
\eqref{eq:exactenss-splines-1}--\eqref{eq:exactenss-splines-4}, since
the other inclusion $\subseteq$ is trivial in all cases. It is also trivial that
\begin{displaymath}
  \R \supseteq \mathrm{ker}( \hat \grad).
\end{displaymath}

Let  $\hat \bu = (\hat u_1,\hat u_2,\hat u_3)\in\Soneh $, then define
\begin{equation}
  \label{eq:phi-h-poincare}
  \hat \phi(\zeta_1,\zeta_2,\zeta_3) = \int _0^{\zeta_1}  \hat
  u_1(\eta,0,0)\, d \eta + \int _0^{\zeta_2}  \hat
  u_2(\zeta_1,\eta,0)\, d \eta  + \int _0^{\zeta_3}  \hat
  u_3(\zeta_1,\zeta_2,\eta)\, d \eta;
\end{equation}
it is easy to check that  $ \hat \bu = \hat \grad \hat
\phi $ when  $\hat \curl  \bu =
\boldsymbol 0  $, and that   $\hat\phi \in \Szeroh$; then 
\begin{displaymath}
   \mathrm{im}(\hat \grad)\supseteq  \mathrm{ker}( \hat\curl).
\end{displaymath}

Consider $\hat \varphi \in\Sthreeh $, and define  $\hat \bv = (\hat
v_1,0,0)\in\Soneh$  such that
\begin{equation}
  \label{eq:phi-h-poincare2}
  \hat v_1(\zeta_1,\zeta_2,\zeta_3) = \int _0^{\zeta_1}  \hat
  \varphi(\eta,\zeta_2,\zeta_3)\, d \eta, 
\end{equation}
as before, it is easy to check that  $ \hat \varphi = \hat \div \hat
\bv $ and that   $\hat\bv \in \Stwoh$; then 
\begin{displaymath}
  \mathrm{im}(\hat\div)\supseteq   \Sthreeh.
\end{displaymath}

In order to complete the proof we need to show that 
\begin{displaymath}
    \mathrm{im}(\hat\curl)\supseteq   \mathrm{ker}( \hat\div),
\end{displaymath}
which is implied by
\begin{equation}
  \label{eq:temp2}
    \mathrm{dim}  (\mathrm{im} (\hat\curl)) =  \mathrm{dim}(\mathrm{ker}( \hat\div)).
\end{equation}
To count dimensions recall from Section \ref{sect:knot_vector} that  $
\mathrm{dim}(S_{p_\ell}(\Xi_\ell))= n_\ell$, $
\mathrm{dim}(S_{p_\ell-1}(\Xi'_\ell))= n_\ell-1$; then from Section
\ref{sect:bivariate} and from  \eqref{eq:discrete-forms-on-omegahat}
we get 
\begin{equation}
    \label{eq:dim-of-discrete-forms-on-omegahat}
  \begin{aligned}
\mathrm{dim}(\Szeroh)& = n_1 n_2 n_3 , \\
\mathrm{dim}(\Soneh)&= (n_1 -1 )  n_2 n_3 +  n_1 ( n_2-1) n_3 + n_1 n_2 (n_3-1),  \\
\mathrm{dim}(\Stwoh) &= n_1 (n_2-1)  (n_3-1) + (n_1 -1 ) n_2 (n_3-1)+(n_1 -1 )( n_2-1)  n_3, \\
\mathrm{dim}(\Sthreeh) &= (n_1 -1 )( n_2-1)(n_3-1) .
\end{aligned}
\end{equation}
Then, by \eqref{eq:exactenss-splines-1}--\eqref{eq:exactenss-splines-2},
\begin{equation*}
  \label{eq:dimension-count-1}
  \begin{aligned}
    \mathrm{dim}  (\mathrm{im} (\hat\curl))  &  =  \mathrm{dim}
    (\Soneh) -   \mathrm{dim}  (\mathrm{ker} (\hat\curl)) \\
  &  =  \mathrm{dim}  (\Soneh) -   \mathrm{dim}  (\mathrm{im} (\hat\grad))\\
&  =  \mathrm{dim}  (\Soneh) -   \mathrm{dim}  (\Szeroh) +
\mathrm{dim}  (\mathbb{R})\\
& = 2 n_1 n_2 n_3 - n_2 n_3 - n_1 n_3  -  n_1 n_2 + 1
  \end{aligned}
\end{equation*}
and by \eqref{eq:exactenss-splines-4}
\begin{equation*}
  \label{eq:dimension-count-2}
  \begin{aligned}
     \mathrm{dim}(\mathrm{ker}( \hat\div))  &  =  \mathrm{dim}
    (\Stwoh) -   \mathrm{dim}  (\mathrm{im} (\hat\div)) \\  &  =  \mathrm{dim}
    (\Stwoh) -   \mathrm{dim}  (\Sthreeh)  \\
& = 2 (n_1-1) (n_2 -1)(n_3 -1) +   (n_2 -1)(n_3 -1) + (n_1-1) (n_3 -1)  +  (n_1-1)  (n_2 -1)\\& = 2 n_1 n_2 n_3 - n_2 n_3 - n_1 n_3  -  n_1 n_2 + 1,
  \end{aligned}
\end{equation*}
which gives \eqref{eq:temp2}, and as a consequence \eqref{eq:exactenss-splines-3}.

\end{proof}

We now show how suitable basis functions for the spaces can be constructed and as well associated 
to geometric entities of the mesh $\M$ by using the concept of anchors.  We focus on basis functions first, and inspired by 
\eqref{eq:spline-basis}--\eqref{eq:spline-derivative-basis}
we define them as follows: \B 

\begin{equation}
  \label{eq:0-forms-basis}
  \Szeroh = \text{span} \left \{ \boldsymbol \zeta \mapsto
    B^{A_1}_{p_1}(\zeta_1)  B^{A_2}_{p_2}(\zeta_2)
    B^{A_3}_{p_3}(\zeta_3):\, A= (A_1, A_2, A_3) \in
\A_{p_1,p_2,p_3}(\Xi_1,\Xi_2,\Xi_3) \right \} ,
\end{equation}
\begin{equation}
  \label{eq:1-forms-basis}
  \begin{aligned}
     \Soneh & = \text{span }  I\cup II \cup III, \text{ with}\\
 I & = \left \{ \boldsymbol \zeta \mapsto  D^{A_1}_{p_1-1}(\zeta_1)  B^{A_2}_{p_2}(\zeta_2)
    B^{A_3}_{p_3}(\zeta_3)  \hat \be_1:\, A= (A_1, A_2, A_3) \in
\A_{p_1-1,p_2,p_3}(\Xi'_1,\Xi_2,\Xi_3) \right \} ,\\
 II & = \left \{ \boldsymbol \zeta \mapsto  B^{A_1}_{p_1}(\zeta_1)  D^{A_2}_{p_2-1}(\zeta_2)
    B^{A_3}_{p_3}(\zeta_3) \hat\be_2:\, A= (A_1, A_2, A_3) \in
\A_{p_1,p_2-1,p_3}(\Xi_1,\Xi'_2,\Xi_3) \right \} ,\\
 III& = \left \{ \boldsymbol \zeta \mapsto 
B^{A_1}_{p_1}(\zeta_1)  B^{A_2}_{p_2}(\zeta_2)
    D^{A_3}_{p_3-1}(\zeta_3)  \hat\be_3:\, A= (A_1, A_2, A_3) \in
\A_{p_1,p_2,p_3-1}(\Xi_1,\Xi_2,\Xi'_3) \right \} ,
  \end{aligned}
\end{equation}

 \begin{equation}
  \label{eq:2-forms-basis}
  \begin{aligned}
     \Stwoh & = \text{span }  I \cup II \cup III , \text{ with}\\
 I  & = \left \{ \boldsymbol \zeta \mapsto  B^{A_1}_{p_1}(\zeta_1)  D^{A_2}_{p_2-1}(\zeta_2)
   D^{A_3}_{p_3-1}(\zeta_3) \hat \be_1 :\, A= (A_1, A_2, A_3) \in
\A_{p_1,p_2-1,p_3-1}(\Xi_1,\Xi'_2,\Xi'_3) \right \} ,\\
 II & = \left \{ \boldsymbol \zeta \mapsto D^{A_1}_{p_1-1}(\zeta_1)  B^{A_2}_{p_2}(\zeta_2)
    D^{A_3}_{p_3-1}(\zeta_3)  \hat \be_2   :\, A= (A_1, A_2, A_3) \in
\A_{p_1-1,p_2,p_3-1}(\Xi'_1,\Xi_2,\Xi'_3) \right \} ,\\
 III & = \left \{ \boldsymbol \zeta \mapsto D^{A_1}_{p_1-1}(\zeta_1)  D^{A_2}_{p_2-1}(\zeta_2)
    B^{A_3}_{p_3}(\zeta_3) \hat \be_3  :\, A= (A_1, A_2, A_3) \in
\A_{p_1-1,p_2-1,p_3}(\Xi'_1,\Xi'_2,\Xi_3) \right \} ,
  \end{aligned}
\end{equation}
\begin{equation}
  \label{eq:3-forms-basis}
  \Sthreeh = \text{span} \left \{ \boldsymbol \zeta \mapsto
    D^{A_1}_{p_1-1}(\zeta_1)  D^{A_2}_{p_2-1}(\zeta_2)
    D^{A_3}_{p_3-1}(\zeta_3):\, A= (A_1, A_2, A_3) \in
\A_{p_1-1,p_2-1,p_3-1}(\Xi'_1,\Xi'_2,\Xi'_3) \right \} ,
\end{equation}
where $\{\hat \be_\ell\}_{\ell=1,2,3}$ denote the canonical basis of
$\R^3$. We remark that all basis functions defined in
\eqref{eq:0-forms-basis}-\eqref{eq:3-forms-basis} are non-negative.

 We discuss now the association of the anchors of the bases
\eqref{eq:0-forms-basis}--\eqref{eq:3-forms-basis} to the mesh $\M$
that is associated to $\Szeroh$, that is, obtained from the knot
vectors $ \Xi_1,\Xi_2,\Xi_3$. 
We focus on the relevant case $p=p_1=p_2=p_3$ and consider two possible choices: $p$ is odd, or $p$ is even. \B

  When $p$ is odd, as an immediate consequence of the definition of anchors in one space dimension, we have that:
\begin{itemize}
\item anchors associated with $\Szeroh$ are $\A_{p,p,p}(\Xi_1,\Xi_2,\Xi_3)$, which are located at the vertices of $\M$, i.e., there is one degree of freedom per vertex;
\item anchors associated with $\Soneh$ are located at edges of $\M$
  and there is one degree of freedom per edge. Indeed, e.g., anchors
  associated with the first component of the space $\Soneh$, which is
  $S_{p-1,p,p} (\Xi'_1, \Xi_2, \Xi_3)$, are
  $\A_{p-1,p,p}(\Xi'_1,\Xi_2,\Xi_3)$ and are located at the edges
  oriented as $\hat \be_1$. This means that to each edge of the mesh a is associated a basis function tangential to the edge. 
\item anchors associated with $\Stwoh$ are located at faces and there
  is one anchor per face. More in detail, if we consider the first
  component of $\Stwoh$, which is  $S_{p,p-1,p-1} (\Xi_1, \Xi'_2,
  \Xi'_3)$, the corresponding anchors  are
  $\A_{p,p-1,p-1}(\Xi_1,\Xi'_2,\Xi'_3)$  and are located at the
  barycenter of the faces $f$ such that $f$ is orthogonal to
  $\hat \be_1$.  This means that a basis functions normal to the
  face is associated to the face. 
\item anchors associated with $\Sthreeh$ are $\A_{p-1,p-1,p-1}(\Xi'_1,\Xi'_2,\Xi'_3)$  and located at barycentres of all elements of $\M$. 
\end{itemize}

We now turn to the case when of even degree $p\geq 2$,
$p_1=p_2=p_3=p$.  Note that, according to our definition, and as explained in Section~\ref{sect:knot_vector}, the
meshes corresponding to  the spaces $\Soneh$, $\Stwoh$ and $\Sthreeh$
differ from $\M$ due to the different number of repeated lines at the
boundary.  Instead of working with different meshes for different
spaces, equivalently, we represent in this case too   the anchors of all
spaces on the mesh $\M$ of $\Szeroh$, keeping into account only the
interior geometrical entities for the representations of anchors of
$\Soneh$, $\Stwoh$ and $\Sthreeh$. \B

 Using the definition of anchors we immediately deduce the following: 
\begin{itemize}
\item  anchors associated with $\Szeroh$ are at the barycentres of all elements in $\M$;
\item anchors associated with $\Soneh$ are attached to the barycentres
  of \emph{internal} faces of $\M$ and the corresponding vector basis function is normal to the face;
\item anchors associated with $\Stwoh$ are attached to  
  \emph{internal} edges of $\M$ and the corresponding vector basis function is tangent to its corresponding edge; 
\item anchors associated with $\Sthreeh$ are attached to  
  \emph{internal}  vertices of $\M$. 
\end{itemize}

Clearly,  the positivity of the bases  induces an orientation of edges
and faces of the mesh $\M$. \B

 With the bases \eqref{eq:0-forms-basis}--\eqref{eq:3-forms-basis}, the discrete differential operators in
\eqref{eq:3D-parametric-diagram} are represented by incidence matrices
for the corresponding geometrical entities. If  $p \equiv
p_1 = p_2 = p_3$ is odd, then the operator $\hat \grad$ is represented by the
\emph{edge-vertex} incidence matrix of $\M$ and when $p\geq 2$ is
even, by  the \emph{face-element} incidence  matrix of $\M$.  
 We observe that, unlike in compatible finite elements, the matrices
representing the differential operators in the selected bases
\eqref{eq:0-forms-basis}--\eqref{eq:3-forms-basis} are
essentially independent of the degree. 

The fundamental consequence of the  observations above is stated in
the following proposition. 

\begin{proposition}
  \label{p:cochain-and-chain}
The following holds:
  \begin{itemize}
\item The spline complex \eqref{eq:3D-parametric-diagram} for odd degree $p$ is isomorphic to the cochain complex associated with the partition $\M$. 
 \item The spline complex \eqref{eq:3D-parametric-diagram} for even
   degree $p$ is isomorphic to the chain complex associated with the
   partition $\M$  \B  without  its boundary, that is, when only the
   interior geometrical entities (faces,
   edges and vertices) are taken    into account, as seen above. \B 
\end{itemize}
\end{proposition}
\B

As a matter of fact, this observation, together with the structure of
the matrix representation of differential operators, makes the geometry
of the spline complex for odd degree $p$ very similar, if not equal,
to the one of the finite element complex of lowest order. However  the
spline complex for $p\geq 1$ delivers an approximation which is far
superior than the one of low order finite element.  

For $p$ even  we have instead  a chain complex without explicitly
constructing the dual mesh, which has no analogue in the finite
element framework. 

Moreover, the use of
anchors and the structure of the mesh at the boundary  guarantee  that in both the chain and cochain complex the
boundary is treated in a simple and canonical way. In the case of
finite elements this is not case (see e.g. \cite{BuCh07}, \cite{GeWiCl04,ClThWe99}, \cite{MR1872728}
or \cite{MR2143847} ) and, moreover,  these features can hardly be obtained in
conjunction with high-order finite element techniques. 
Discretization methods based on the use of both chain and cochain complexes in the framework of isogeometric methods are very promising and object of on-going research.\B


We conclude this section with a remark on boundary conditions. 
 Consider the case when homogeneous boundary conditions are
 imposed on the whole boundary $\partial \hat {\Omega}$,  leading  to the definition of the discrete spaces  $\Szerooh := \Szeroh \cap
  \Honeor$, $\Soneoh := \Soneh \cap \Hocurlr$, $\Stwooh := \Stwoh \cap
  \Hodivr$ and $\Sthreeoh :=  \Sthreeh $. These spaces are constructed
  as usual, by removing the functions with non-null trace at the boundary, because univariate  B-spline functions are interpolatory at
  the boundary, as we have discussed in Remark
\ref{rem:interpolatory-B-splines}. The associated De Rham
  complex 
\begin{equation}\label{eq:3D-parametric-diagram-bc}
\begin{CD}
0 @>>>\Szerooh @>\hat \grad>> \Soneoh @>\hat \curl>> \Stwooh @>\hat \div>> \Sthreeoh@> \int>> \mathbb{R}
\end{CD}
\end{equation}
is exact, as easily follows by a variation of the argument of
Theorem~\ref{th:exact}.
 The same holds in more general cases, for example  when
the boundary conditions are imposed on a part of  $\partial \hat
{\Omega}$ formed by the union of some faces of the cube $\Omegaref$. 
Since boundary conditions do not represent a conceptual difficulty, in order to keep the presentation as clear as
possible often in our presentation we will not take them into the
framework.

\subsection{Push-forward to the single-patch physical  domain $\Omega$}
\label{sec:compatible -splines-phys-domain}

Following Section \ref{sec:spline-geometries}, we suppose
that the domain $\Omega$ is obtained from $\hat \Omega$ through a
spline or NURBS single-patch mapping $\bF$. Clearly, we need to choose the space for $\bF$. 
\begin{assumption}[Isogeometric mapping]\label{ass:F}
  We assume that $\bF$ is either a spline function in
  $[\Szeroh]^3$,
  or $\bF$ is  a NURBS function as in \eqref{eq:F}, with numerator in
$[\Szeroh]^3$ and weight denominator in   $\Szeroh$.
\end{assumption}
Assumption \ref{ass:F} is indeed very natural  in the context of
isogeometric methods: it means that the discrete fields are
constructed from the geometry  knot vectors and bases, possibly after
refinement.  

We denote by $\pM$ the image of $\M$ through the mapping $\bF$. 
$\pM$ is then a partition of the physical domain $\Omega$, similar to
the finite element mesh, even though  it  contains elements of
zero area due to knot  multiplicity.  \B
\B

The discrete spaces $\Xzeroh, \ldots ,
\Xthreeh $ on the physical domain $\Omega$ can be  defined from the
spaces \eqref{eq:discrete-forms-on-omegahat}  on the parametric
domain $\Omegaref$ by {\it push-forward}, that is, the inverse of the
transformations  defined in \eqref{def:iota3}, that commute with the
differential operators (as given by the diagrams \eqref{eq:diag1} and
\eqref{eq:diag2}):
\begin{equation} \label{eq:diag-discrete}
\begin{CD}
\mathbb{R} @>>> \Szeroh @>\hat \grad>> \Soneh @>\hat \curl>> \Stwoh @>\hat \div>> \Sthreeh @>>> 0 \\
@. @A\pb^0AA @A\pb^1AA @A\pb^2AA @A\pb^3AA \\
\mathbb{R} @>>> \Xzeroh @>\grad>> \Xoneh @>\curl>> \Xtwoh @>\div>> \Xthreeh @>>> 0, \\
\end{CD}
\end{equation}
 that is, the discrete spaces in the physical domain are defined as
\begin{equation}\label{eq:push-forwarded-discrete-diff-forms}
\begin{aligned}
\Xzeroh := \{ \uzero : \pb^0(\uzero) \in \Szeroh \} , \\
\Xoneh := \{ \uone : \pb^1(\uone) \in \Soneh \}, \\
\Xtwoh := \{ \utwo : \pb^2(\utwo) \in \Stwoh \}, \\
\Xthreeh := \{ \uthree : \pb^3(\uthree) \in \Sthreeh \} .
\end{aligned}
\end{equation}
We remark that the space $\Xoneh$, which is a discretization of $\Hcurl$, is defined through the curl conserving transformation $\pb^1$, and that the space $\Xtwoh$, which is a discretization of $\Hdiv$, is defined through the divergence conforming transformation $\pb^2$. These are equivalent to the curl and divergence preserving
transformations that are used to define edge and face elements,
respectively (see \cite[Sect. 3.9]{Monk}).

Thanks to the properties of the operators  \eqref{def:iota3} the
push-forwarded spaces $\Xzeroh, \ldots ,
\Xthreeh$  inherit the same fundamental properties of $\Szeroh, \ldots ,
\Sthreeh$, that we have discussed in the previous section: 
\begin{itemize}
\item they form an exact
De Rham complex without boundary conditions
\begin{equation}\label{eq:3D-physical-diagram}
\begin{CD}
\mathbb{R}  @>>>\Xzeroh @>  \grad>> \Xoneh @>  \curl>> \Xtwoh @>  \div>> \Xthreeh@>>> 0,
\end{CD}
\end{equation}
 or with boundary conditions
\begin{equation}\label{eq:3D-physical-diagram-bc}
\begin{CD}
0 @>>>\Xzerooh @>  \grad>> \Xoneoh @>  \curl>> \Xtwooh @>  \div>> \Xthreeoh@> \int>> \mathbb{R}.
\end{CD}
\end{equation}
\item the basis functions for $\Xzeroh, \ldots ,
\Xthreeh$ are defined by push-forward of the basis functions of $\Szeroh, \ldots, \Sthreeh$, similarly to
\eqref{eq:push-forwarded-discrete-diff-forms}, and are in  one-to one relation 
with  the images of the anchors through $\bF$. See Figure \ref{fig:pipe_ancore}.
\begin{figure}[h!]
\centering
\includegraphics[width=.4\textwidth]{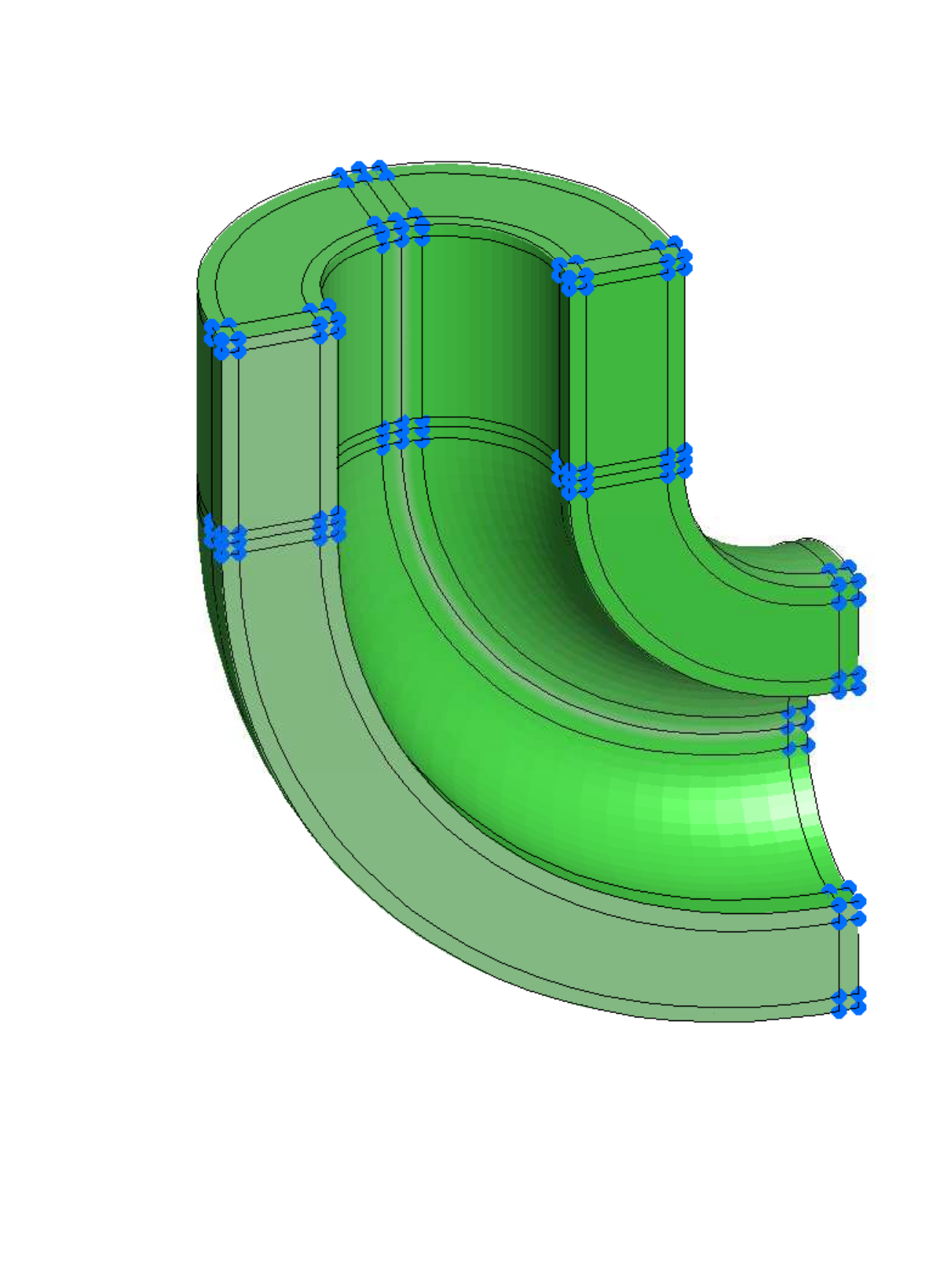}
\vspace*{-1cm}
\caption{We show the mesh $\pM$ and the image of the anchors related to the space $\Xzeroh$ on an example geometry.} \label{fig:pipe_ancore}
\end{figure}

\item since \eqref{eq:diag-discrete}, the
matrices associated with the differential operators $\grad$, $\curl$
and $\div$ on $\Omega$ are the same as the matrices of  $\hat \grad$, $\hat\curl$
and $\hat\div$ on $\Omegaref$, that is, incidence matrices of
the mesh $\pM$.
\item when $p$ is odd (even, respectively), the complex ($\Xzeroh,\ldots,\Xthreeh$) is isomorphic to the cochain (chain, respectively) complex associated to the partition $\pM$. 
\end{itemize}
 
 Finally, the discrete spaces $\Xzeroh, \ldots ,
\Xthreeh$ inherit from their pull-back $\Szeroh, \ldots ,
\Sthreeh $ optimal approximation properties, if the geometrical
mapping $\bF$ satisfies Assumption~\ref{ass:F} and its inverse is smooth enough (see \cite{BRSV11} for details).

\subsection{Control fields and  degrees-of-freedom interpretation}

\label{sec:dofs}
\B   In this section, we introduce the concept of \emph{control fields},
thanks to which we give an interpretation of the degrees-of-freedom of
the isogeometric fields defined in Section \ref{sec:compatible
  -splines-parametric-domain}--\ref{sec:compatible
  -splines-phys-domain}. The control fields are for the B-spline isogeometric
fields what  the control mesh is for the the B-spline geometry. \B 
 We recall that from the geometry control points  we define $\bF_C$ (see
\eqref{eq:Fh-Bspline}),  the piecewise trilinear
function on the  Greville
mesh $\M_G$. The  image of $\bF_C$  is the so-called
control domain $\Omega_C$. The
so-called control mesh  $\pM_C$ (which is a partition of
$\Omega_C$) is the image through $\bF_C$ of the Greville
mesh $\M_G$. 

As described in Section \ref{sec:spline-geometries},  the standard
way to manipulate a spline parametrization  $\bF$ is by moving its control
points, that is,  the vertices of the control mesh $\pM_C$. The
parametrization $\bF_C$ or, equivalently, the control mesh  $\pM_C$, carries the {degrees-of-freedom} for the
geometry.  The distance between the two parametrizations $\bF$ and
$\bF_C$  is at most $O(h^2)$, as in \eqref{eq:-control-mesh-convergence}.
\B  We now apply the same rationale  for the complex \B  $(\Xzeroh,..., \Xthreeh \B)$.
Let us first focus on scalar functions on the parametric domain $\hat
\Omega$, 
i.e., on the space $\Szero_h$. Given a spline
\begin{equation}\label{eq:zero-form-field} 
  \hat \phi(\bzeta) = \sum_{A \in \A_{p,p ,p}} c^A
 B^A_{p,p,p}(\bzeta), \qquad \text{with } \bzeta \in \hat \Omega,
\end{equation}
where $c^A$ are its control variables, we associate the
piecewise trilinear function defined on the mesh $\M_G$:
\begin{equation}\label{eq:trilinear-zero-form-field} 
  \hat \phi_C(\bzeta) = \sum_{A \in \A_{p,p ,p}} c^A
  \lambda^A(\bzeta), \qquad \text{with } \bzeta \in \hat \Omega,
\end{equation}
which  carries the same degrees-of-freedom for $\hat \phi$ and indeed is close to $\hat \phi$ (the distance between the two functions
is at most $O(h^2)$,  analogously to  \eqref{eq:-control-polygon-convergence}).  By this relation,
we can interpret the degrees-of-freedom \B  $c^A$ of  $\hat \phi$ as \B  the values
of $ \hat \phi_C $ at each Greville site in $\M_G$. 

It should be noted that, if the values of  these degrees-of-freedom are chosen wisely, 
splines deliver approximation error of order  $O(h^{p})$ in
the norm of $H^1(\hat \Omega)$, where $p$ is the degree of the spline,
while the corresponding trilinear function can only provide
approximation errors of order $O(h)$. \B

\smallskip
Let  now the geometry come into play. Using \eqref{eq:push-forwarded-discrete-diff-forms}, we set:
   \begin{equation}
\label{eq:anna1}
\phi \circ \bF = \hat \phi  \text{ and } \phi_C \circ \bF_C = \hat \phi_C. 
\end{equation}
 The degrees-of-freedom for $\phi$ are \B   the values 
of $ \phi_C $ at the vertices of $\pM_C$, that is, at the control
points. Or, we can say that  the field  $\phi_C$ determines, or
\emph{controls}, $\phi$,  and its degrees-of-freedom are the values of
$\phi_C$ at control points. \B In
 Figure \ref{fig:pipe-ctrl-net}, the location of control points (blue
 bullet) is shown on an example geometry.  The
 field $\phi_C$  plays the role of \emph{control field} for $\phi$.
As for the parametric space, there are wise choices of the degrees-of-freedom
which ensure an approximation error of order  $O(h^{p})$ in
the norm of $H^1(\Omega)$, while the corresponding trilinear function can only provide
approximation errors of order $O(h)$.

The same reasoning can be applied to the whole complex
$(\Xzeroh,..., \Xthreeh \B)$,  defined in  Section
\ref{sec:compatible -splines-parametric-domain} and
\ref{sec:compatible -splines-phys-domain},  from degrees  $ p_\ell$ and
knot vectors $ \Xi_\ell$.  \B  Indeed, we  introduce the
\emph{control complex} $(\Zzeroh,..., \Zthreeh \B)$ which is
obtained, still following Section \ref{sec:compatible -splines-parametric-domain} and
\ref{sec:compatible -splines-phys-domain},  with  the following
choices for \B  $\hat \Zzeroh $: \B 

\begin{itemize}
\item degrees in all directions equal to $1$;
\item the knot vector in the $\ell$ direction is the ordered
  collection of points $\{ \gamma_\ell^A  \ :\ A\in \A_{p_\ell}
  (\Xi_\ell) \}$, $\ell=1,2,3$,  along with repeated 0 and 1 to make the knot vectors open,
\end{itemize}
and {\Rd replacing the geometric mapping $\bF$ with $\bF_C$ in the
pullbacks} \eqref{def:iota3}. 
The complex $(\Zzeroh,...,  \Zthreeh \B)$ corresponds to the low
order finite element complex defined on the control mesh $\pM_C$ and
it is immediate to see that if $\phi$ in \eqref{eq:anna1} belongs to
$\Xzeroh$, then the corresponding $\phi_C$ belongs to $\Zzeroh$.
We denote by $I^0_h: \Xzeroh \to \Zzeroh$ the operator which
associates  $\phi$ to $\phi_C$, \B and in an analogous  way,  we
define the operators  $I^j_h : X^j_h \to Z^j_h$,
$j=0,\ldots,3$. These operators \B are represented by  identity matrices when
the spaces are endowed with the bases described in Section
\ref{sec:compatible -splines-parametric-domain}.  
It is not difficult to see that, in view of the structure of the matrices associated to differential operators,  the following diagram commutes: 
\begin{equation}
\label{eq:cd-XZ}
\begin{CD}
\mathbb{R} @>>> \Xzeroh @> \grad>> \Xoneh @>\curl>> \Xtwoh  @>\div>> \Xthreeh @>>> 0 \\
@. @V{I^0_h}VV @V{I^1_h}VV @V{I^2_h}VV @V{I^3_h}VV \\
\mathbb{R} @>>> \Zzeroh @>\grad>> \Zoneh @>\curl>> \Ztwoh @>\div>> \Zthreeh @>>> 0. \\
\end{CD}
\end{equation}

 Let us comment about the meaning of the diagram
 \eqref{eq:cd-XZ}. First of all,  it says that the geometric structure
 of the spline complex $(\Xzeroh,..., \Xthreeh \B)$  is the one of
 the  low order finite element complex $(\Zzeroh,..., \Zthreeh \B)$
 on the control mesh.  The discrete fields in  $X^j_h$ can be
 associated to \emph{control fields} in  $Z^j_h$, through the operator
 $I^j_h$, as we have discussed for $j=0$ above. \B  For example, we can
 say that there is a N\'ed\'elec field $\bu_C$ which controls $\bu$
 and the degrees-of-freedom are, in this case,  its circulation on the
 edges of the control mesh $\pM_C$.   Moreover,
 following a reasoning similar to the one in Section
 \ref{sec:spline-geometries},  the operators $I^j_h$ are point-wise
 converging to the identity when $h$ goes to zero.
 We stress again that the order of approximation of the complex 
$(\Xzeroh,..., \Xthreeh \B)$ is $O(h^p)$ while the control complex
$(\Zzeroh,..., \Zthreeh \B)$ only exhibits  first order convergence {\Rd in the norm of $X^i$}.

\begin{figure}[h!]
\centering
\includegraphics[width=.4\textwidth]{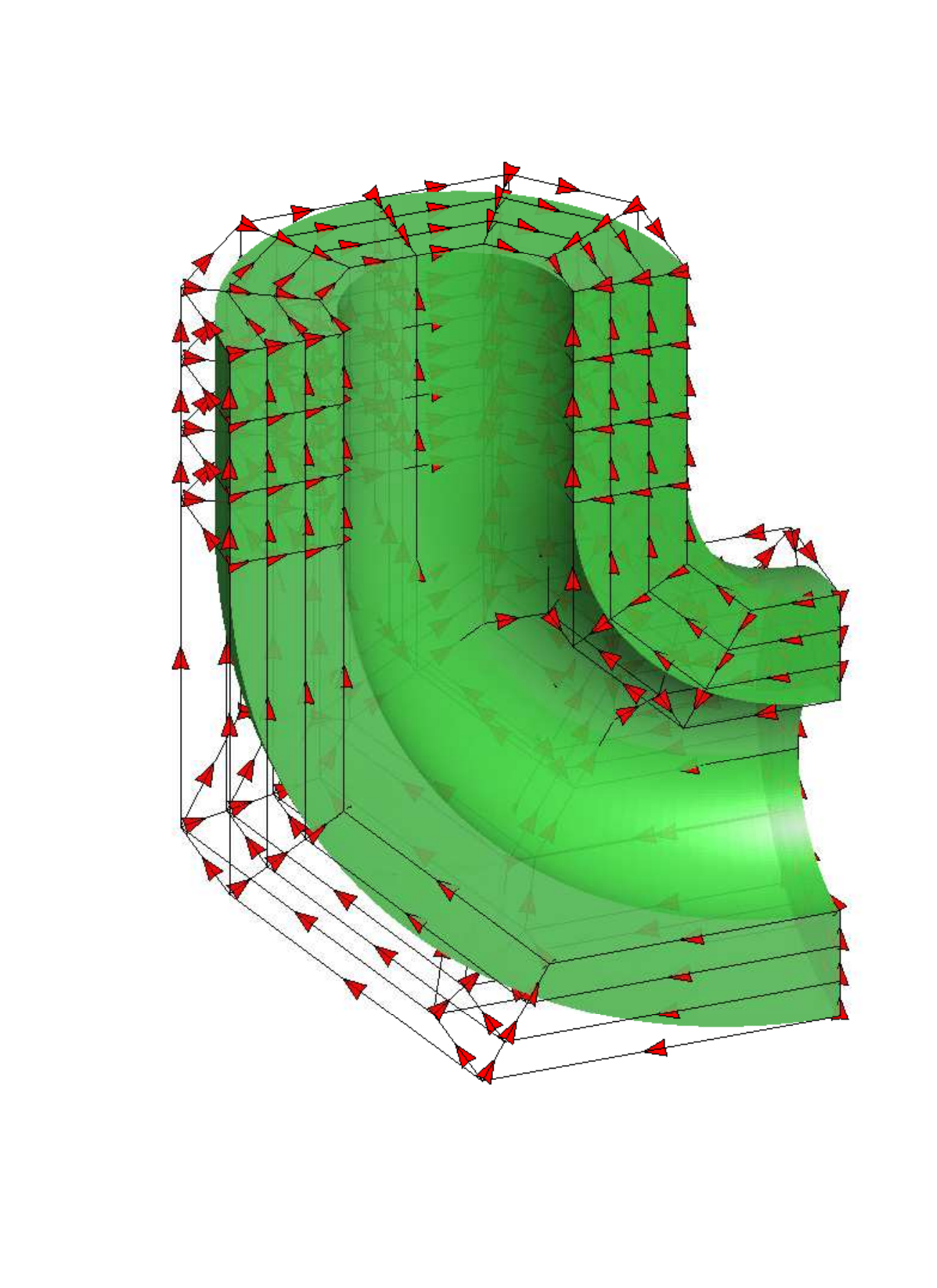}
\vspace*{-1cm}
\caption{Representation of the degrees-of-freedom location for the space $\Xoneh$, on the green geometry for degrees $p_1=p_2=p_3=3$.} \label{fig:pipe-ned}
\end{figure}

 Finally, it should be noted that, as it is well known,  the complex $(\Zzeroh,..,\Zthreeh)$ is always isomorphic to the cochain complex of  the partition $\pM_C$, while for the complex  $(\Xzeroh,..,\Xthreeh)$ Proposition \ref{p:cochain-and-chain} holds. This is in accordance with the fact that, when $p_\ell$ are all even, the control mesh $\pM_C$ can be interpreted as a partition dual to $\pM$, in the sense that the chain of $\pM$ is isomorphic to the cochain complex of  $\pM_C$.

\subsection{Push-forward to the multi-patch  physical  domain
  $\Omega$}
\label{sec:multi-patch-complex-continuity} 
\Bd In this section we construct the spline complex on a multi-patch geometry by addressing the questions of 
how conformity is imposed at the interfaces between patches.\B

We consider now a multi-patch domain $\Omega$ which is parametrized
from a reference patch $\hat \Omega$ through the 
spline or NURBS mappings $\bF_k$, $k=1,\ldots,N$, as in
\eqref{eq:F-multi-patch}.  Each patch is endowed with a (possibly
different) spline space and therefore for each  $k=1,\ldots,N$ we
can define discrete spaces $[\Szeroh]_k, \ldots, [\Sthreeh]_k$ such
that a De Rham complex \eqref{eq:3D-parametric-diagram} holds.
Assuming each  $\bF_k$ verifies Assumption \ref{ass:F}, then, as shown in
Section \ref{sec:compatible -splines-phys-domain} we  push-forward
patch-by-patch the discrete spaces $[\Szeroh]_k, \ldots, [\Sthreeh]_k$
and obtain, on each $ \Omega_k = \bF_k (\hat \Omega) $, the discrete
spaces $[\Xzeroh]_k, \ldots, [\Xthreeh]_k$ that fulfill the De Rham
complex \eqref{sec:compatible -splines-phys-domain} on each patch. 

Then, the last and main step is
to assemble the spaces $X_h^j  \subset \bigoplus_{k=1,\ldots,N}
[X_h^j]_k$, and add the relevant continuity conditions  at the
inter-patches boundaries: trace continuity for $\Xzeroh$, tangential  trace
continuity for $\Xoneh$, normal trace continuity for $\Xtwoh$, no continuity for $\Xthreeh$.
For this purpose, we introduce a \emph{conformity} condition as,
e.g., in \cite{kleissieti}. This condition guarantees that the
geometry parametrizations of the patches are equivalent at the
patch interfaces and, since Assumption \ref{ass:F}, it can be stated on
the spaces $[\Szeroh]_k$. 

\begin{assumption}[Geometrical conformity]
  \label{ass:IETI}
On each non-empty patch interface $\Gamma= \partial \Omega_{k}
\cap \partial \Omega_{k'} $, with $k \neq k'$,  the spaces
$ [\Xzeroh]_{k} | _{\Gamma}$ and $ [\Xzeroh]_{k'} | _{\Gamma}$
coincide, as the corresponding bases do.
\end{assumption}

This means that the meshes $\pM_k$ and
$\pM_{k'}$ match on $\Gamma$, and therefore \B  
\begin{displaymath}
  \pM = \bigcup_{k=1,\ldots,N} \pM_k
\end{displaymath} \B 
is  a locally structured but globally unstructured mesh $\pM$  on
$\Omega$. In a similar way, the patch control meshes $[\pM_C]_k$ match
conformally and \B 
\begin{displaymath}
  \pM_C = \bigcup_{k=1,\ldots,N} [\pM_C]_k
\end{displaymath} \B
is a locally structured but globally unstructured mesh $\pM_C$ of hexahedra
on $\Omega_C$, the union of the patch control domains
$[\Omega_C]_k$.   

Assumption \ref{ass:IETI}  corresponds to the \emph{full-matching}
conditions of  \cite{kleissieti}, to which we refer for further
details. 

Having conformity we can implement the continuity conditions
easily. Indeed, due to the definitions in Sections~\ref{sec:compatible
  -splines-parametric-domain}--\ref{sec:dofs}, the needed continuity  holds
for $(\Xzeroh,..., \Xthreeh)$ if and only if it holds for $(\Zzeroh,..., \Zthreeh
)$ on the global mesh $\pM_C$. Continuity for $\Zzeroh, \Zoneh$ and
$\Ztwoh$ is imposed by merging the coincident degrees-of-freedom at
the interfaces, which in the case of $\Zoneh$ and $\Ztwoh$ also
requires to take into account the orientation; see Figure~\ref{multi-patch}. This is however the
same as in finite elements (indeed, the control fields are classical
finite elements). This merging automatically gives the degrees-of-freedom for fields in
$(\Xzeroh, ..., \Xthreeh)$.

\begin{figure}[h]
\begin{subfigure}[Control variables of the two patches before merging.]
{\includegraphics[width=0.48\textwidth]{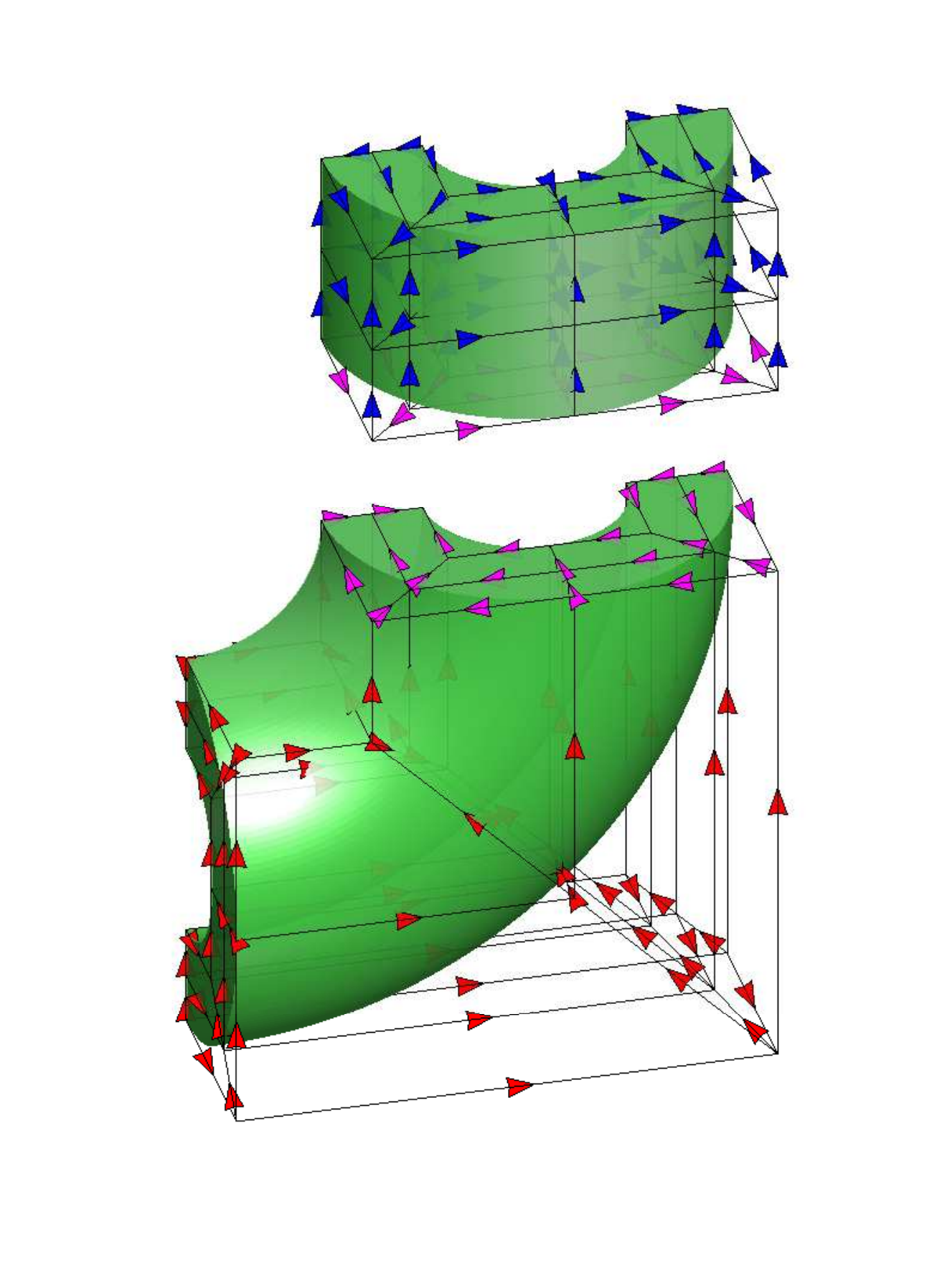}} \label{fig:multipatch1}
\end{subfigure}
\begin{subfigure}[Control variables of the merged patch.]
{\includegraphics[width=0.41\textwidth]{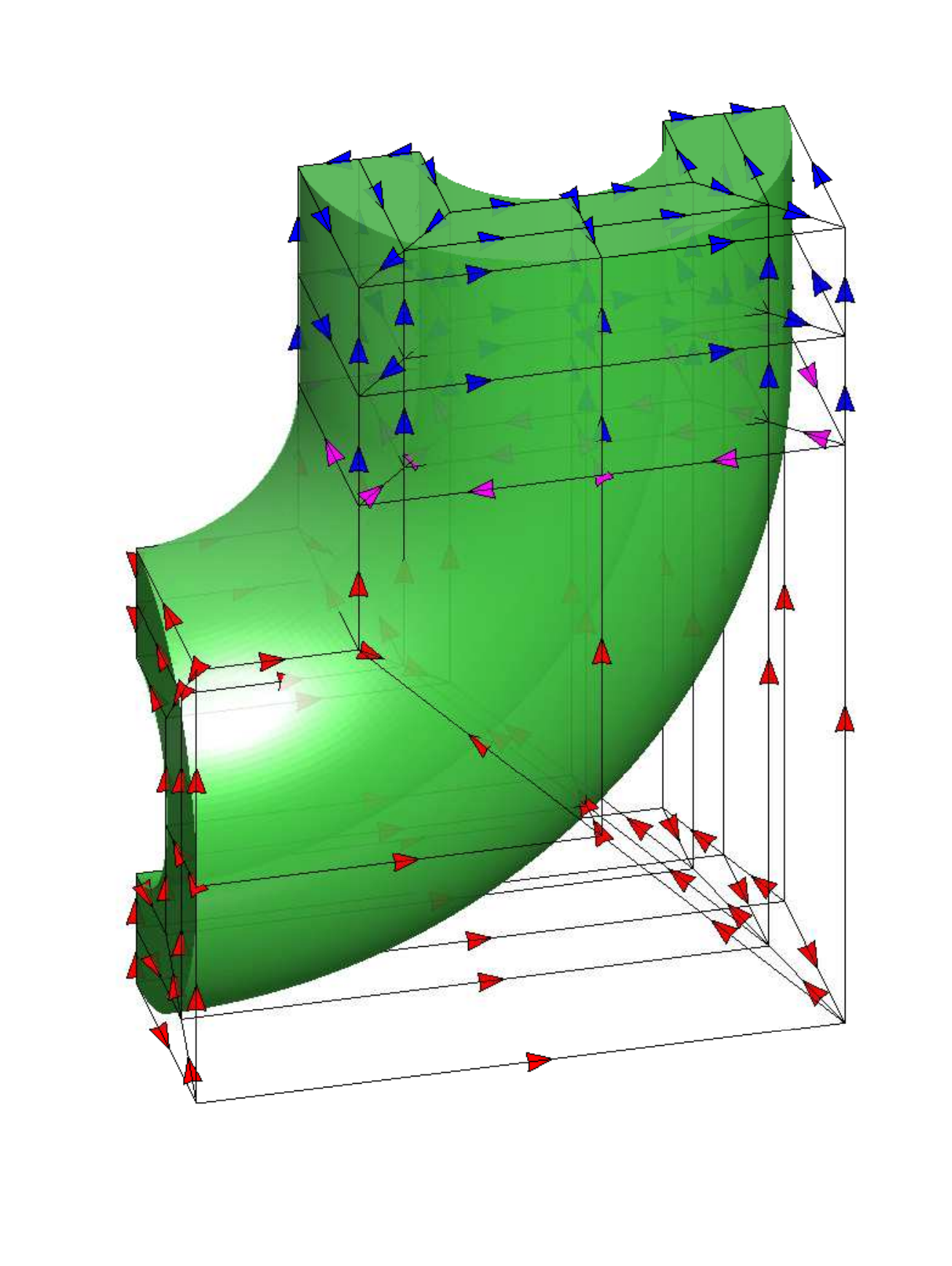}} \label{fig:multipatch2}
\end{subfigure}
\caption{Implementing continuity for $ \Xoneh$ on a two-patch domain. The orientation of the edges at the interface is chosen as that of the {\Rd lower} patch.}
\label{multi-patch}
\end{figure}

\section{Beyond the tensor product structure: T-splines}
\label{sec:TT}

In this section, we generalize the definition of tensor-product
B-splines to  T-splines
\cite{Sederberg_Zheng,Sederberg_Cardon,Bazilervs_Calo_Cottrell_Evans}. The
theory of T-splines is well developed in two dimensions (see the very
recent papers \cite{LZSHS12, BBCS12, BBSV12, Li_Scott})
while it is still incomplete in three dimensions (though some recent
important advances have been recently proposed  in \cite{Zhang2012}). For this reason,  we only present in Sections
\ref{sect:T-mesh} and~\ref{sect:AS}, T-splines in two
dimensions and construct, in Section \ref{sect:T-compatible-parametric}, a
discrete T-spline based complex. The extension to three
dimensions is given in Section~\ref{sect:3D-T-splines} by tensor-product of the
two-dimensional T-spline spaces with B-spline one-dimensional spaces.
\B
\subsection{T-mesh}\label{sect:T-mesh}
Let $n_\ell \in \N$ and the degree $p_\ell \in \N$, and let $\Xi_\ell
= \{\xi_{\ell,1}, \ldots, \xi_{\ell,n_\ell+p_\ell+1}\}$ be a
$p_\ell$-open knot vector for $\ell = 1,2$. A  T-mesh is  a
rectangular tiling of the unit square $[0,1]^2$, such that all
\emph{vertices}  belong to $\Xi_1 \times \Xi_2$. A T-mesh may contain
 interior
vertices that connect only three edges,  called \emph{T-junctions},
that break the tensor product structure of the mesh (see
Figure~\ref{fig:Tmesh-p3}).  We will say that a T-junction is horizontal (respectively, vertical) if 
the missing edge is horizontal (resp. vertical). By an abuse of notation,
we still denote a T-mesh by $\M$. 

As before,  we represent the  knot multiplicities by repeated lines
close to each other, with now  the line multiplicity possibly
varying along lines  (see
\cite[Section~4.3]{Bazilervs_Calo_Cottrell_Evans}). The only exception are the boundary lines, that maintain the same multiplicity all along the line. As in B-spline
meshes, the vertical (resp. horizontal) lines at the boundaries have
multiplicity $\lfloor p_1/2 \rfloor + 1$ (resp. $\lfloor p_2/2 \rfloor + 1$).

\subsection{Analysis suitable T-meshes and  T-splines.}
\label{sect:AS}
We define, for a horizontal (resp. vertical) T-junction $T$, the
$k$-bay face-extension as the horizontal (resp. vertical) closed
segment that  extends from $T$  in the direction of the missing
edge until it intersects $k$ lines of the mesh $\M$. The $k$-bay
edge-extension is defined analogously extending the segment in the
opposite direction.

Following \cite{BBSV12}, we define the extension of a horizontal
(resp. vertical) T-junction $T$  the union of the
$\lfloor(p_1+1)/2\rfloor$-bay face-extension, and the
$\lfloor(p_1-1)/2\rfloor$-bay edge-extension (resp.  the union of the $\lfloor(p_2+1)/2\rfloor$-bay face-extension, and the $\lfloor(p_2-1)/2\rfloor$-bay edge-extension). More precisely, if $p_\ell$ is odd we extend $(p_\ell+1)/2$ bays in the direction of the missing edge, and $(p_\ell-1)/2$ bays in the opposite direction; if $p_\ell$ is even we extend $p_\ell/2$ bays in both directions. An example is given in Figure~\ref{fig:extensions}.

\begin{figure}[h!]
\centering
\includegraphics[width=0.4\linewidth]{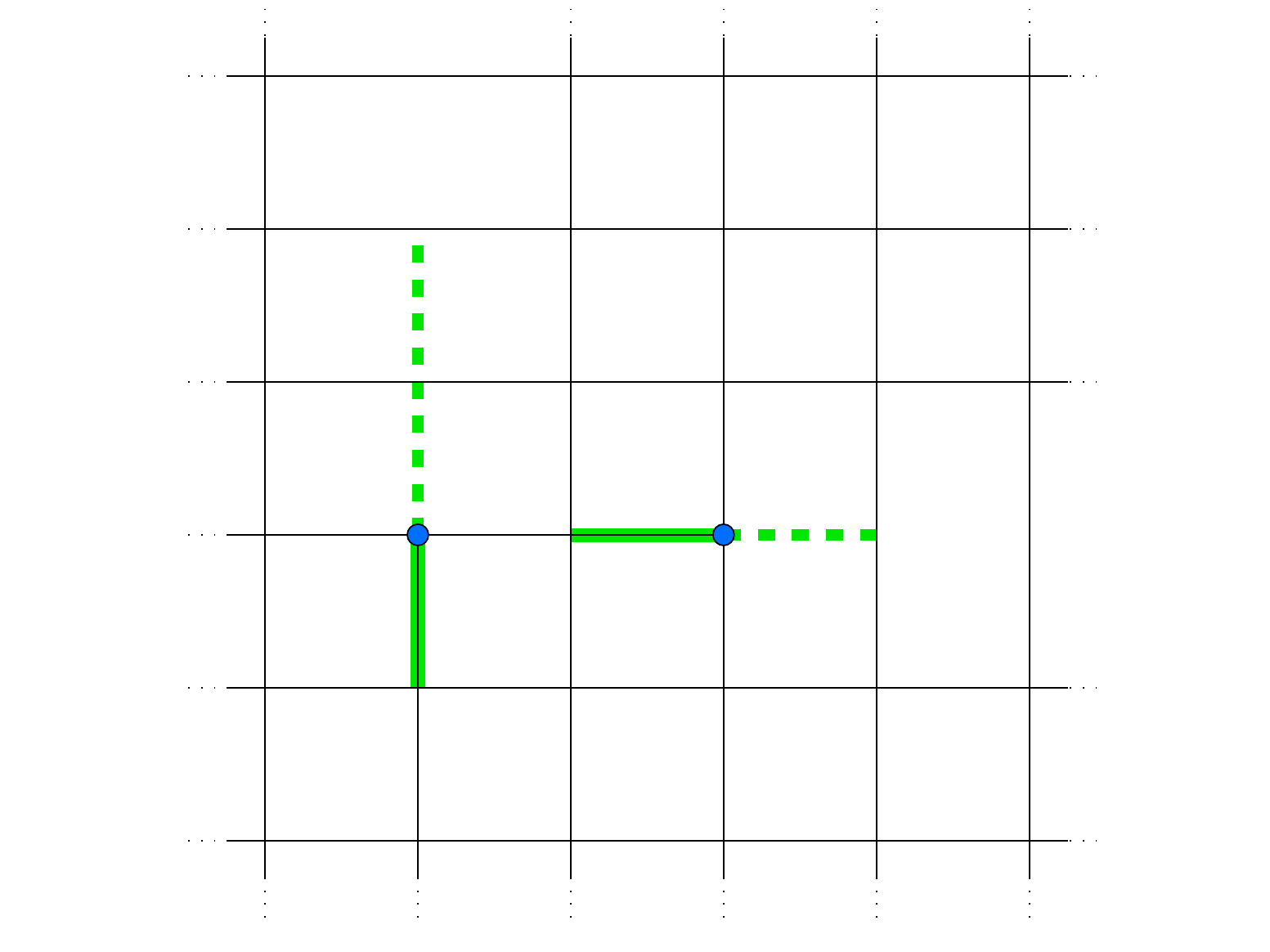}
\caption{Extensions for degree $p_1 = 2$ (horizontal) and $p_2 = 3$ (vertical). Dashed lines represent the face extensions.} \label{fig:extensions}
 \end{figure}

\begin{definition}\label{def:analysis-suitable}
A T-mesh $\M$ is analysis suitable for degrees $p_1$ and $p_2$,
denoted  $\M \in \AS_{p_1,p_2}$, if vertical extensions and horizontal extensions do not intersect.
\end{definition}

Analysis suitable T-meshes were first identified in \cite{LZSHS12} in
the bi-cubic case, and generalized to arbitrary degree in
\cite{BBSV12}.  Despite their very geometric definition, analysis suitable T-meshes and T-splines enjoy fundamental properties  which make their use in isogeometric analysis really promising. Some of these properties will be discussed in what follows.

As in the case of B-splines, 
anchors are inferred from the T-mesh $\M$ and their position depends upon the parity of $p_1$ and $p_2$. 
For example, if both $p_1$ and $p_2$ are odd, anchors are at the vertices of the $\M$, if they are even, anchors are at the barycenters of elements and so on  (see \cite{BBSV12}). We will denote the set of
anchors by $\A_{p_1,p_2}(\M)$, or simply $\A_{p_1,p_2}$. 

T-spline basis functions are constructed as B-splines associated to the anchors
$\A_{p_1,p_2}(\M)$, and defined from two \emph{local knot vectors}, $\Xi_1^A = \{\xi_{1,i_1},
\ldots, \xi_{1,i_{p_1+2}} \} \subset \Xi_1$ and $\Xi_2^A = \{\xi_{2,i_1},
\ldots, \xi_{2,i_{p_2+2}} \} \subset \Xi_2$. To construct the horizontal knot vector $\Xi_1^A$ we trace the horizontal line through $A$ and select its intersections with vertical lines of $\M$, depending on the degree $p_1$: if $p_1$ is even we choose the first $(p_1+2)/2$ intersections to the left of $A$, and the first $(p_1+2)/2$ to the right; if $p_1$ is odd we first select the coordinate of the anchor $A$, and then the first $(p_1+1)/2$ intersections to the left and to the right of $A$. In the case that we arrive at the boundary, we add the value 0 or 1 as many times as needed to complete the $p_1+2$ entries of $\Xi_1^A$. The construction of $\Xi_2^A$ is analogous, and depends on $p_2$. Examples are shown in Figure~\ref{fig:local_knot_vector} for $p_1 = 2$ and $p_2 = 3$. For more details we refer to \cite{Bazilervs_Calo_Cottrell_Evans}.


\begin{figure}
\centering
\includegraphics[scale=0.4]{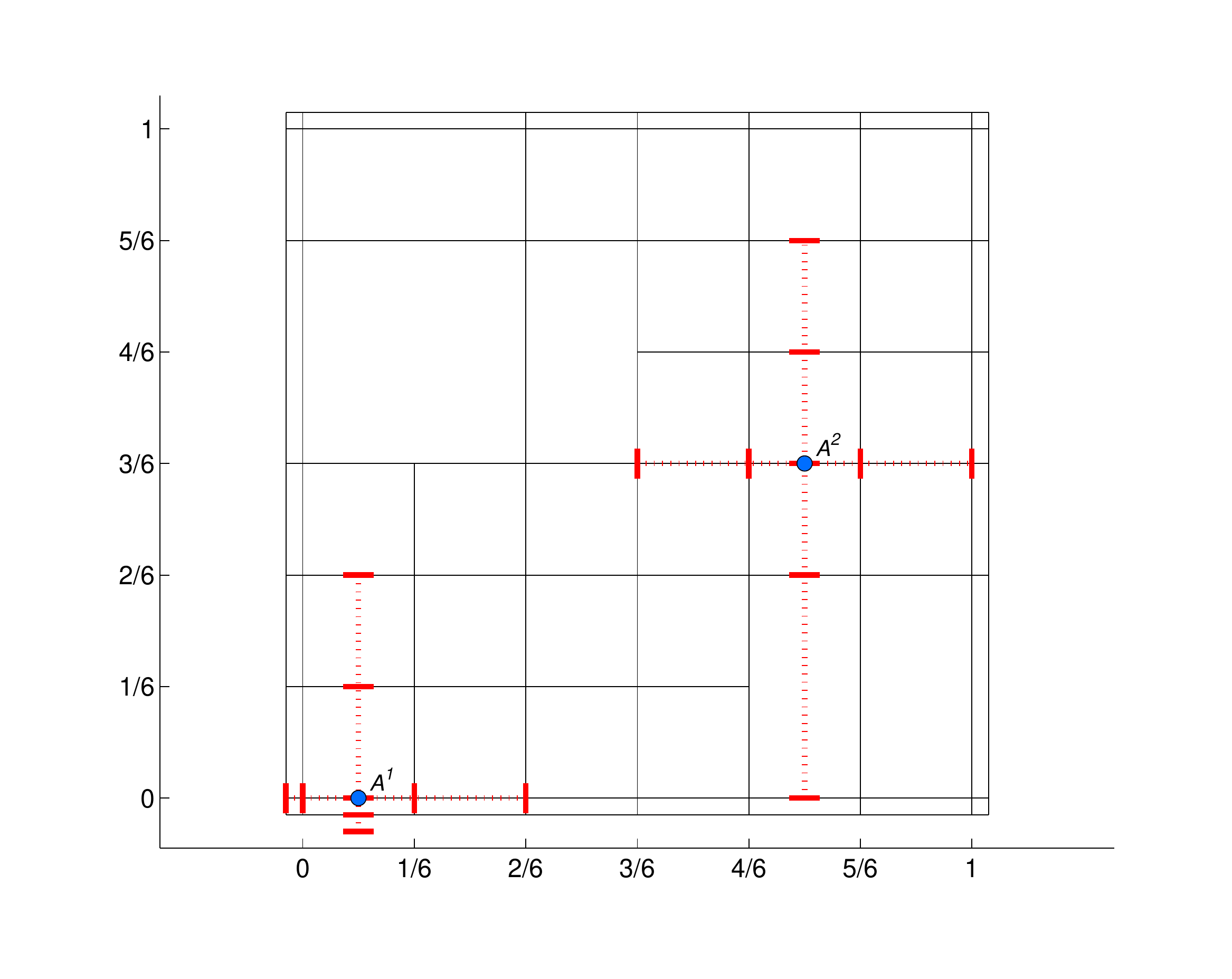}
\caption{\Bd Representation \B of the local knot vectors for degrees $p_1 = 2$ and $p_2 = 3$. The local knot vectors are $\Xi_1^{A^1} = \{0, 0, 1/6, 2/6\}$, $\Xi_2^{A^1} = \{0, 0, 0, 1/6, 2/6\}$, and $\Xi_1^{A^2} = \{3/6, 4/6, 5/6, 1\}$, $\Xi_2^{A^2} = \{0, 2/6, 3/6, 4/6, 5/6\}$.} \label{fig:local_knot_vector}
\end{figure}

The T-spline function associated to the anchor $A$ is denoted as:
\begin{equation}\label{eq:T-spline-function}
B^A_{p_1,p_2}(\bzeta) = N[\Xi_1^A] (\zeta_1) N[\Xi_2^A] (\zeta_2),
\quad \bzeta = (\zeta_1, \zeta_2) \in (0,1)^2, 
\end{equation}
they are linearly independent (see \cite{BBSV12}) and by definition span the
T-spline space $T_{p_1,p_2}= T_{p_1,p_2}(\M)$:
\begin{equation}\label{Tsplinespace}
 T_{p_1,p_2}(\M) := \mbox{span} \{B_{p_1,p_2}^{A} \ : \ A \in \A_{p_1,p_2}(\M) \} .
\end{equation}

Definition \ref{def:analysis-suitable} guarantees fundamental
properties of the T-spline space \eqref{Tsplinespace}.  In
\cite{BBCS12,BBSV12} it is defined a dual basis for
the T-spline functions constructed from an analysis suitable T-mesh,
thus proving the linear independence of \eqref{eq:T-spline-function} (see also \cite{LZSHS12}) and good approximation properties
for the  space \eqref{Tsplinespace}.  We also remark that the construction of the
local knot vectors described above is analogous to the one in 
  \cite{BBSV12}  for analysis suitable T-meshes. 

Finally, we define the \emph{extended T-mesh} of $\M$, and denote it
by $\M_\ext$, as the T-mesh obtained from $\M$ by adding all the
T-junction extensions.
 The extended T-mesh, sometimes also called B\'ezier mesh, is the
 minimal mesh such that the functions \eqref{eq:T-spline-function} restricted to the
 non-empty elements are bivariate polynomials of degree
 $(p_1,p_2)$. The importance of the extended mesh for implementation
 \cite{Scott2011126}, local refinement \cite{Li_Scott} and approximation
 properties of T-splines \cite{BBSV12} is already known.  In particular,  for the implementation, 
and in order to ensure accuracy, integration has to be performed on the elements of $\M_\ext$ and 
this means that the data structure is constructed based on $\M_\ext$. 
 
Finally, a key result which is useful in the construction of compatible T-spline discretizations,  is the characterization stated in the following
proposition.
\begin{proposition}\label{prop:characterization}
  Given an analysis suitable T-mesh $\M$, if furthermore no T-junction
  extensions of any kind intersect each other or intersect mesh
  lines with multiplicity greater than one,  then the T-spline  space \eqref{Tsplinespace}
  is the space of all piece-wise bivariate polynomials of degree
 $(p_1,p_2)$ on   $\M_\ext$ with the same continuity of the T-spline
 functions \eqref{eq:T-spline-function} at the mesh lines. 
\end{proposition}
\begin{proof}
  The case $p_1=p_2=3$ has been covered
  in \cite{Li_Scott}, while the
  general case is a work in progress by A. Bressan in
  \cite{bressan-in-preparation}. 
Related works are also \cite{Mourr10}, and \cite{LiCh11} which show the 
mathematical complexity of the problem.
 The condition that the extensions do not intersect lines with
multiplicity greater than one can be removed at the price of a more complex
statement, which we do not consider here for the sake of simplicity. \B
\end{proof}

\subsection{Two-dimensional De Rham complex with T-splines  on the parametric domain $(0,1)^2$} \label{sect:T-compatible-parametric}
The aim of this section is to introduce a two-dimensional  T-spline
based De Rham complex, thus generalizing the tensor-product
construction  of  \Bd section~\ref{sec:spline-complex}. \B Throughout this
Section we will assume, for the sake of simplicity, $p_1 = p_2 =
p$. The results are also valid in the general case, but the proofs
become more intricate.

As for B-splines, T-splines spaces are
constructed by a suitable selection  of the polinomial  degree in the two
directions and by a suitable design of the mesh, that is, the knot vectors. The main
difference now is that we need to modify the mesh $\M$, depending of the
form degree, not only at the boundary but also around T-junctions.

Let $p \in \N$, let $\Xi_1,\Xi_2$ be two $p$-open knot
vectors, and let $\M \in \AS_{p,p}$ be a T-mesh with knot repetitions,
as defined in Section~\ref{sect:T-mesh}. 
The starting mesh is $\Mzero \equiv \M$, on which we define the space of scalar fields:
\begin{equation}
    \label{eq:T-splines-0-forms-on-parametric}
 \hat Y^0_h := T_{p,p}(\Mzero).
\end{equation}

 The \B  T-splines  vector fields \B  are defined on the two T-meshes $\Moneh$ and $\Monev$. 
 If $p$ is odd, $\Moneh$ is obtained from $\M$ by adding the first-bay
  face-extension of all  horizontal T-junctions. If $p$ is even,
 $\Moneh$  is  equal to $\M$ everywhere but on the boundary where, due to the definition of $\M$, and recalling Section~\ref{sect:knot_vector},  the first and the last
 vertical columns of elements of $\M$ are removed. We define analogously $\Monev$,
 reasoning in the vertical direction: if $p$ is odd  $\Monev$  is
 defined by adding the first-bay face-extension of all the vertical
 T-junctions, and if $p$ is even, it is defined by removing the first
 and last horizontal rows of elements of $\M$. Then the \Rd vector
 fields \B   and the {\Rd rotated} \Rd vector
 fields \B  are defined as

 \begin{equation}
    \label{eq:T-splines-1-forms-on-parametric}
 \hat Y^1_h := T_{p-1,p}(\Moneh) \times  T_{p,p-1}(\Monev).
\end{equation}

\begin{equation}
    \label{eq:T-splines-1*-forms-on-parametric}
 \hat Y^{1*}_h := T_{p,p-1}(\Monev) \times  T_{p-1,p}(\Moneh).
\end{equation}

Finally,  the \B  last space $\hat Y^{2}_h$ \B is  defined on the T-meshes $\Mtwo$: if $p$ is
odd, $\Mtwo$ is obtained from $\M$ by adding all the first-bay
face-extensions (horizontal and vertical), and if $p$ is even it is
defined by removing the first and last rows and columns of elements in
$\M$. 

Then 
\begin{equation}
    \label{eq:T-splines-2-forms-on-parametric}
 \hat Y^{2}_h := T_{p-1,p-1}(\Mtwo).
\end{equation}

An example of the sequence of meshes is shown in
Figure~\ref{fig:Tmesh-p3} for $p=3$, and in Figure~\ref{fig:Tmesh-p2}
for $p=2$. We notice that whenever $\Mzero = \M$ is a tensor product
mesh, the construction is equivalent to the one presented in
Section~\ref{sec:compatible -splines-parametric-domain} for
B-splines. Indeed, for odd $p$ the four meshes are equal to $\M$,
because there are no T-junctions, and for even $p$ they only differ in
the number of line repetitions on the boundary.

\begin{figure}[h!]
\begin{center}
\begin{subfigure}[Mesh $\Mzero$]{\includegraphics[width=.47\textwidth]{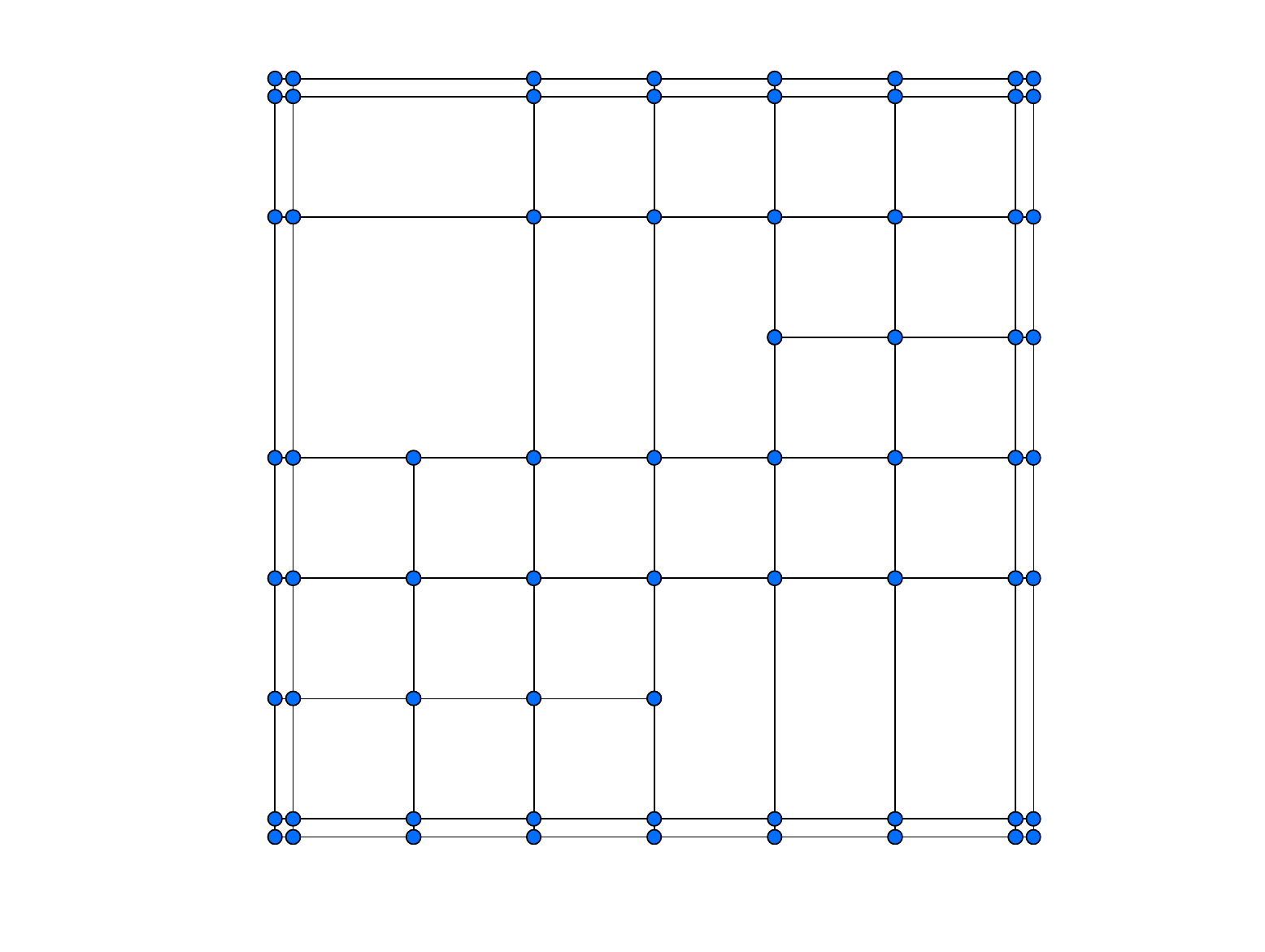}}
\end{subfigure}
\end{center}
\begin{subfigure}[Mesh $\Moneh$]{\includegraphics[width=.47\textwidth]{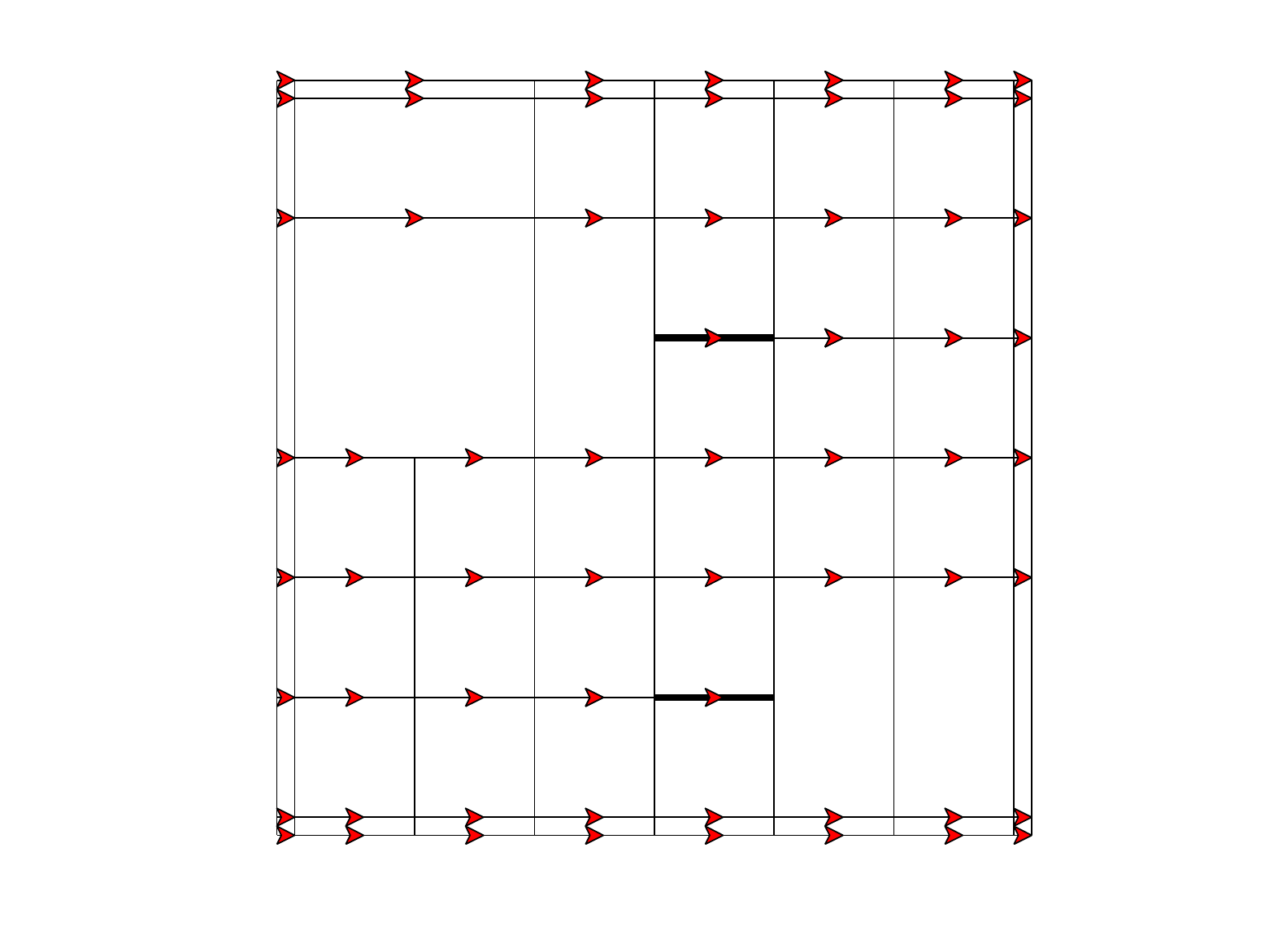}}
\end{subfigure}
\begin{subfigure}[Mesh $\Monev$]{\includegraphics[width=.47\textwidth]{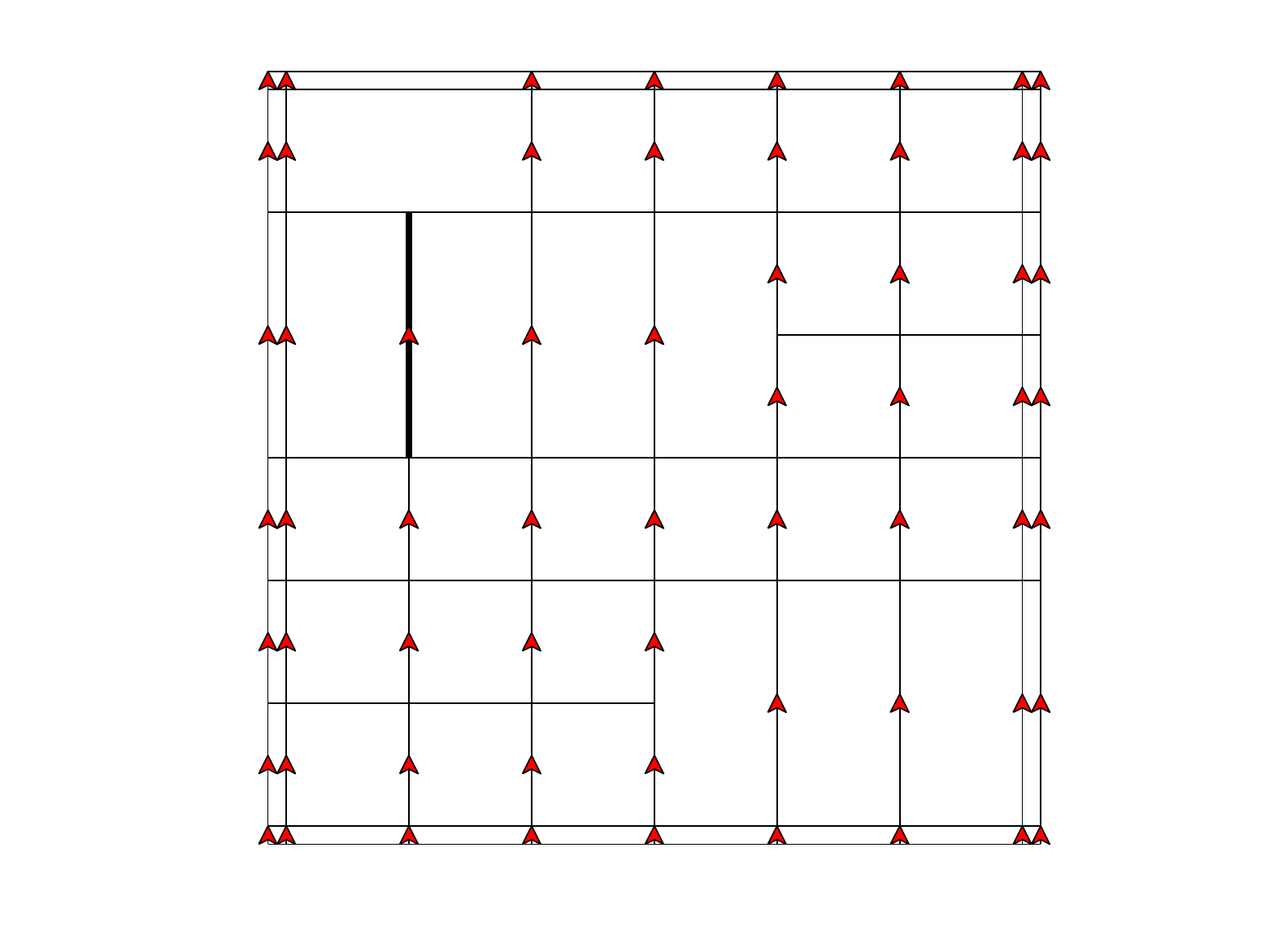}}
\end{subfigure}
\begin{center}
\begin{subfigure}[Mesh $\Mtwo$]{\includegraphics[width=.47\textwidth]{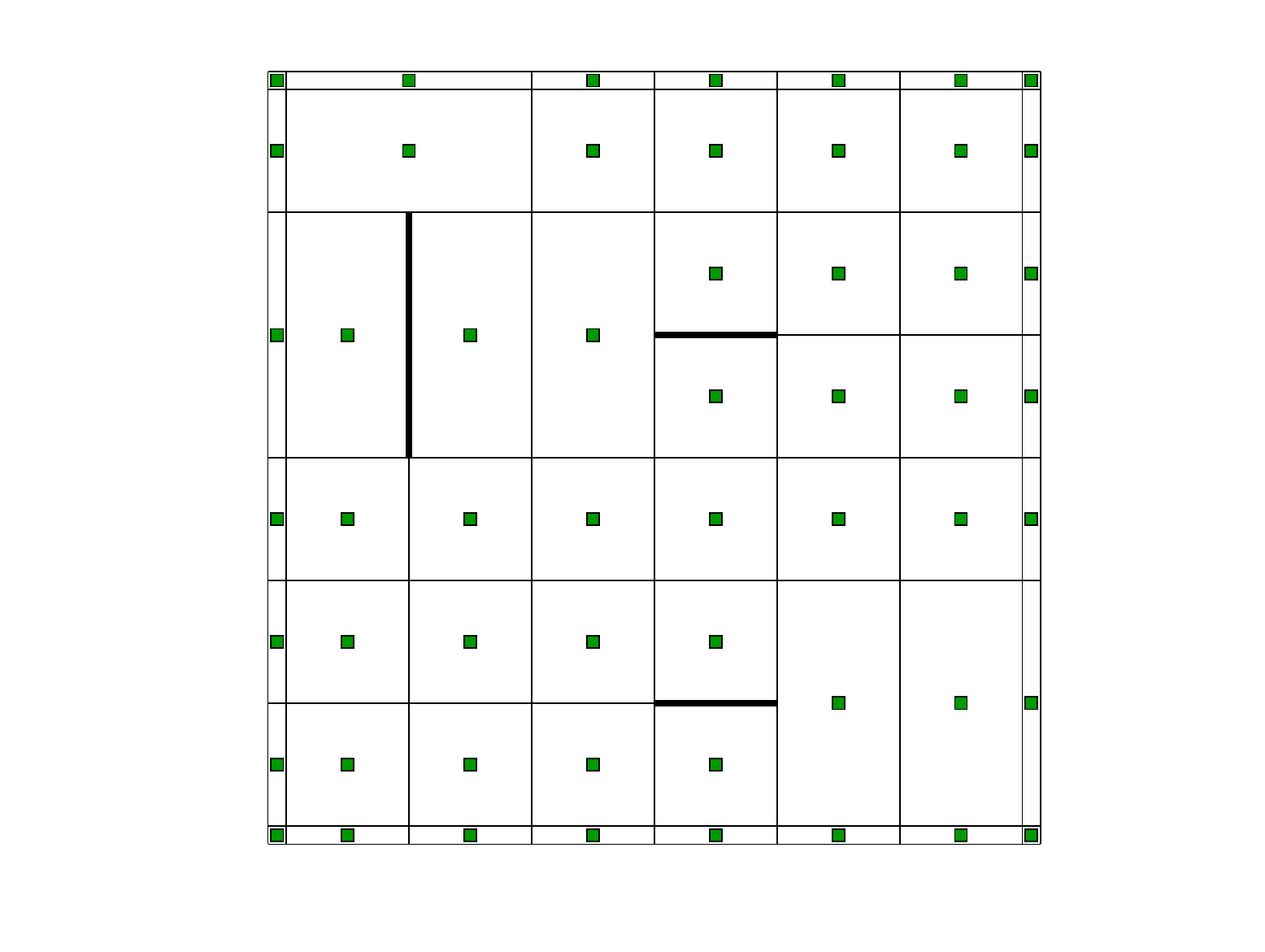}}
\end{subfigure}
\end{center}
\caption{Sequence of meshes for the spline complex, with their respective anchors, for $p=3$.} \label{fig:Tmesh-p3}
\end{figure}

\begin{figure}[h!]
\begin{center}
\begin{subfigure}[Mesh $\Mzero$]{\includegraphics[width=.47\textwidth]{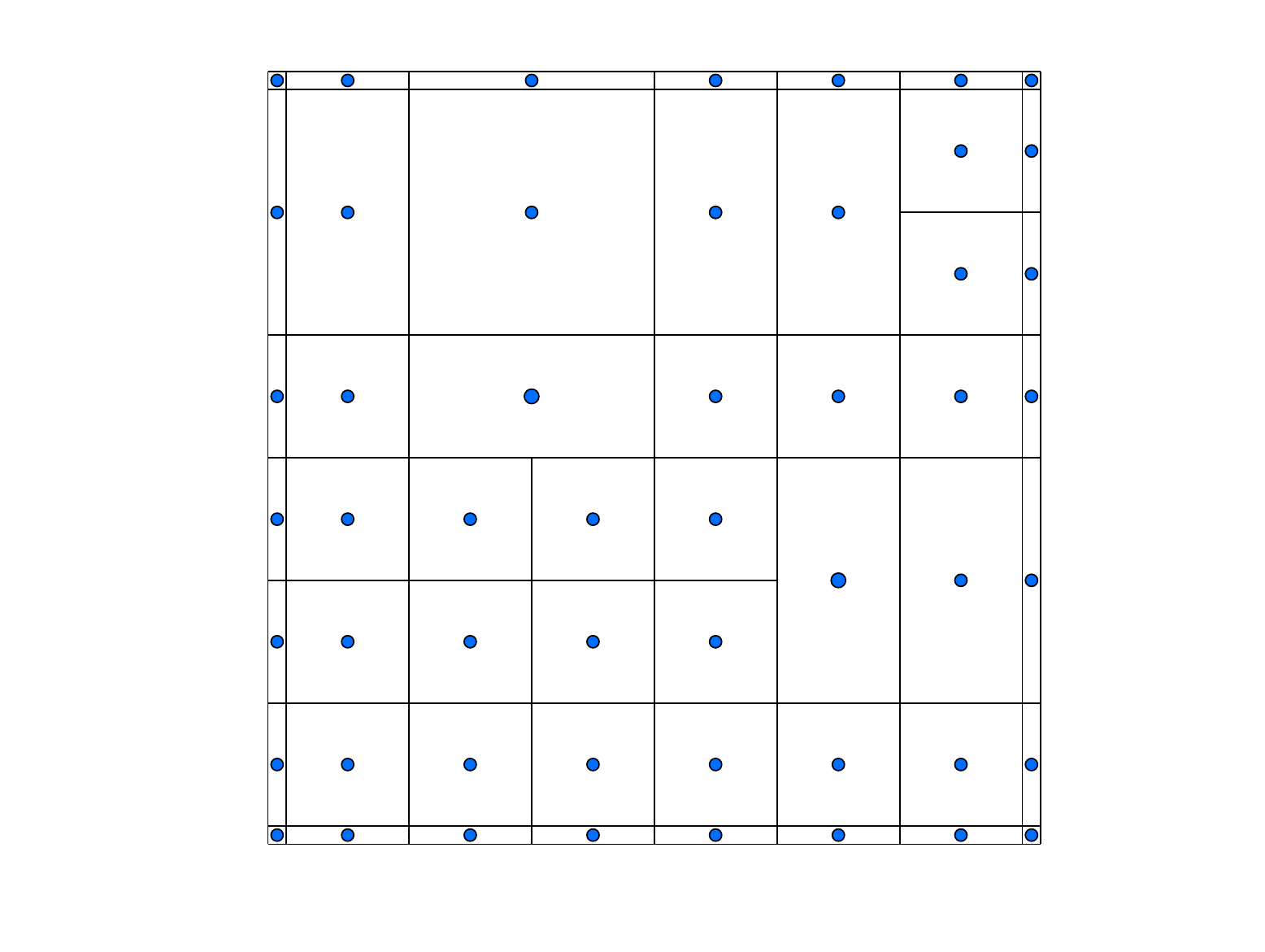}}
\end{subfigure}
\end{center}
\begin{subfigure}[Mesh $\Moneh$]{\includegraphics[width=.47\textwidth]{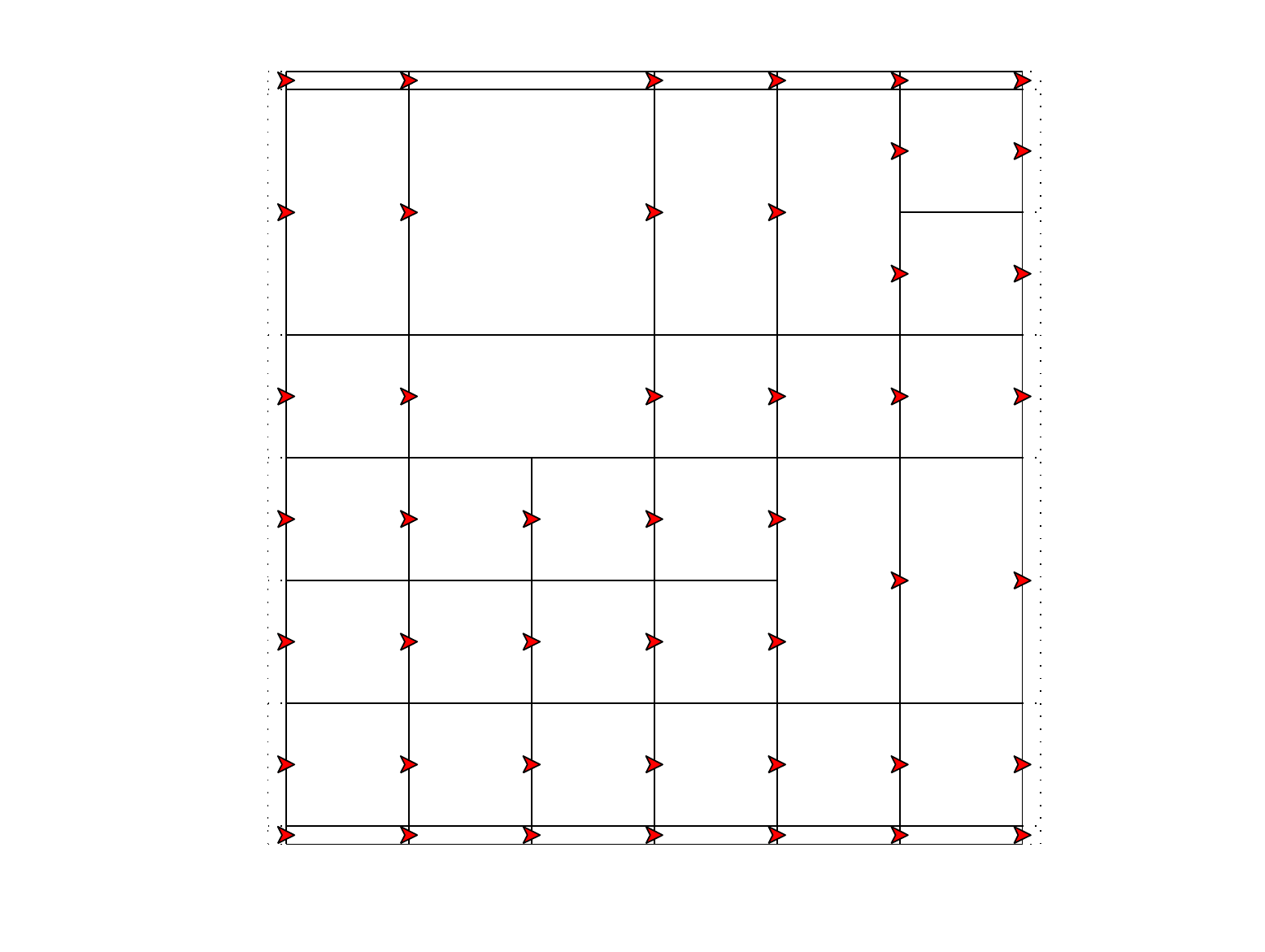}}
\end{subfigure}
\begin{subfigure}[Mesh $\Monev$]{\includegraphics[width=.47\textwidth]{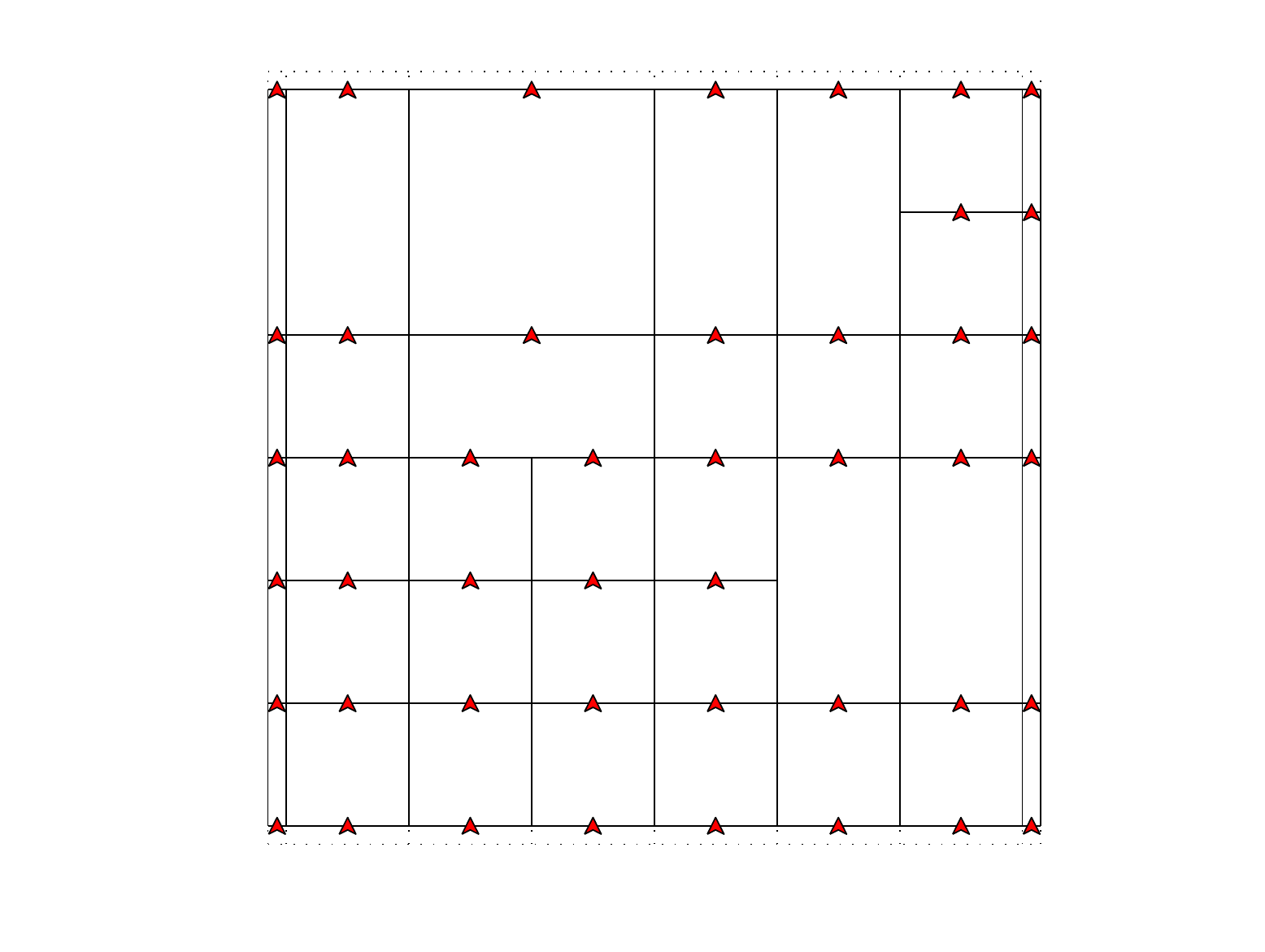}}
\end{subfigure}
\begin{center}
\begin{subfigure}[Mesh $\Mtwo$]{\includegraphics[width=.47\textwidth]{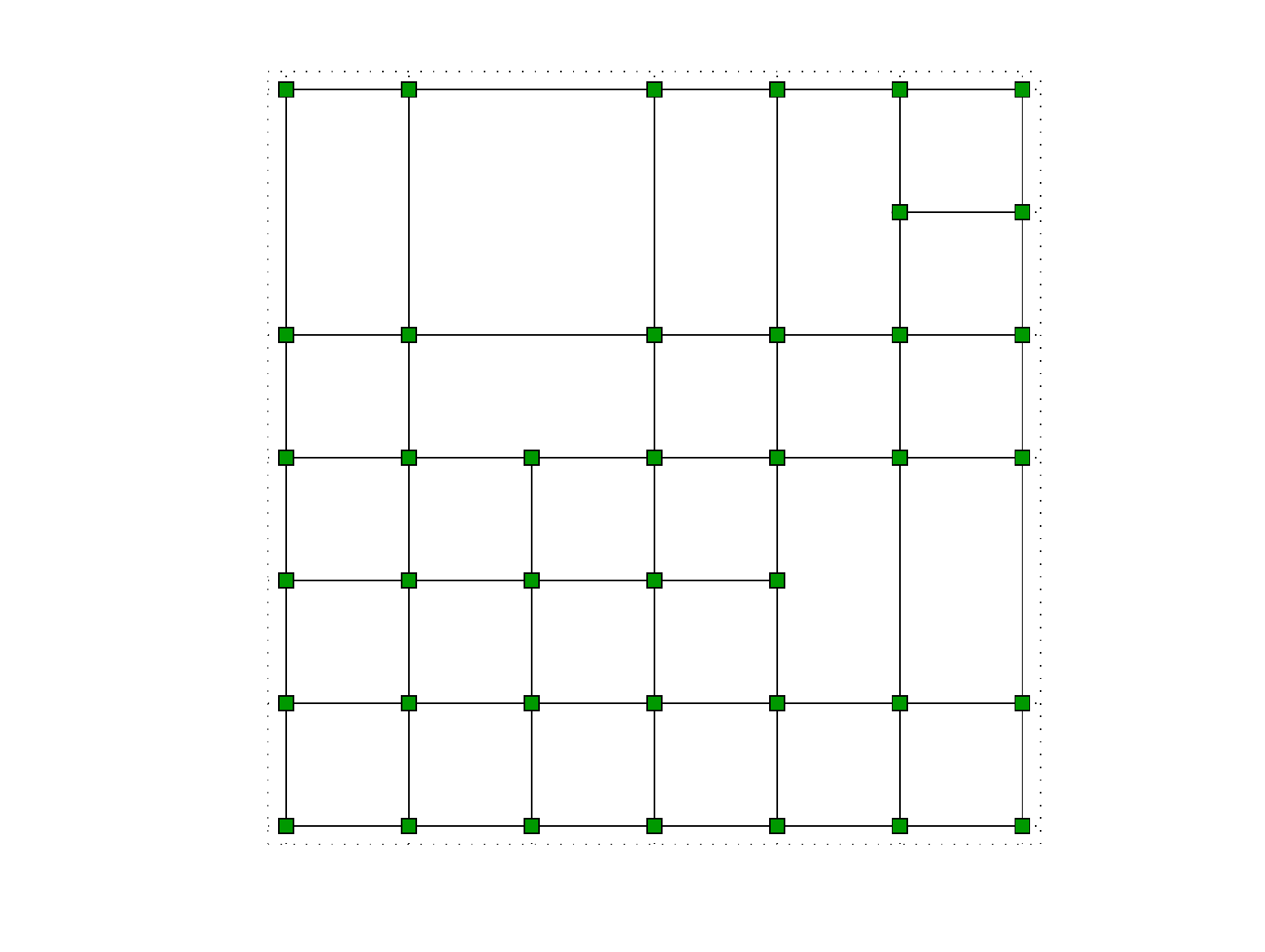}}
\end{subfigure}
\end{center}
\caption{Sequence of meshes for the spline complex, with their respective anchors, for $p=2$.} \label{fig:Tmesh-p2}
\end{figure}

The choice of these meshes becomes clear when computing the derivatives. For instance, let $p=3$ and consider the simple example of a mesh with only one horizontal T-junction, as in Figure~\ref{fig:oneT0}. Choosing the anchor $A \in \A_{p,p}(\Mzero)$ located at the T-junction, it is clear from \eqref{eq:derivative-of-splines} that $\frac{\partial B^A_{p,p}}{\partial x}$ is a linear combination of $B^{A^l}_{p-1,p}$ and $B^{A^r}_{p-1,p}$, with $A^l, A^r \in \A_{p-1,p}(\Moneh)$ as in Figure~\ref{fig:oneT11}. Hence, $\frac{\partial B^A_{p,p}}{\partial x} \in T_{p-1,p}(\Moneh)$ (see \eqref{eq:T-splines-1-forms-on-parametric}), but since $A^l \not \in \A_{p-1,p}(\Mzero)$, we have $\frac{\partial B^A_{p,p}}{\partial x} \not \in T_{p-1,p}(\Mzero)$. The argument is analogous for the  partial derivative with respect to the $y$ direction, with vertical T-junctions.

\begin{figure}
\begin{subfigure}[Anchor $A \in \A_{p,p}(\Mzero)$]{\includegraphics[width=.47\textwidth]{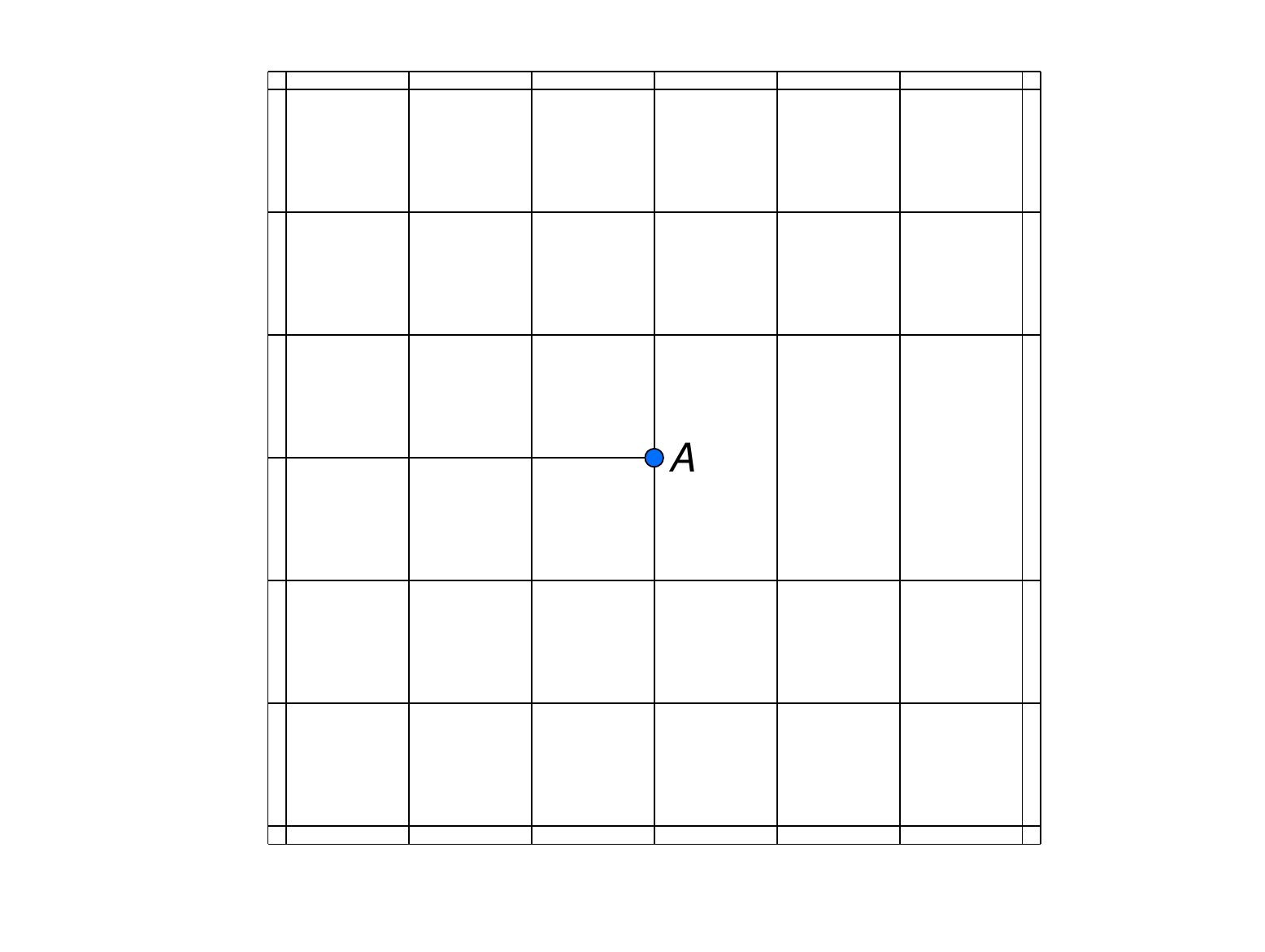} \label{fig:oneT0}}
\end{subfigure}
\begin{subfigure}[Anchors $A^l, A^r \in \A_{p-1,p}(\Moneh)$]{\includegraphics[width=.47\textwidth]{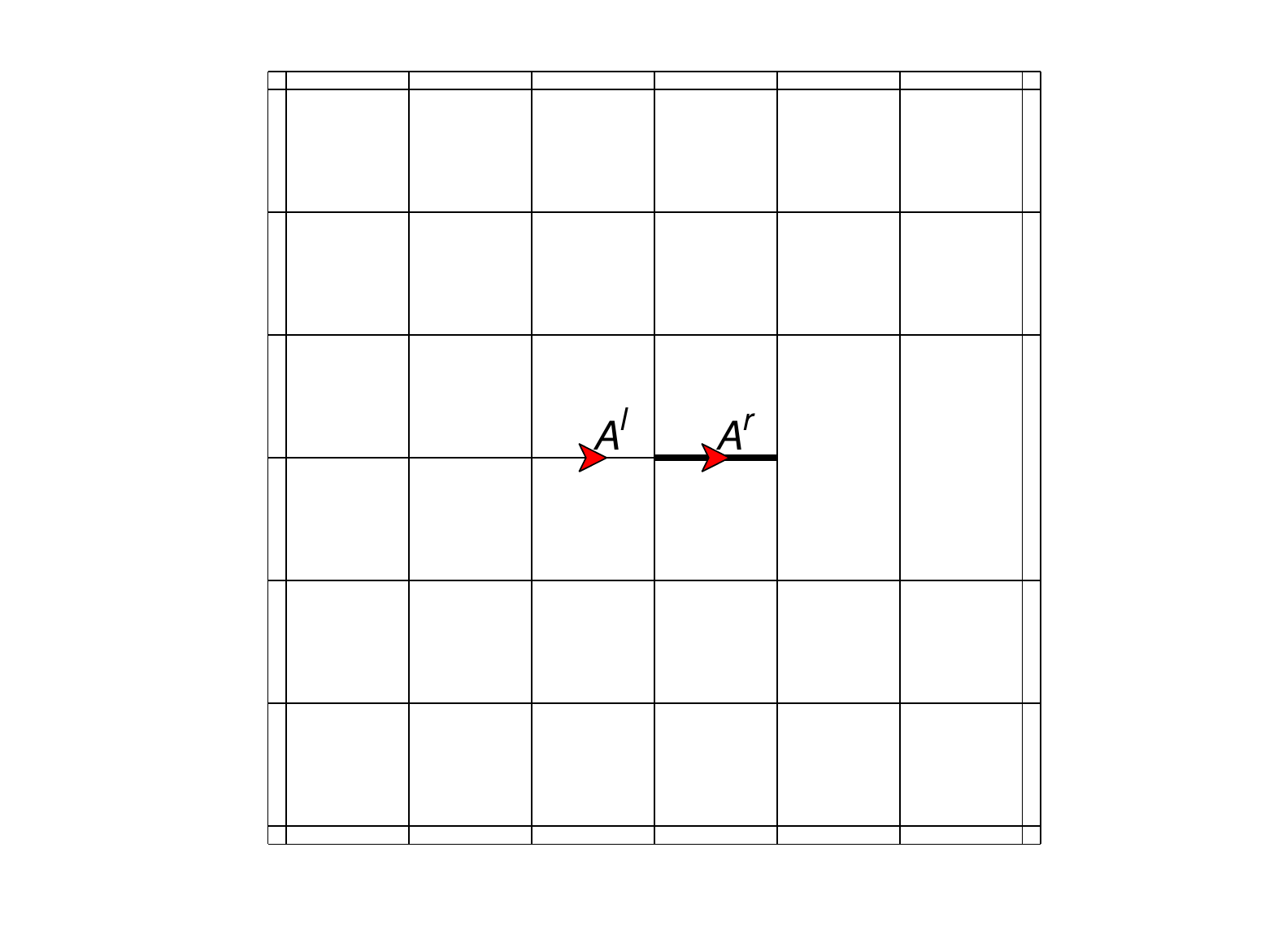} \label{fig:oneT11}}
\end{subfigure}
\caption{The first partial derivative $\frac{\partial B^A_{p,p}}{\partial x}$ is a linear combination of $B^{A^l}_{p-1,p}$ and $B^{A^r}_{p-1,p}$. \label{fig:oneT}}
\end{figure}

We have the following result.
\begin{proposition}\label{prop:extension}
Assuming    $\M^0 \in \AS_{p,p}$, \B  it holds that $\Moneh \in \AS_{p-1,p}$, $\Monev \in \AS_{p,p-1}$, and $\Mtwo \in \AS_{p-1,p-1}$.
 Moreover,  $(\M^0)_{\ext} = (\Moneh)_\ext = (\Monev)_{\ext} =
 (\Mtwo)_{\ext}$.  

\end{proposition}
\begin{proof}
The result is an immediate consequence of $\M \in \AS_{p,p}$, and the length of the extensions specified in Section~\ref{sect:AS}.
\end{proof}
\begin{remark} Note that, although the four meshes are different, all integral
  computations are carried out in the extended T-mesh, which is the
  same for all the spaces. As a consequence the four spaces can be
  implemented within the same data structure, which is based on one
  single mesh, but with different basis functions for each space. This
  is also what occurs  with  standard finite elements.

\end{remark}
\B

The bases for $ \hat Y^{0}_h, \ldots   \hat Y^{2}_h $, are formed by
T-spline functions \eqref{eq:T-spline-function}  with a scaling as in
Section~\ref{sec:compatible -splines-parametric-domain}. Precisely, introducing the notation  $D[\Xi^A_i](\zeta_i) := \frac{p}{|\Xi^A_i|}N[\Xi^A_i](\zeta_i)$, we have:

\begin{equation}
  \label{eq:0-forms-basis-Tsplines}
  \hat Y^{0}_h = \text{span} \left \{ (\zeta_1, \zeta_2)  \mapsto
   N[\Xi_1^A] (\zeta_1) N[\Xi_2^A] (\zeta_2) :\, A \in \A_{p,p}(\M^0)  \right \} ,
\end{equation}

\begin{equation}
  \label{eq:1-forms-basis-Tspline}
  \begin{aligned}
    \hat Y^{1}_h & = \text{span }  I \cup II, \text{ with}\\
 I & =\left \{ (\zeta_1, \zeta_2)  \mapsto
    D[\Xi_1^A] (\zeta_1) N[\Xi_2^A] (\zeta_2) \hat \be_1 :\, A \in \A_{p-1,p}(\M^1_1)  \right \} ,\\
 II & =\left \{ (\zeta_1, \zeta_2)  \mapsto
    N[\Xi_1^A] (\zeta_1) D[\Xi_2^A] (\zeta_2) \hat \be_2 :\, A \in \A_{p,p-1}(\M^1_2)  \right \}  .
  \end{aligned}
\end{equation}

\begin{equation}
  \label{eq:1*-forms-basis-Tspline}
  \begin{aligned}
    \hat Y^{1*}_h & = \text{span }  I \cup II, \text{ with}\\
 I & =\left \{ (\zeta_1, \zeta_2)  \mapsto
    N[\Xi_1^A] (\zeta_1) D[\Xi_2^A] (\zeta_2) \hat \be_1 :\, A \in \A_{p,p-1}(\M^1_2)  \right \} ,\\
 II & =\left \{ (\zeta_1, \zeta_2)  \mapsto
    D[\Xi_1^A] (\zeta_1) N[\Xi_2^A]
   (\zeta_2) \hat \be_2 :\, A \in \A_{p-1,p}(\M^1_1)  \right \} . 
  \end{aligned}
\end{equation}
\begin{equation}
  \label{eq:2-forms-basis-Tsplines}
  \hat Y^{2}_h = \text{span} \left \{ (\zeta_1, \zeta_2)  \mapsto
    D[\Xi_1^A] (\zeta_1) D[\Xi_2^A] (\zeta_2) :\, A \in \A_{p-1,p-1}(\M^2)  \right \} ,
\end{equation}

The main result of this section is the following.
\begin{theorem}\label{thm:T-spline-complexes}
Under the assumptions of Proposition~\ref{prop:characterization}, the
following two-dimensional complexes 
\begin{equation}\label{eq:Tsplines-complex}
\begin{CD}
\R @>>>\hat Y^{0}_h  @>\hat \grad>>  \hat Y^{1}_h @>\hat {\rot}>>
\hat Y^{2}_h @> >> 0,
\end{CD}
\end{equation}
\begin{equation}\label{eq:Tsplines-complex-*}
\begin{CD}
\R @>>>\hat Y^{0}_h  @>\hat \brot>>  \hat Y^{1*}_h @>\hat {\div}>>
\hat Y^{2}_h @> >> 0,
\end{CD}
\end{equation}
where  $\hat {\rot} \bu = (\partial_1 u_2 - \partial_2 u_1)$ is the scalar rotor and $\brot u =  (\partial_2 u,  - \partial_1 u)^T$ is the vector rotor, are well defined and exact.
\end{theorem}
\begin{proof}
  In the proof we only consider \eqref{eq:Tsplines-complex}, since
  \eqref{eq:Tsplines-complex-*} is equivalent. The well posedness of
  the complex follows from
  \begin{equation}
    \label{eq:proof-step-1}
    \hat \grad : \hat Y^{0}_h \rightarrow \hat Y^{1}_h \qquad
    \text{and} \qquad   \hat \rot : \hat Y^{1}_h \rightarrow \hat
    Y^{2}_h,
  \end{equation}
which, in turn, easily follows from the definitions
\eqref{eq:T-splines-0-forms-on-parametric}--\eqref{eq:T-splines-2-forms-on-parametric}
and from  Proposition~\ref{prop:characterization}.

Exactness of \eqref{eq:Tsplines-complex} means
\begin{equation}
  \label{eq:exactenss-1}
\R =  \mathrm{ker}( \hat \grad),
\end{equation}
\begin{equation}
  \label{eq:exactenss-2}
  \mathrm{im}(\hat \grad) =  \mathrm{ker}( \hat \rot),
\end{equation}
\begin{equation}
  \label{eq:exactenss-3}
  \mathrm{im}(\hat \rot)=  Y^{2}_h.
\end{equation}
The first part, i.e., \eqref{eq:exactenss-1}, is obvious. Moreover
\eqref{eq:exactenss-2} is also simple: indeed if $\hat \bu \in \hat
Y^{1}_h$ has null $\hat \rot $, then $ \hat \bu = \hat \grad \hat
\phi$, where, e.g., 
\begin{equation}
  \label{eq:phi-h-poincare-T-splines}
  \hat \phi(\zeta_1,\zeta_2) = \int _0^{\zeta_1}  \hat
  u_1(\eta,0)\, d \eta + \int _0^{\zeta_2}  \hat
  u_2(\zeta_1,\eta)\, d \eta.
\end{equation}

Since, $ \hat \bu = \hat \grad \hat
\phi$, then  $\hat \phi$  has to be element by element (of $\M_{ext}$)
a $p$-degree tensor-product polynomial. Then, $\hat \phi$ inherits  the interelement regularity from $\hat \bu$ and has the one of functions in  $ \hat Y^{0}_h$.  Then, by Proposition~\ref{prop:characterization}\B,
 $\hat \phi \in \hat Y^{0}_h$.  The last
point, \eqref{eq:exactenss-3},  follows from the dimension formula 
\begin{equation}
  \label{eq:proof-step-2}
  \dim (\hat Y^{0}_h ) +   \dim (\hat Y^{2}_h ) =   \dim (\hat Y^{1}_h ) + 1.
\end{equation}
Indeed, using \eqref{eq:exactenss-2}, \eqref{eq:exactenss-1}, and \eqref{eq:proof-step-2},
\begin{displaymath}
  \begin{aligned}
        \dim ( \mathrm{im}(\hat \rot)) &=   \dim (\hat Y^{1}_h )  - \dim (
    \mathrm{ker}( \hat \rot))  \\ 
&=   \dim (\hat Y^{1}_h )  - \dim ( \mathrm{im}(\hat \grad) )) \\  
&=   \dim (\hat Y^{1}_h )  - \dim ( \hat Y^{0}_h) +\dim ( \mathrm{ker}(\hat \grad) ))\\  
&=   \dim (\hat Y^{1}_h )  - \dim ( \hat Y^{0}_h) +1\\  
& = \dim (\hat Y^{2}_h ) .
  \end{aligned}
\end{displaymath}
In order to prove \eqref{eq:proof-step-2}, we recall the  Euler's formula for the T-mesh $\M$
\begin{equation}\label{eq:euler}
F_0 + V_0 = E_0 + 1,
\end{equation}
where $F_0$ is the number of faces, $E_0$ the number of edges and $V_0$ the number of vertices of $\M$, including knot repetitions, zero length edges and empty elements. The proof is different for odd and even $p$.

\begin{trivlist}
\item[Case 1)] Let $p$ be odd. We can separate the edges into
  horizontal and vertical ones, and with self-explaining notation we
  have $E_0 = E_0^H + E_0^V$. Similarly, the vertices can be divided
  into horizontal T-junctions, vertical T-junctions and all the other
  vertices (including those on the boundary), in the form $V_0 = V_0^H
  + V_0^V + V_0^+$. For odd $p$ the meshes $\Moneh$ and $\Monev$ are
  constructed by adding the first-bay face-extension of horizontal and
  vertical T-junctions, respectively. Thus, using the assumption that
  T-junction extensions do not intersect,  the number of horizontal edges in $\Moneh$ is $E_1^H = E_0^H + V_0^H$, and the number of vertical edges in $\Monev$ is $E_1^V = E_0^V + V_0^V$. Similarly, the mesh $\Mtwo$ is contructed by adding all the first-bay face-extensions, and the number of faces in $\Mtwo$ is equal to $F_2 = F_0 + V_0^H + V_0^V$. 

Since $p$ is odd, and from the positions of the anchors in every mesh (see Figure~\ref{fig:Tmesh-p3}), the dimensions of the spaces are
\begin{equation*}
\dim (\hat Y^{0}_h) = V_0, \quad \dim (\hat Y^{1}_h) = E_1^H + E_1^V = E_0 + V_0^H + V_0^V, \quad \dim(\hat Y^{2}_h) = F_2 = F_0 + V_0^H + V_0^V,
\end{equation*}
and using \eqref{eq:euler} the proof is finished.

\item[Case 2)] Let $p$ be even. We denote by $V_0^B$ and $E_0^B$ the number of boundary vertices and boundary edges in $\M$, and we note that $V_0^B = E_0^B$. As before, we distinguish between horizontal and vertical edges, $E_0 = E_0^V + E_0^H$, and also for the boundary edges $E_0^B = E_0^{B,V} + E_0^{B,H}$. For even $p$ the mesh $\Moneh$ (resp. $\Monev$) is constructed by removing the first and last columns (resp. rows) of elements from $\M$. Hence, the number of vertical edges in $\Moneh$ is $E_1^V = E_0^V - E_0^{B,V}$, and the number of horizontal edges in $\Monev$ is $E_1^H = E_0^H - E_0^{B,H}$. Similarly, the mesh $\Mtwo$ is constructed by removing the first and last rows and columns of elements from $\M$, thus the number of vertices in $\Mtwo$ is $V_2 = V_0 - V_0^B$.

From the position of the anchors for even $p$ (see Figure~\ref{fig:Tmesh-p2}), the dimensions of the spaces are
\begin{equation*}
\dim (\hat Y^{0}_h) = F_0, \quad \dim (\hat Y^{1}_h) = E_1^V + E_1^H = E_0 - E_0^B, \quad \dim(\hat Y^{2}_h) = V_2 = V_0 - V_0^B.
\end{equation*}
Using \eqref{eq:euler} and that $V_0^B = E_0^B$ the proof is finished. \qed
\end{trivlist}

\end{proof}

\subsection{Three-dimensional De Rham complex based on  T-splines and B-splines} \label{sect:3D-T-splines}

We construct a three-dimensional complex on the parametric domain  by
tensor product of the two-dimensional T-spline complexes
\eqref{eq:Tsplines-complex}--\eqref{eq:Tsplines-complex-*} and the
 one-dimensional complex \eqref{eq:1D-diagram}. Then we define the spaces on the
 parametric domain $\hat \Omega = (0,1)^3$:
\begin{equation}
    \label{eq:Tsplines-3D-discrete-forms-on-omegahat}
  \begin{aligned}
\Szeroh &:= \hat Y^{0}_h\otimes S_p(\Xi), \\
\Soneh &:= [\hat Y^{1}_h\otimes S_p(\Xi)] \times [ \hat Y^{0}_h\otimes S_{p-1}(\Xi')] ,  \\
\Stwoh &:=  [\hat Y^{1*}_h\otimes S_{p-1}(\Xi')] \times [ \hat Y^{2}_h\otimes S_p(\Xi)], \\
\Sthreeh &:= \hat Y^{2}_h\otimes  S_{p-1}(\Xi') ;
\end{aligned}
\end{equation}
 which form a  complex of the kind \eqref{eq:diag1} (or \eqref{eq:diag2}
 if we also impose homogeneous Dirichlet boundary conditions). 

Assume now that the geometry map $\bF$ is
 tensor-product single-patch spline or NURBS, and fulfills Assumption~\ref{ass:F}, now with $\Szeroh$ defined as in \eqref{eq:Tsplines-3D-discrete-forms-on-omegahat}. Therefore, the push-forwards
 \eqref{eq:diag1}--\eqref{eq:diag2} give the correct \B  complex
 $(\Xzeroh  $, \ldots,  $\Xthreeh  )$ \B on $\Omega$:
 this procedure is completely analogous to what we have already
 described in Section~\ref{sec:compatible -splines-phys-domain} and
 is not detailed here.

It is not a difficulty to consider,  more generally, a multi-patch, or a  T-spline geometry
mapping. This is not detailed here, for the sake of brevity, but the
first case will be addressed in the numerical tests of the next section. 

\subsection{Concluding remarks on the T-spline complex}

As it appears from our presentation, the understanding of the T-spline complex is much less sound than the one of the spline complex, even in two space dimensions. Moreover some of the properties we have studied for splines do not hold in general for T-splines. For example, 
\begin{itemize}
\item the matrices corresponding to the operators are no more the incidence matrices of the mesh $\M$ and a similar fact is true for standard finite elements with hanging nodes, i.e., the T-spline  complex with $p=1$;
\item the definition of control mesh and control fields is not trivial especially when $p$ is even and the analogue of Section \ref{sec:dofs} is not available for T-splines. This deserves further studies. 
\end{itemize}

\section{Numerical results}
\label{sec:num}

In this section we present numerical tests showing the behavior of isogeometric methods for electromagnetic problems.
Since numerical tests for B-splines have already been presented in other works, see e.g., \cite{VB10, Buffa_Sangalli_Vazquez, BRSV11}, we will concentrate here on examples involving also T-splines. All our numerical tests have been performed with the Matlab library GeoPDEs \cite{dFRV11}. It should be said though that GeoPDEs does not have full T-splines capability, and in particular does not provide any T-splines adaptivity in the sense of \cite{Scott_Li_Sederberg_Hughes}. 

The mappings we use in this section always verify the Assumption in Section~\ref{sect:3D-T-splines}, and can be either single-patch or
multi-patch; the meshes we describe are the ones corresponding to the space $\Xzeroh$. The meshes for the other spaces are
constructed following the procedure detailed in
Section~\ref{sect:T-compatible-parametric}. In the figures of this
section, repeated lines of the mesh are represented with thicker
lines, independently of the number of repetitions. In all cases,
internal mesh lines have multiplicity one. 

 \subsection{Maxwell eigenproblem in the square domain} \label{sect:square}

As a first test we solve the two-dimensional eigenvalue problem: \textit{Find $(\bu,\omega) \in \Horot \times \R$ such that}
\begin{equation} \label{eq:eigenproblem_2D}
\int_\Omega \rot \bu \, \rot \overline \bv = \omega^2 \int_\Omega \bu \cdot \overline \bv \quad \forall \bv \in \Horot,
\end{equation}
in the square domain $\Omega = (0,\pi)^2$, for which the exact
eigenvalues are $ \omega^2 = m^2 + n^2$, with $m,n = 0, 1, \ldots$. The aim of this test is to show that the discretization of the problem with T-splines does not present spurious modes.

The coarsest mesh consists of 8 square (non-empty)
elements in the left half, and 4 rectangular (non-empty) elements in
the right half, thus creating several T-junctions on the vertical line
$\zeta_1 = 0.5$. Finer meshes are created by dividing each element into 4 (see Figure~\ref{fig:mesh_square}).

In Table~\ref{tab:eigv} we present the first non-null eigenvalues for degree~3 and for the sequence of meshes explained above. The results show that there are no spurious eigenvalues, and that a good convergence rate is obtained. In Figure~\ref{fig:square_eigv} we display the first non-null eigenvalues computed with discretizations of degree~4 and~5 in a mesh formed by 768~non-empty elements, and its comparison with the exact eigenvalues. Again, it is seen that the discrete eigenvalues are computed with the right multiplicity.

\begin{figure}[h!]
\centering
\includegraphics[width=.4\textwidth]{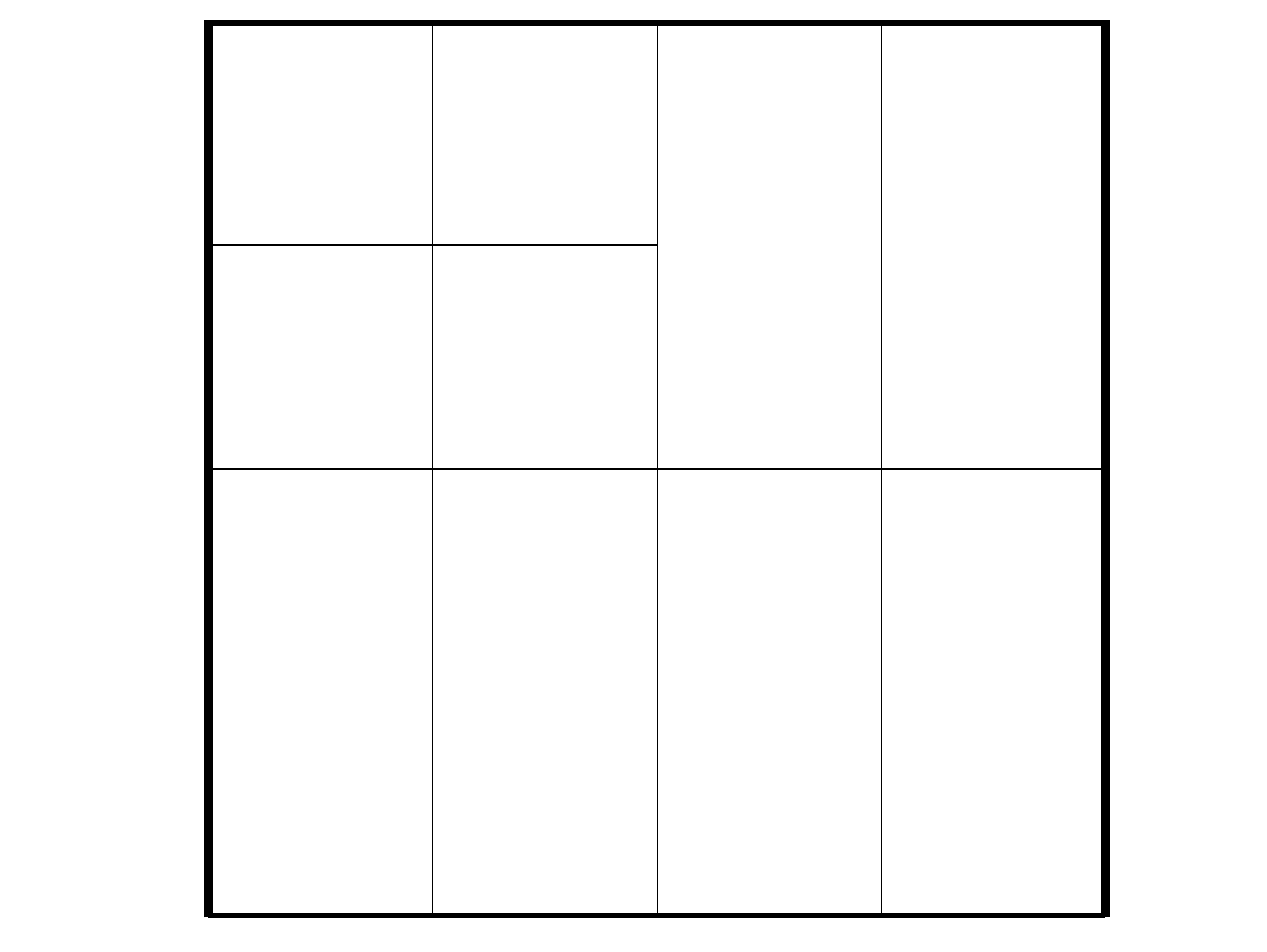}
\includegraphics[width=.4\textwidth]{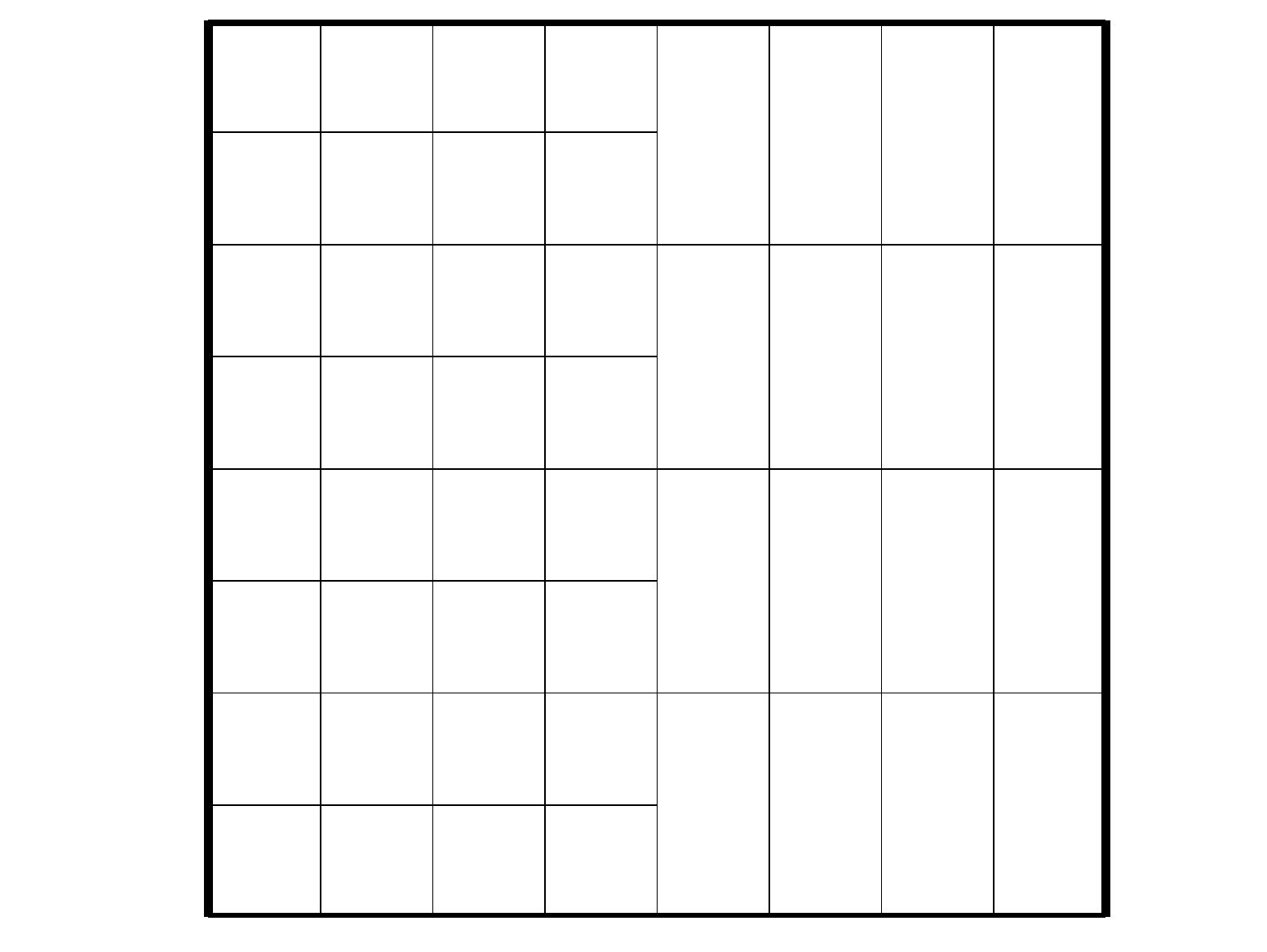}
\caption{Coarsest mesh for the square, and mesh after one refinement step.} \label{fig:mesh_square}
\end{figure}

\begin{figure}[h!]
\centering
\includegraphics[width=.48\textwidth]{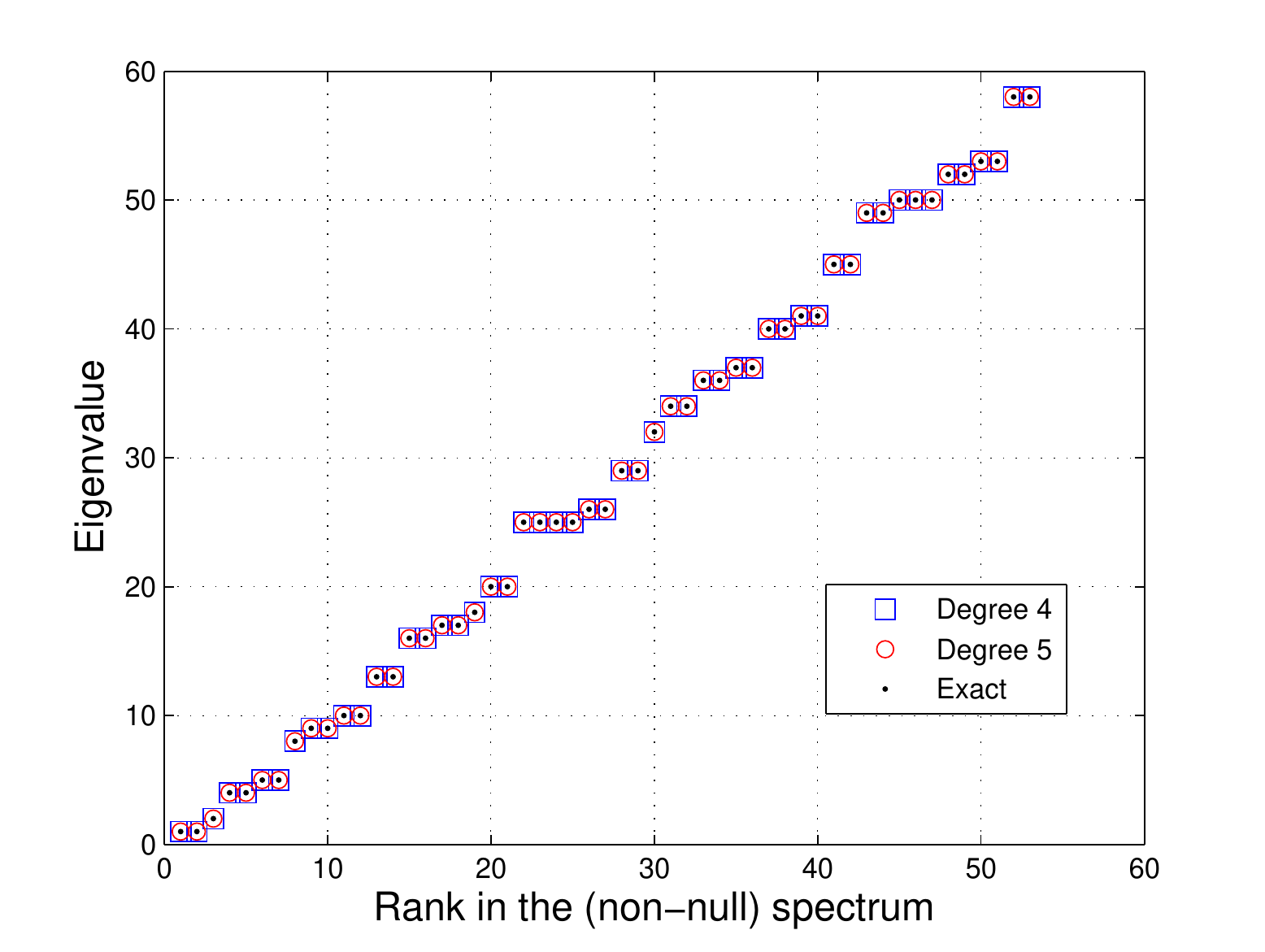}
\caption{First non-null eigenvalues computed in the square for degrees 4 and 5. } \label{fig:square_eigv}
\end{figure}

\begin{table}[!htbp]
\centering
{\footnotesize
\begin{tabular}{|c|c||c|c|c|c|c|}
\hline 
Mode & Exact & \multicolumn{5}{|c|}{Computed} \\
\hline
(1,0) & 1.00000  &1.00001 &  1.00000 &  1.00000  & 1.00000  & 1.00000\\
(0,1) & 1.00000  &1.00005 &  1.00000 &  1.00000  & 1.00000  & 1.00000\\
(1,1) & 2.00000  &2.00016 &  2.00000 &  2.00000  & 2.00000  & 2.00000\\
(2,0) & 4.00000  &4.00396 &  4.00004 &  4.00000  & 4.00000  & 4.00000\\
(0,2) & 4.00000  &4.03882 &  4.00134 &  4.00002  & 4.00000  & 4.00000\\
(2,1) & 5.00000  &5.00395 &  5.00003 &  5.00000  & 5.00000  & 5.00000\\
(1,2) & 5.00000  &5.10164 &  5.00208 &  5.00002  & 5.00000  & 5.00000\\
(2,2) & 8.00000  &8.05454 &  7.99989 &  8.00001  & 8.00000  & 8.00000\\
(3,0) & 9.00000  &9.06255 &  9.00135 &  9.00001  & 9.00000  & 9.00000\\
(0,3) & 9.00000  &9.12399 &  9.02102 &  9.00057  & 9.00001  & 9.00000\\
(3,1) & 10.0000  &10.0614 &  10.0014 &  10.0000  & 10.0000  & 10.0000\\
(1,3) & 10.0000  &10.2361 &  10.0324 &  10.0007  & 10.0000  & 10.0000\\
(3,2) & 13.0000  &12.8159 &  13.0028 &  13.0000  & 13.0000  & 13.0000\\
(2,3) & 13.0000  &13.2002 &  13.0091 &  13.0004  & 13.0000  & 13.0000\\
(4,0) & 16.0000  &17.9413 &  16.0181 &  16.0002  & 16.0000  & 16.0000\\
(0,4) & 16.0000  &19.8934 &  16.2962 &  16.0076  & 16.0001  & 16.0000\\
(4,1) & 17.0000  &19.9586 &  17.0181 &  17.0002  & 17.0000  & 17.0000\\
(1,4) & 17.0000  &20.8937 &  18.0245 &  17.0092  & 17.0001  & 17.0000\\
(3,3) & 18.0000  &21.4707 &  18.7373 &  18.0008  & 18.0000  & 18.0000\\
(4,2) & 20.0000  &24.0689 &  20.0191 &  20.0002  & 20.0000  & 20.0000\\
(2,4) & 20.0000  &26.1844 &  21.6138 &  20.0056  & 20.0001  & 20.0000\\
\hline \multicolumn{2}{|c||}{d.o.f.}  & 74 & 184 & 548 & 1852 & 6764 \\
\hline \multicolumn{2}{|c||}{number of zeros} & 21 & 65 & 225 & 833 & 3201 \\
\hline
\end{tabular}
}
\caption{First non-null eigenvalues computed in the square for $p=3$.}
\label{tab:eigv}
\end{table}

\begin{remark}
We have also solved the previous test by the mixed formulations in
\cite{Boffi07}, that make use of the full two-dimensional De Rham
complex \eqref{eq:Tsplines-complex}. The computed non-null
eigenvalues are the same as for the plain formulation
\eqref{eq:eigenproblem_2D}, while the zero eigenvalues are filtered   with the mixed formulation\B. These results, that we do not present here for the sake of brevity, confirm that the construction of the De Rham complex with T-splines is correct.
\end{remark}

\subsection{Maxwell eigenproblem in the thick L-shaped domain} \label{sect:thickL}
As a second test case, we solve the three-dimensional eigenvalue problem: \textit{Find $(\bu,\omega) \in \Hocurl \times \R$ such that}
\begin{equation}
  \label{eq:eigenvalues-thickL-3D}
  \int_\Omega \curl \bu \cdot\curl \overline \bv = \omega^2 \int_\Omega \bu \cdot \overline \bv \quad \forall \bv \in \Hocurl,
\end{equation}
in the thick L-shaped domain $\Omega = \Sigma \times (0,1)$, where $\Sigma = (-1,1)^2 \setminus [-1,0]^2$. From \cite{CD2000}, it is known that the reentrant edge introduces a singularity in the first eigenfunction, which only belongs to the space $H^{2/3-\varepsilon}(\Omega)$ for any $\varepsilon > 0$. 

It is well known that in order to recover the optimal convergence
rate we need to suitably refine the mesh toward the reentrant edge,
see e.g., \cite{MR1860445} or \cite{BuCoDa05}. Anisotropic
elements need to be used in this case (see \cite{Beirao_Cho_Sangalli} for some theoretical
background on the topic). We propose here a
dyadic refinement based on T-splines.

For the geometry representation, the thick L-shaped domain is parametrized as
the union of three cubic patches.  Following Section  \ref{sec:multi-patch-complex-continuity}, scalar fields in $\Xzeroh$ are only continuous at the interfaces between patches, and the fields in $\Xoneh$, which are used in the
discretization of \eqref{eq:eigenvalues-thickL-3D}, are  only
tangentially continuous at these interfaces (like for standard
edge finite elements), but at least $C^{p-2}$ within patches.

The refinement is obtained via T-splines by dyadic partitioning of elements which are close to the reentrant edges \cite[Ch.~4]{Apel}.   We perform the refinement first in an L-shaped two-dimensional section, and then propagate to the three-dimensional domain with a uniform mesh in the $z$-direction, as already explained in Section~\ref{sect:3D-T-splines}. The refinement is performed identically for every patch, in such a way that conformity can be ensured at the patch interfaces. 

To construct the two-dimensional mesh, at each refinement step, and for every patch, we refine a small square region near the reentrant edge, subdividing each element into 4. Then some T-junction extensions are added, depending on the degree, to make the mesh analysis suitable, as defined in Section~\ref{sect:AS}. For instance, in the example of Figure~\ref{fig:mesh_thickL} we start with an $8\times8$ mesh for each patch, which is drawn in black. At the first step we refine a region of $3 \times 3$ elements on each patch. Since the degree is $p=4$, two-bay extensions must be added to make the mesh analysis suitable. The refined elements at this step are given by the blue lines. At the second refinement step, which is marked in red, we first refine a square region of $2 \times 2$ elements, and again we add the two-bay extensions to make the mesh analysis suitable. Finally we remark that, since the dyadic partition and the definition of analysis suitable T-mesh depend on the degree, different meshes are used for different degrees. \B


\begin{figure}[h!]
\centering
\includegraphics[width=.4\textwidth]{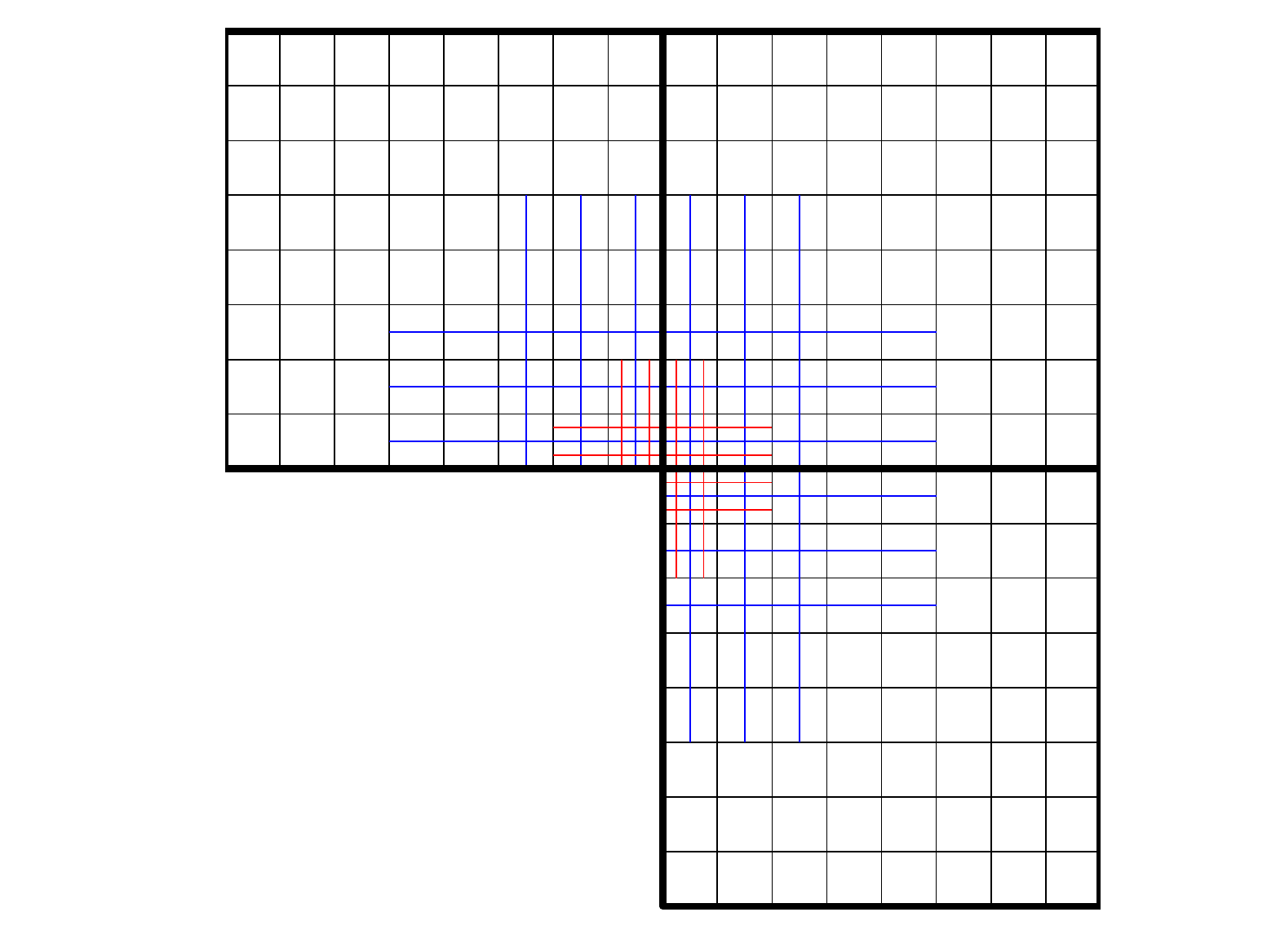}
\includegraphics[width=.45\textwidth]{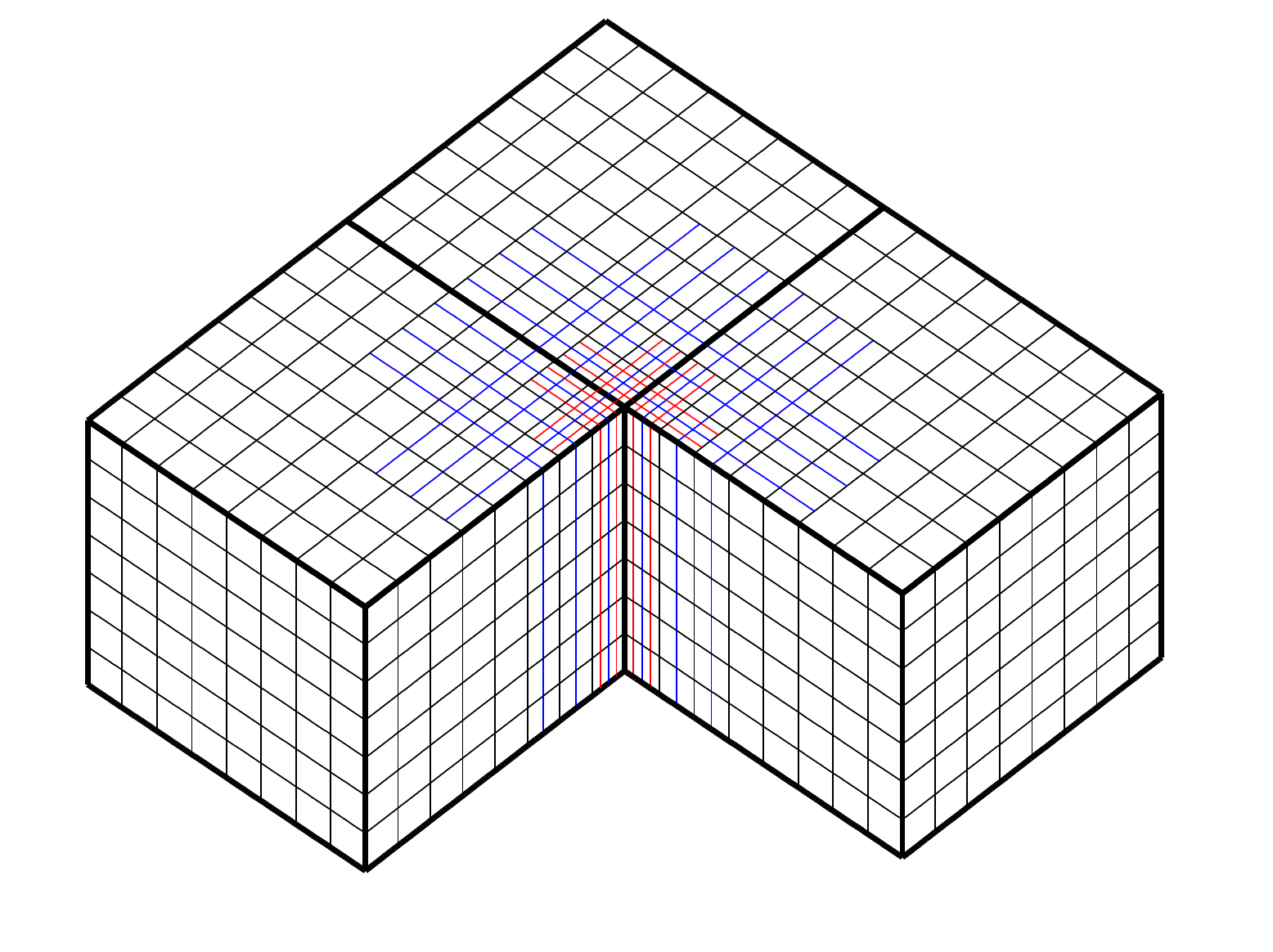}
\caption{Example of a two-dimensional mesh in the L-shaped domain $\Sigma$, and its extension to the three-dimensional domain $\Omega = \Sigma\times (0,1)$ , for $p=4$.} \label{fig:mesh_thickL}
\end{figure}

The problem has been solved for degrees~$4$ and~$5$. In Tables~\ref{tab:thickL_deg4} and~\ref{tab:thickL_deg5} we present the first non-null computed eigenvalues in the three cases, and its comparison with the exact solution. In Figure~\ref{fig:thickL_eigv} we display the convergence rate for the first eigenvalue, comparing the results obtained with T-splines and with a B-spline discretization of the same degree in the corresponding refined tensor-product dyadic mesh. As can be seen, with T-splines we obtain the same error with an important reduction in the number of the degrees of freedom.

\begin{figure}[h!]
\centering
\begin{subfigure}[Degree 4]{\includegraphics[width=.48\textwidth]{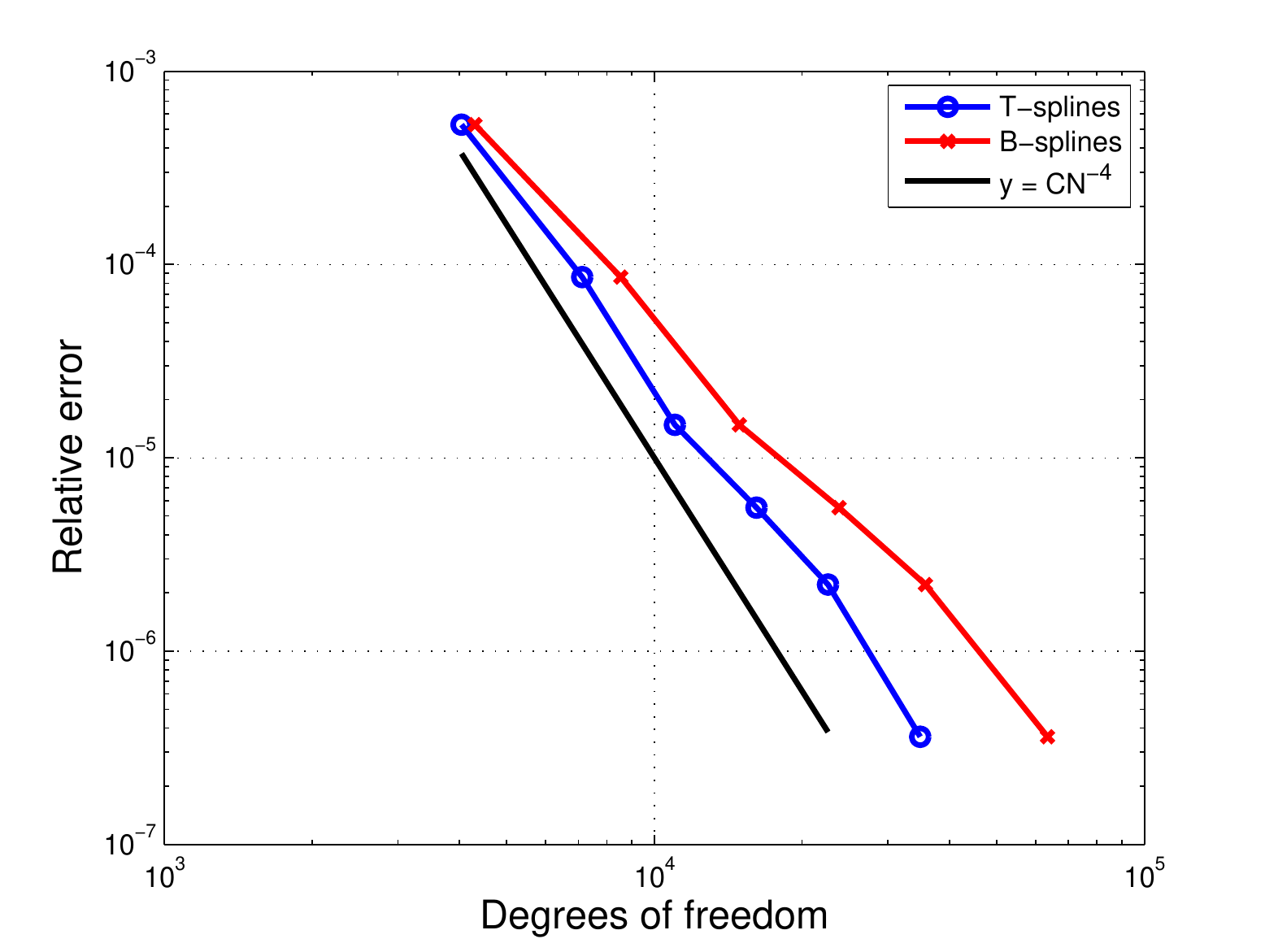}}
\end{subfigure}
\begin{subfigure}[Degree 5]{\includegraphics[width=.48\textwidth]{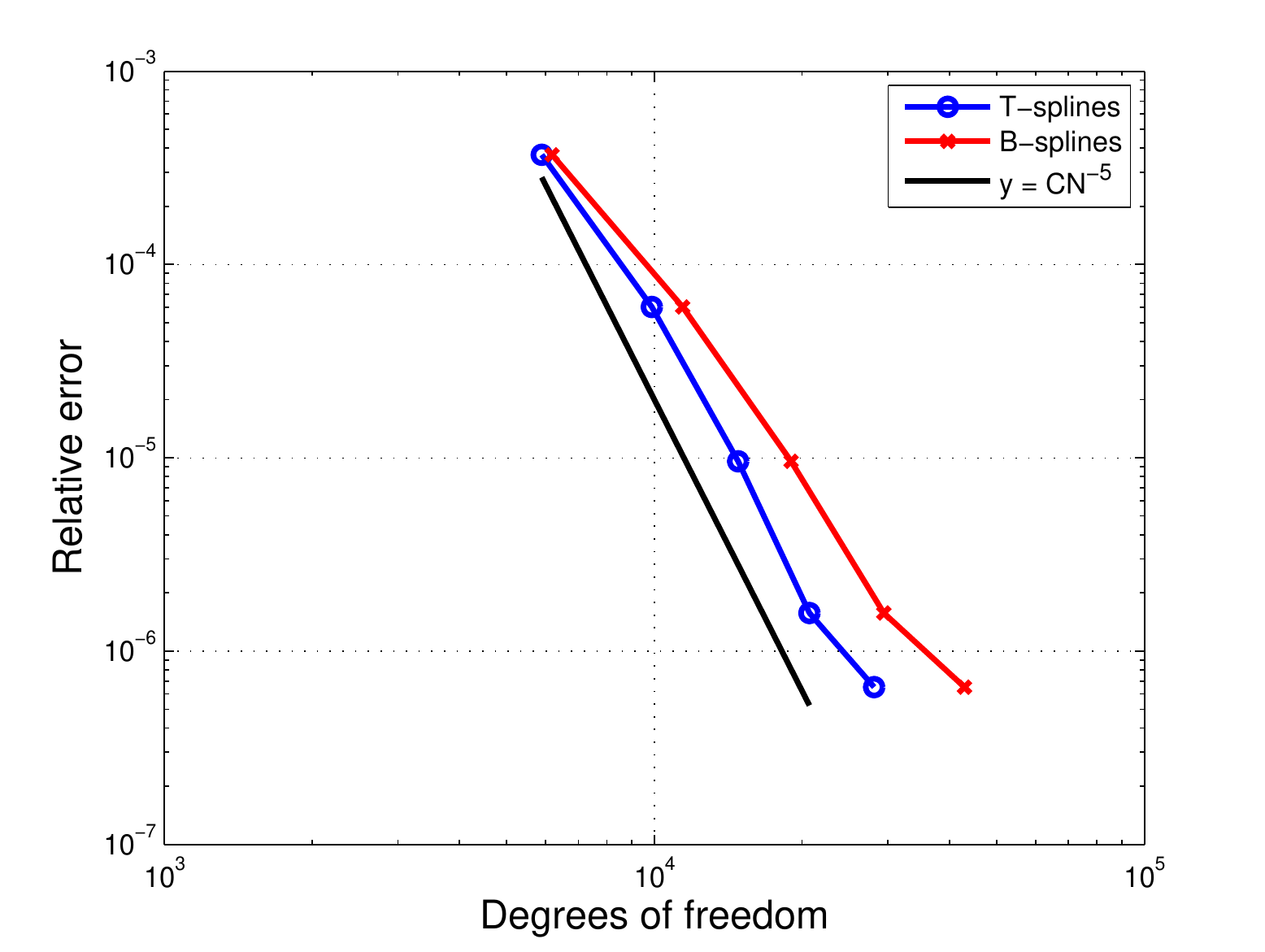}}
\end{subfigure}
\caption{Convergence of the first eigenvalue in the thick L-shaped domain.} \label{fig:thickL_eigv}
\end{figure}

\begin{table}[!htbp]
\centering
{\footnotesize
\begin{tabular}{|c||c|c|c|c|c|c|}
\hline 
Exact & \multicolumn{6}{|c|}{Computed} \\
\hline
9.63972384472	&9.64482260735	&9.64055367165	&9.63986647533	&9.63977706731	&9.63974511214	&9.63972731966 \\
\hline
11.3452262252	&11.3444193267	&11.3450973393	&11.3452056015	&11.3452178503	&11.3452228875	&11.3452256921 \\
\hline
13.4036357679	&13.4036330719	&13.4036359208	&13.4036359870	&13.4036357431	&13.4036357654	&13.4036357699 \\
\hline
15.1972519265	&15.1973643408	&15.1973310163	&15.1973301300	&15.1972556440	&15.1972524223	&15.1972523662 \\
\hline
19.5093282458	&19.5144198480	&19.5101576732	&19.5094708082	&19.5093814308	&19.5093494993	&19.5093317180 \\
\hline
19.7392088022	&19.7392474090	&19.7392464705	&19.7392464473	&19.7392098765	&19.7392090606	&19.7392090522 \\
\hline
19.7392088022	&19.7392474115	&19.7392464714	&19.7392464480	&19.7392098765	&19.7392090617	&19.7392090536 \\
\hline
19.7392088022	&19.7392854949	&19.7392835833	&19.7392835402	&19.7392109156	&19.7392092780	&19.7392092574 \\
\hline
21.2590837990	&21.2591164396	&21.2591199815	&21.2591200740	&21.2590848357	&21.2590840605	&21.2590840611 \\
\hline
d.o.f. &        4042    &    7126   &    11018   &    16162   &    22630    &   34894\\
\hline
\end{tabular}
}
\caption{First non-null eigenvalues computed in the thick L-shaped domain for $p=4$.}
\label{tab:thickL_deg4}
\end{table}

\begin{table}[!htbp]
\centering
{\footnotesize
\begin{tabular}{|c||c|c|c|c|c|}
\hline 
Exact & \multicolumn{5}{|c|}{Computed} \\
\hline
9.63972384472	&9.64328299807	&9.64030443177	&9.63981618307	&9.63973902051	&9.63973012738 \\
\hline 
11.3452262252	&11.3446611860	&11.3451346157	&11.3452117088	&11.3452238284	&11.3452252153 \\
\hline
13.4036357679	&13.4036342774	&13.4036357233	&13.4036357613	&13.4036357622	&13.4036357630 \\
\hline
15.1972519265	&15.1972704673	&15.1972555872	&15.1972551905	&15.1972551805	&15.19725206245 \\
\hline
19.5093282458	&19.5128817631	&19.5099083654	&19.5094205444	&19.5093433951	&19.5093345086 \\
\hline
19.7392088022	&19.7392095886	&19.7392095763	&19.7392095758	&19.7392095758	&19.7392088273 \\
\hline
19.7392088022	&19.7392095886	&19.7392095763	&19.7392095758	&19.7392095758	&19.7392088273 \\
\hline
19.7392088022	&19.7392103678	&19.7392103456	&19.7392103444	&19.7392103444	&19.7392088514 \\
\hline
21.2590837990	&21.2590824582	&21.2590845213	&21.2590845754	&21.2590845767	&21.2590838179 \\
\hline
d.o.f. &              5891     &   9883   &    14827   &    20723   &    28105\\
\hline
\end{tabular}
}
\caption{First non-null eigenvalues computed in the thick L-shaped domain for $p=5$.}
\label{tab:thickL_deg5}
\end{table}

\subsection{Maxwell source problem in three quarters of the cylinder}
For the third test case we consider the model source problem: \textit{Find $\bu \in \HcurloD$ such that}
\begin{equation}
\int_\Omega \curl \bu \cdot \curl \overline \bv + \int_\Omega \bu \cdot \overline \bv = \int_\Omega {\bf f} \cdot \overline \bv \quad \forall \bv \in \HcurloD,
\end{equation}
where $\HcurloD$ is the set of functions with null tangential trace on
$\Gamma_D \subset \partial \Omega$, i.e., \[\HcurloD := \{\bv \in \Hcurl : \bv \times \bn =
\zero \text{ on } \Gamma_D\}. \]

The geometry $\Omega$ is three quarters of a cylinder of radius and length equal to one (see Figure~\ref{fig:mesh_camembert}), that in cylindrical coordinates is given by $\Omega = \{(r, \theta, z): 0<r<1, \, 0 <\theta<\frac{3}{2}\pi, \, 0 <z<1\}$. We impose the null tangential component on $\Gamma_D = \{(r,\theta,z) : \theta \in \{0,\frac{3}{2}\pi\}\}$, and the source term ${\bf f}$ is taken such that the exact solution is $\bu = \grad (r^{2/3} \sin (2\theta/3) \sin(\pi z))$, i.e., it is singular in the first two directions, but it is regular in the $z$ direction, for which the local refinement of the previous test is well suited for this case.

As in the previous example, the domain is defined with three patches,
and the discrete fields $\Xoneh \subset \Hcurl $ are only tangentially continuous between them. The construction of the mesh in the parametric domain is identical to the one in the previous example: for each patch we first create a two-dimensional mesh locally refined towards the corner, and extend it to the three dimensional domain by tensor product. The mesh is then mapped to the physical domain, as can be seen in Figure~\ref{fig:mesh_camembert}.

\begin{figure}[h!]
\centering
\includegraphics[width=.65\textwidth]{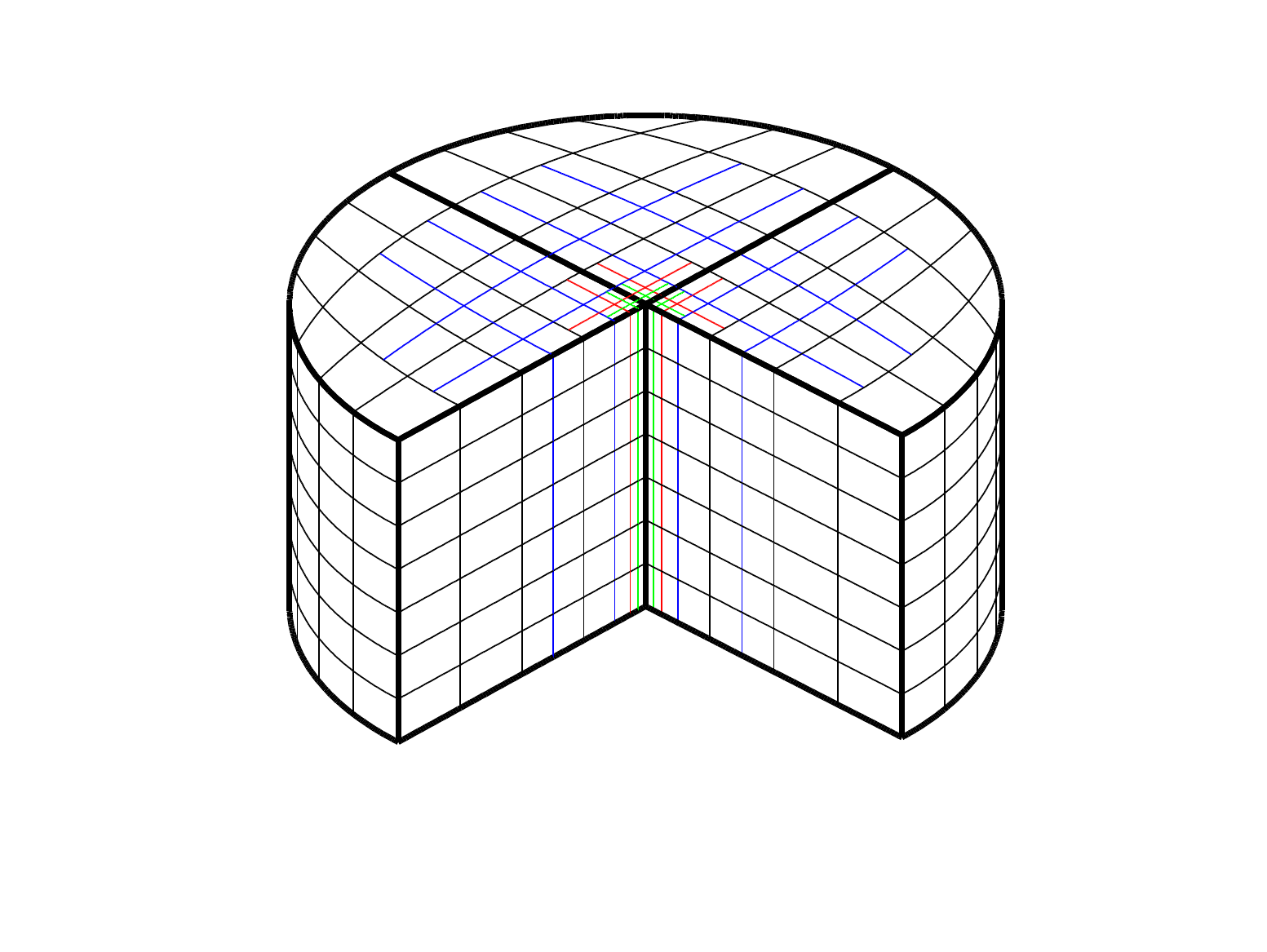}
\vspace*{-2cm}
\caption{Example of a mesh for three quarters of the cylinder.} \label{fig:mesh_camembert}
\end{figure}

The problem is solved with T-splines of degree~3, and also with B-splines of the same degree in the corresponding tensor product mesh. The errors in ${\bf H}(\bf curl)$-norm for the two methods are compared in Figure~\ref{fig:error_camembert}. As in the previous example, with T-splines we are able to obtain results similar to those of B-splines with a reduction of the number of degrees of freedom.

\begin{figure}[h!]
\centering
\includegraphics[width=.4\textwidth]{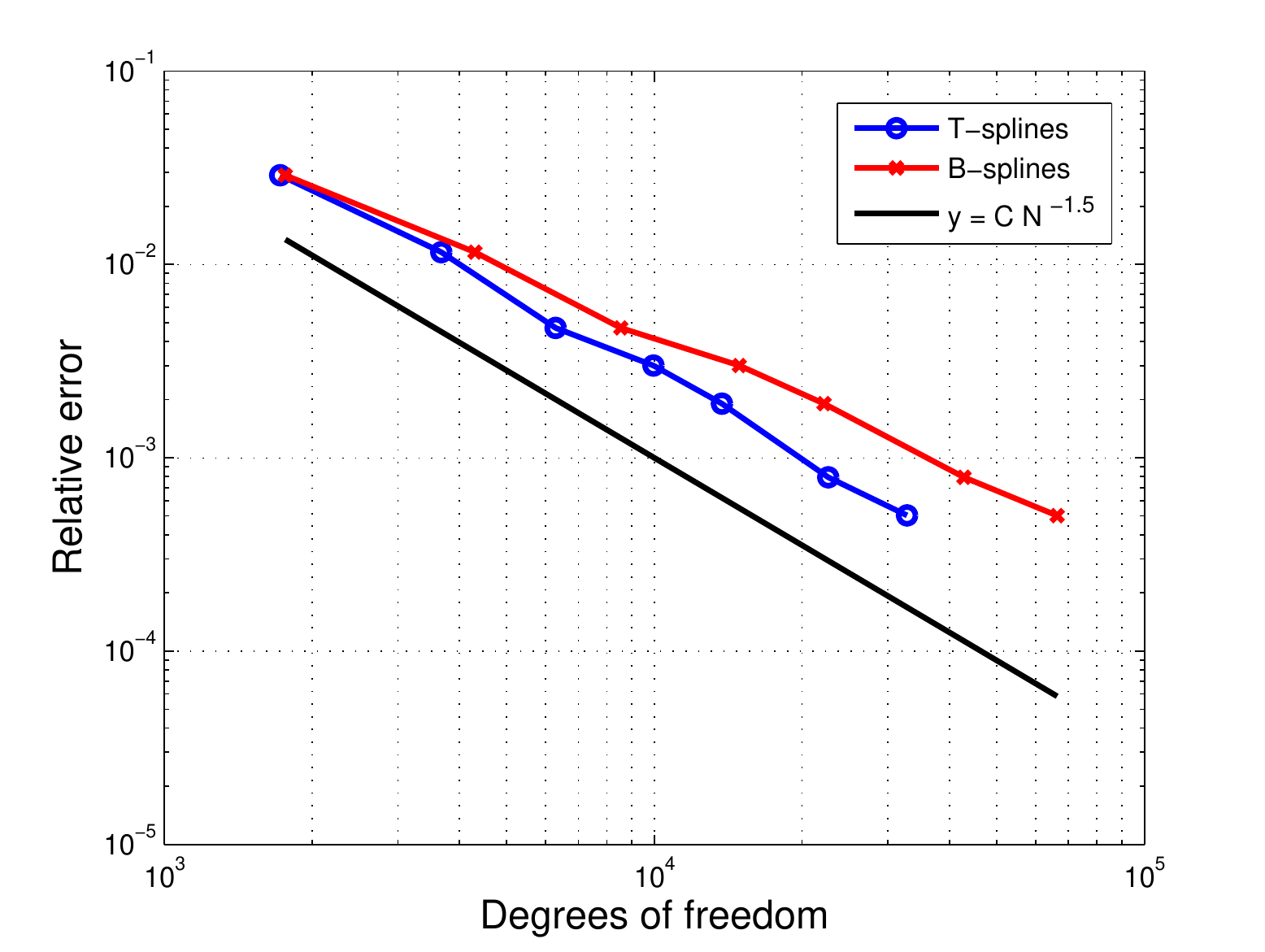}
\caption{Comparison of the error for T-splines and B-splines.} \label{fig:error_camembert}
\end{figure}

\subsection{Numerical simulation of a twisted waveguide}

As the last numerical test we use T-splines to simulate the
propagation of a singular mode in a waveguide with a twist. The
configuration, which is presented in Figure~\ref{fig:waveguide}, is
the same given in \cite[Ch.~8]{Jin}, changing the material
discontinuity by a geometric inhomogeneity (the twist). The problem is solved in a waveguide with a twist of 90~degrees, with a section of three quarters of the circle of 2~cm radius, and for which we assume that the walls are perfect electrical conductors. We also assume that the waveguide extends to infinity without other inhomogeneities, and it is truncated by the planes $\Gamma_1$ and $\Gamma_2$ to obtain the computational domain, which consists of three different regions: a middle region where the waveguide is twisted (see Figure~\ref{fig:waveguide}), and two straight regions near the ports, to keep the inhomogeneity far enough from them, in such a way that only the dominant mode TE$_{10}$ can propagate without attenuation. The total length of the computational domain is 24~cm: 4~cm for each straight region, and 16~cm for the region with the twist. The frequency $\omega$ is taken equal to 32~GHz, and it is between the cutoff frequencies for the first mode (21~GHz) and the second mode (33.84~GHz).
\B

Following \cite{Jin}, and working in the time harmonic regime at a given frequency $\omega$, the (complex-valued) electric field $\bE \in \HcurloD$ is solution of the problem
\begin{equation}\label{eq:vf_waveguide}
\begin{array}{rcl}
\ds \int_\Omega ( \curl \bE \cdot \curl \overline \bG - k^2 \bE \cdot \overline \bG) & + & \ds \int_{\Gamma_1 \cup \Gamma_2} \imag \beta_{10} (\bn \times \bE) \cdot (\bn \times \overline \bG) = \\
&& \ds 2\imag \beta_{10} \, \int_{\Gamma_1} \bE^{\rm inc} \cdot \overline \bG, \quad \forall \bG \in \HcurloD,
\end{array}
\end{equation}
where  
$k = \sqrt{\omega^2 \mu_0 \varepsilon_0}$ with $\mu_0$ and $\varepsilon_0$ the magnetic permeability and electric permittivity of free space. The incident electric field $\bE^{\rm inc}$ at the port $\Gamma_1$, and the wavenumber of the first mode $\beta_{10}$ are defined as
\begin{align*}
\bE^{\rm inc}(x,y,z) = \be_{10}(x,y) e^{-\imag \beta_{10} z}, &&
\beta_{10} = \sqrt{k^2 - k_{10}^2}.
\end{align*}
In the case of waveguides of rectangular or circular section, the
value of the constant $k_{10}$ and the mode $\be_{10}$ are known. In
the general case, they can be obtained by solving a 2D eigenvalue
problem on the port $\Gamma_1$, which consists on finding the minimum eigenvalue $k_{10} \in \R$, and its associated eigenvector $\be_{10} \in {\bf H}_0(\rot;\Gamma_1)$, such that
\begin{equation}\label{eq:2d_waveguide}
\int_{\Gamma_1} \rot \be_{10} \, \rot \overline \bv = k_{10}^2 \int_{\Gamma_1} \be_{10} \cdot \bv \quad \forall \bv \in {\bf H}_0(\rot;\Gamma_1).
\end{equation}

The electric field $\bE$ in equation \eqref{eq:vf_waveguide} is discretized with T-splines of degree~3, using the approach already explained in Section~\ref{sect:3D-T-splines}. The two-dimensional T-mesh for the section is built as in the previous examples, and in the $z$ direction we use one element for each straight region near the ports, and 4 elements along the twist, for a total of 7936 degrees of freedom. For the solution of the 2D problem \eqref{eq:2d_waveguide}, it is enough to restrict a field in $\Hcurl$ to its tangential components on the port $\Gamma_1$, which in practice is equivalent to solve with two-dimensional T-splines.

The magnitude of the real part of the computed solution $\bE$ is shown in Figure~\ref{fig:waveguide_mode}, which shows that the mode is correctly propagated. Finally, we also compute the reflection and transmission coefficients, given by the equations
\begin{equation*}
R = \frac{e^{-\imag \beta z_1} \int_{\Gamma_1} \bE \cdot \be_{10}}{\int_{\Gamma_1} \be_{10}\cdot \be_{10}} - e^{-2 \imag \beta z_1}, \qquad T = \frac{e^{\imag \beta z_2} \int_{\Gamma_1} \bE \cdot \be_{10}}{\int_{\Gamma_1} \be_{10}\cdot \be_{10}} - e^{-2 \imag \beta z_2},
\end{equation*}
and we obtain the values $|R| = 0.0025$ and $|T| = 0.9998$, which confirms that the twist does not affect the propagation of the mode, as expected.

\begin{figure}[h!]
\begin{subfigure}[Geometry of the waveguide]
{\includegraphics[width=.48\textwidth]{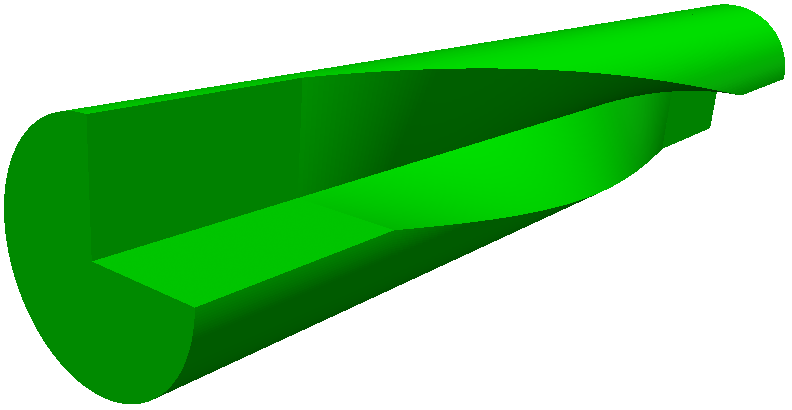} \label{fig:waveguide}}
\end{subfigure}
\begin{subfigure}[Real part of the computed electric field]
{\includegraphics[width=.45\textwidth]{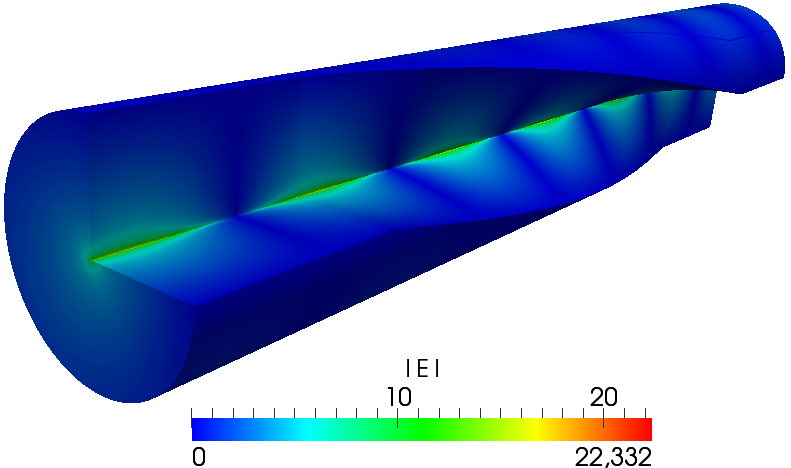} \label{fig:waveguide_mode}}
\end{subfigure}
\caption{Geometry of the waveguide and real part of the computed electric field.} 
\end{figure}

\bibliographystyle{elsarticle-num}

\begin{thebibliography}{10}
\expandafter\ifx\csname url\endcsname\relax
  \def\url#1{\texttt{#1}}\fi
\expandafter\ifx\csname urlprefix\endcsname\relax\def\urlprefix{URL }\fi
\expandafter\ifx\csname href\endcsname\relax
  \def\href#1#2{#2} \def\path#1{#1}\fi

\bibitem{HIP02a}
R.~Hiptmair, Finite elements in computational electromagnetism, Acta Numer. 11
  (2002) 237--339.

\bibitem{AFW06}
D.~N. Arnold, R.~S. Falk, R.~Winther, Finite element exterior calculus,
  homological techniques, and applications, Acta Numer. 15 (2006) 1--155.

\bibitem{Boff10}
D.~Boffi, Finite element approximation of eigenvalue problems, Acta Numer. 19
  (2010) 1--120.

\bibitem{BFGP}
D.~Boffi, P.~Fernandes, L.~Gastaldi, I.~Perugia, Computational models of
  electromagnetic resonators: analysis of edge element approximation, SIAM J.
  Numer. Anal. 36~(4) (1999) 1264--1290 (electronic).

\bibitem{BrSc08}
D.~Braess, J.~Sch{\"o}berl, Equilibrated residual error estimator for edge
  elements, Math. Comp. 77~(262) (2008) 651--672.

\bibitem{Piegl}
L.~Piegl, W.~Tiller, {The Nurbs Book}, Springer-Verlag, New York, 1997.

\bibitem{Sederberg_Zheng}
T.~Sederberg, J.~Zheng, A.~Bakenov, A.~Nasri, T-splines and {T-NURCCS}s, ACM
  Trans. Graph. 22~(3) (2003) 477--484.

\bibitem{IGA-book}
J.~A. Cottrell, T.~J.~R. Hughes, Y.~Bazilevs, Isogeometric {A}nalysis: toward
  integration of {CAD} and {FEA}, John Wiley \& Sons, 2009.

\bibitem{HCB05}
T.~J.~R. Hughes, J.~A. Cottrell, Y.~Bazilevs, Isogeometric analysis: {CAD},
  finite elements, {NURBS}, exact geometry and mesh refinement, Comput. Methods
  Appl. Mech. Engrg. 194~(39-41) (2005) 4135--4195.

\bibitem{MR2443159}
Y.~Bazilevs, V.~M. Calo, T.~J.~R. Hughes, Y.~Zhang, Isogeometric
  fluid-structure interaction: theory, algorithms, and computations, Comput.
  Mech. 43~(1) (2008) 3--37.

\bibitem{Cottrell_Reali_Bazilevs_Hughes}
J.~A. Cottrell, A.~Reali, Y.~Bazilevs, T.~J.~R. Hughes, Isogeometric analysis
  of structural vibrations, Comput. Methods Appl. Mech. Engrg. 195~(41-43)
  (2006) 5257--5296.

\bibitem{Echter2010374}
R.~Echter, M.~Bischoff, Numerical efficiency, locking and unlocking of {NURBS}
  finite elements, Computer Methods in Applied Mechanics and Engineering
  199~(5–8) (2010) 374 -- 382.

\bibitem{Lipton_Evans_Bazilevs}
S.~Lipton, J.~A. Evans, Y.~Bazilevs, T.~Elguedj, T.~J.~R. Hughes, Robustness of
  isogeometric structural discretizations under severe mesh distortion, Comput.
  Methods Appl. Mech. Engrg. 199~(5-8) (2010) 357 -- 373.

\bibitem{rank12}
E.~Rank, M.~Ruess, S.~Kollmannsberger, D.~Schillinger, A.~D\"{u}ster, Geometric
  modeling, isogeometric analysis and the finite cell method, Comput. Methods
  Appl. Mech. Engrg. (2012) doi: http://dx.doi.org/10.1016/j.cma.2012.05.022.

\bibitem{MR2835758}
L.~De~Lorenzis, {\.I}.~Temizer, P.~Wriggers, G.~Zavarise, A large deformation
  frictional contact formulation using {NURBS}-based isogeometric analysis,
  Internat. J. Numer. Methods Engrg. 87~(13) (2011) 1278--1300.

\bibitem{DQS10}
T.~Dokken, E.~Quak, V.~Skytt, Requirements from {I}sogeometric {A}nalysis for
  changes in product design ontologies, in: Proceedings of the FOCUS K3D
  Conference on Semantic 3D Media and Content (INRIA Sophia Antipolis -
  M\'{e}diterran\'{e}e, 2010), IMATI-CNR, Genova, Italy, 2010, pp. 11--15.

\bibitem{Martin_Cohen_Kirby}
T.~Martin, E.~Cohen, R.~Kirby, Volumetric parameterization and trivariate
  b-spline fitting using harmonic functions, Computer Aided Geometric Design
  26~(6) (2009) 648 -- 664, {S}olid and {P}hysical {M}odeling 2008, ACM
  Symposium on Solid and Physical Modeling and Applications.

\bibitem{CoMa10}
E.~Cohen, T.~Martin, R.~M. Kirby, T.~Lyche, R.~F. Riesenfeld, Analysis-aware
  modeling: understanding quality considerations in modeling for isogeometric
  analysis, Comput. Methods Appl. Mech. Engrg. 199~(5-8) (2010) 334--356.

\bibitem{Zhang2012}
Y.~Zhang, W.~Wang, T.~J.~R. Hughes, Solid {T}-spline construction from boundary
  representations for genus-zero geometry, Comput. Methods Appl. Mech. Engrg.
  (2012) (to appear).

\bibitem{BBCHS06}
Y.~Bazilevs, L.~Beir{\~a}o~da Veiga, J.~A. Cottrell, T.~J.~R. Hughes,
  G.~Sangalli, Isogeometric analysis: approximation, stability and error
  estimates for {$h$}-refined meshes, Math. Models Methods Appl. Sci. 16~(7)
  (2006) 1031--1090.

\bibitem{Beirao_Cho_Sangalli}
L.~Beir{\~a}o~da Veiga, A.~Buffa, D.~Cho, G.~Sangalli, Anisotropic {NURBS}
  approximation in {I}sogeometric {A}nalysis, Comput. Methods Appl. Mech.
  Engrg. 209-212 (2012) 1--11.

\bibitem{BRSV11}
A.~Buffa, J.~Rivas, G.~Sangalli, R.~V\'{a}zquez, Isogeometric discrete
  differential forms in three dimensions, SIAM J. Numer. Anal. 49~(2) (2011)
  818--844.

\bibitem{nwidth-iga}
J.~A. Evans, Y.~Bazilevs, I.~Babuska, T.~J.~R. Hughes, N-widths, sup-infs, and
  optimality ratios for the k-version of the isogeometric finite element
  method, Comput. Methods Appl. Mech. Engrg. 198~(21-26) (2009) 1726--1741.

\bibitem{Vuong_giannelli_juttler_simeon}
A.-V. Vuong, C.~Giannelli, B.~J\"uttler, B.~Simeon, A hierarchical approach to
  adaptive local refinement in isogeometric analysis, Comput. Methods Appl.
  Mech. Engrg. 200~(49-52) (2011) 3554--3567.

\bibitem{BePa12}
L.~Beir{\~a}o~da Veiga, D.~Cho, L.~Pavarino, S.~Scacchi, Overlapping {S}chwarz
  methods for {I}sogeometric {A}nalysis, SIAM J. Numer. Anal. 50~(3) (2012)
  1394–1416.

\bibitem{kleissieti}
S.~K. Kleiss, C.~Pechstein, B.~J{\"u}ttler, S.~Tomar, {IETI}-{I}sogeometric
  {T}earing and {I}nterconnecting, Comput. Methods Appl. Mech. Engrg. (2012)
  (accepted for publication).

\bibitem{Buffa_Sangalli_Vazquez}
A.~Buffa, G.~Sangalli, R.~V{\'a}zquez, Isogeometric analysis in
  electromagnetics: B-splines approximation, Comput. Methods Appl. Mech. Engrg.
  199~(17-20) (2010) 1143 -- 1152.

\bibitem{ratnani}
A.~Ratnani, E.~Sonnendr{\"u}cker, An arbitrary high-order spline finite element
  solver for the time domain {M}axwell equations, J. Sci. Comput. 51 (2012)
  87--106.

\bibitem{cohen2001geometric}
E.~Cohen, R.~Riesenfeld, G.~Elber, Geometric modeling with splines: an
  introduction, Vol.~1, AK Peters Wellesley, MA, 2001.

\bibitem{DeBoor}
C.~de~Boor, A practical guide to splines, revised Edition, Vol.~27 of Applied
  Mathematical Sciences, Springer-Verlag, New York, 2001.

\bibitem{Sederberg_Cardon}
T.~Sederberg, D.~Cardon, G.~Finnigan, N.~North, J.~Zheng, T.~Lyche, T-spline
  simplication and local refinement, ACM Trans. Graph. 23~(3) (2004) 276--283.

\bibitem{Bazilervs_Calo_Cottrell_Evans}
Y.~Bazilevs, V.~Calo, J.~A. Cottrell, J.~A. Evans, T.~J.~R. Hughes, S.~Lipton,
  M.~Scott, T.~Sederberg, Isogeometric analysis using {T}-splines, Comput.
  Methods Appl. Mech. Engrg. 199~(5-8) (2010) 229 -- 263.

\bibitem{Buffa_Cho_Sangalli}
A.~Buffa, D.~Cho, G.~Sangalli, Linear independence of the {T}-spline blending
  functions associated with some particular {T}-meshes, Comput. Methods Appl.
  Mech. Engrg. 199~(23--24) (2010) 1437--1445.

\bibitem{Scott_Li_Sederberg_Hughes}
M.~Scott, X.~Li, T.~Sederberg, T.~J.~R. Hughes, Local refinement of
  analysis-suitable {T}-splines, Comput. Methods Appl. Mech. Engrg. 213 - 216
  (2012) 206 -- 222.

\bibitem{Wang2011477}
W.~Wang, Y.~Zhang, M.~Scott, T.~J.~R. Hughes, Converting an unstructured
  quadrilateral mesh to a standard {T}-spline surface, Comput. Mech. 48~(4)
  (2011) 477--498.

\bibitem{Beirao_Buffa_Cho_Sangalli}
L.~Beir{\~a}o~da Veiga, A.~Buffa, D.~Cho, G.~Sangalli, Iso{G}eometric analysis
  using {T}-splines on two-patch geometries, Comput. Methods Appl. Mech. Engrg.
  200~(21-22) (2011) 1787--1803.

\bibitem{LZSHS12}
X.~Li, J.~Zheng, T.~Sederberg, T.~Hughes, M.~Scott, On linear independence of
  {T}-spline blending functions, Comput. Aided Geom. Design 29~(1) (2012) 63 --
  76.

\bibitem{BBCS12}
L.~Beir{\~a}o~da Veiga, A.~Buffa, D.~Cho, G.~Sangalli, Analysis-{S}uitable
  {T}-splines are {D}ual-{C}ompatible, Comput. Methods Appl. Mech. Engrg.
  (2012) (to appear).

\bibitem{Scott}
M.~Scott, T-splines as a design-through-analysis technology, Ph.D. thesis, The
  University of Texas at Austin (2011).

\bibitem{AFW-2}
D.~N. Arnold, R.~S. Falk, R.~Winther, Finite element exterior calculus: from
  {H}odge theory to numerical stability, Bull. Amer. Math. Soc. (N.S.) 47~(2)
  (2010) 281--354.

\bibitem{Buffa_deFalco_Sangalli}
A.~Buffa, C.~de~Falco, G.~Sangalli, Isogeometric {A}nalysis: stable elements
  for the 2{D} {S}tokes equation, Internat. J. Numer. Methods Fluids 65~(11-12)
  (2011) 1407--1422.

\bibitem{EvHu12-4}
J.~A. Evans, T.~J.~R. Hughes, {D}iscrete spectrum analyses for various mixed
  discretizations of the {S}tokes eigenproblem, Computational Mechanics (2012)
  (accepted for publication).

\bibitem{EvHu12}
J.~A. Evans, T.~J.~R. Hughes, Isogeometric divergence-conforming {B}-splines
  for the {D}arcy-{S}tokes-{B}rinkman equations., Math. Models Methods Appl.
  Sci. (2012) (accepted for publication)\href
  {http://dx.doi.org/10.1142/S0218202512500583}
  {\path{doi:10.1142/S0218202512500583}}.

\bibitem{EvHu12-2}
J.~A. Evans, T.~J.~R. Hughes, Isogeometric divergence-conforming {B}-splines
  for the {S}teady {N}avier-{S}tokes {E}quations, Tech. Rep. 12-15, ICES, UT
  Austin (2012).

\bibitem{EvHu12-3}
J.~A. Evans, T.~J.~R. Hughes, Isogeometric divergence-conforming {B}-splines
  for the {U}nsteady {N}avier-{S}tokes {E}quations, Tech. Rep. 12-16, ICES, UT
  Austin (2012).

\bibitem{Monk}
P.~Monk, Finite {E}lement {M}ethods for {M}axwell's {E}quations, Oxford
  University Press, Oxford, 2003.

\bibitem{BBSV12}
L.~Beir\~ao~da Veiga, A.~Buffa, G.~Sangalli, R.~V\'azquez, Analysis-suitable
  {T}-splines of arbitrary degree: definition and properties, Tech. rep.,
  IMATI-CNR (2012).

\bibitem{NED80}
J.-C. N{\'e}d{\'e}lec, Mixed finite elements in {$R^3$}, Numer. Math. 35 (1980)
  315--341.

\bibitem{BuCh07}
A.~Buffa, S.~H. Christiansen, A dual finite element complex on the barycentric
  refinement, Math. Comp. 76~(260) (2007) 1743--1769 (electronic).

\bibitem{GeWiCl04}
H.~De~Gersem, M.~Wilke, M.~Clemens, T.~Weiland, Efficient modelling techniques
  for complicated boundary conditions applied to structured grids, COMPEL
  23~(4) (2004) 904--912.

\bibitem{ClThWe99}
M.~Clemens, P.~Thoma, T.~Weiland, U.~van Rienen, Computational
  electromagnetic-field calculation with the finite-integration method, Surveys
  Math. Indust. 8~(3-4) (1999) 213--232.

\bibitem{MR1872728}
R.~Hiptmair, Discrete {H}odge operators, Numer. Math. 90~(2) (2001) 265--289.

\bibitem{MR2143847}
A.~Bossavit, Discretization of electromagnetic problems: the ``generalized
  finite differences'' approach, in: Handbook of numerical analysis. {V}ol.
  {XIII}, Handb. Numer. Anal., XIII, North-Holland, Amsterdam, 2005, pp.
  105--197.

\bibitem{Li_Scott}
X.~Li, M.~Scott, On the nesting behavior of {T}-splines, Tech. Rep. 11-13, ICES
  (2011).

\bibitem{Scott2011126}
M.~Scott, M.~Borden, C.~Verhoosel, T.~Sederberg, T.~J.~R. Hughes, Isogeometric
  finite element data structures based on {B\'e}zier extraction of {T}-splines,
  Internat. J. Numer. Methods Engrg. 88~(2) (2011) 126--156.

\bibitem{bressan-in-preparation}
A.~Bressan, T-splines characterization, in preparation.

\bibitem{Mourr10}
B.~Mourrain, On the dimension of spline spaces on planar {T}-meshes, Tech.
  rep., INRIA, http://hal.inria.fr/inria-00533187/en (2012).

\bibitem{LiCh11}
X.~Li, F.~Chen, On the instability in the dimension of splines spaces over
  {T}-meshes, Comput. Aided Geom. Design 28~(7) (2011) 420--426.

\bibitem{VB10}
R.~V\'{a}zquez, A.~Buffa, Isogeometric analysis for electromagnetic problems,
  Magnetics, IEEE Transactions on 46~(8) (2010) 3305 --3308.

\bibitem{dFRV11}
C.~de~Falco, A.~Reali, R.~V\'azquez, Geo{PDE}s: a research tool for
  {I}sogeometric {A}nalysis of {PDE}s, Adv. Eng. Softw. 42~(12) (2011)
  1020--1034.

\bibitem{Boffi07}
D.~Boffi, Approximation of eigenvalues in mixed form, discrete compactness
  property, and application to {$hp$} mixed finite elements, Comput. Methods
  Appl. Mech. Engrg. 196~(37-40) (2007) 3672--3681.

\bibitem{CD2000}
M.~Costabel, M.~Dauge, Singularities of electromagnetic fields in polyhedral
  domains, Arch. Ration. Mech. Anal. 151~(3) (2000) 221--276.

\bibitem{MR1860445}
S.~Nicaise, Edge elements on anisotropic meshes and approximation of the
  {M}axwell equations, SIAM J. Numer. Anal. 39~(3) (2001) 784--816
  (electronic).

\bibitem{BuCoDa05}
A.~Buffa, M.~Costabel, M.~Dauge, Algebraic convergence for anisotropic edge
  elements in polyhedral domains, Numer. Math. 101~(1) (2005) 29--65.

\bibitem{Apel}
T.~Apel, Anisotropic finite elements: local estimates and applications,
  Advances in Numerical Mathematics, B. G. Teubner, Stuttgart, 1999.

\bibitem{Jin}
J.~Jin, The finite element method in electromagnetics, 2nd Edition,
  Wiley-Interscience [John Wiley \& Sons], New York, 2002.

\end{thebibliography}
\def\cprime{$'$} \def\cprime{$'$}

\end{document}